\documentclass[11pt,a4paper]{amsart}

\usepackage[margin=3cm,top=2.5cm,bottom=2cm]{geometry}

\usepackage{amsmath, amssymb, amsthm, mathrsfs}
\usepackage{graphicx,enumitem,microtype,indentfirst} 
\usepackage[dvipsnames]{xcolor}

\usepackage[colorlinks=true,linkcolor=red!70!black,urlcolor=MidnightBlue,citecolor=MidnightBlue]{hyperref}

\usepackage[utf8]{inputenc} 
\usepackage[T1]{fontenc}
\usepackage{lmodern}

\newcommand{\R}{\mathbf R}
\newcommand{\T}{\mathbf T}

\newcommand{\Z}{\mathbf Z}

\renewcommand{\P}{\mathbf P}
\renewcommand{\S}{\mathbf S}

\newcommand{\FF}{\mathcal F}

\newcommand{\la}{\langle}
\newcommand{\ra}{\rangle}

\renewcommand{\d}{\mathrm{d}}

\newcommand{\eps}{\varepsilon}

\newcommand{\Nt}{|\hskip-0.04cm|\hskip-0.04cm|}

\DeclareMathOperator{\re}{Re}
\DeclareMathOperator{\Ker}{Ker}
\DeclareMathOperator{\Div}{div}

\numberwithin{equation}{section}

\setlist[enumerate]{wide,labelindent=0cm,label=\textnormal{(\arabic*)},itemsep=5pt,topsep=4pt}

\theoremstyle{plain}
\newtheorem{theo}{Theorem}[section]
\newtheorem{prop}[theo]{Proposition}
\newtheorem{lem}[theo]{Lemma}

\theoremstyle{remark}
\newtheorem{rem}{Remark}[section]
\theoremstyle{definition}

\title{Hydrodynamic limit for the non-cutoff Boltzmann equation}
\author[C. Cao]{Chuqi Cao}
\address[Chuqi Cao]{Department of Applied Mathematics, The Hong Kong Polytechnic University, Hong Kong, China} \email{chuqicao@gmail.com}
\author[K. Carrapatoso]{Kleber Carrapatoso}
\address[Kleber Carrapatoso]{Centre de Math\'ematiques Laurent Schwartz, \'Ecole polytechnique, Institut Polytechnique de Paris, 91128 Palaiseau cedex, France}
\email{kleber.carrapatoso@polytechnique.edu}


\thanks{The authors thank Tong Yang for fruitful discussions concerning the results of \cite{YY2}. C.C.\ research is partially supported by the Research Centre for Nonlinear Analysis, Hong Kong Polytechnic University. K.C.\ has been partially supported by the Project EFI ANR-17-CE40-0030 of the French National Research Agency.}

\begin{document}

\begin{abstract}
This work deals with the non-cutoff Boltzmann equation for all type of potentials, in both the torus $\T^3$ and in the whole space $\R^3$, under the incompressible Navier-Stokes scaling. We first establish the well-posedness and decay of global mild solutions to this rescaled Boltzmann equation in a perturbative framework, that is for solutions close to the Maxwellian, obtaining in particular integrated-in-time regularization estimates. 
We then combine these estimates with spectral-type estimates in order to obtain the strong convergence of solutions to the non-cutoff Boltzmann equation towards the incompressible Navier-Stokes-Fourier system.
\end{abstract}

\maketitle

\tableofcontents

\section{Introduction}

Since Hilbert~\cite{H2}, an important problem in kinetic theory concerns the rigorous link between different scales of description of a gas. More precisely, one is interested in passing rigorously from a mesoscopic description of a gas, modeled by the kinetic Boltzmann equation, towards a macroscopic description, modeled by Euler or Navier-Stokes fluid equations, through a suitable scaling limit. We are interested in this paper on the convergence of solutions to the Boltzmann equation towards the incompressible Navier-Stokes equation, and we refer to the book~\cite{SR} and the references therein to a detailed description of this type of problem as well as to different scalings and fluid limit equations.

We introduce in Section~\ref{sec:intro_boltzmann} below the (rescaled) Boltzmann equation, and then in Section~\ref{sec:intro_nsf} we describe the  incompressible Navier-Stokes-Fourier system, which is the expected limit. We finally present our main results in Section~\ref{sec:results}.

\subsection{The Boltzmann equation}\label{sec:intro_boltzmann}

The Boltzmann equation is a fundamental model in kinetic theory that describes the evolution of a rarefied gas out of equilibrium by taking into account binary collisions between particles. More precisely, it describes the evolution in time of the unknown $F(t,x,v) \ge 0$ which represents the density of particles that at time $t \ge 0$ and position $x \in \Omega_x = \T^3$ or $\Omega_x = \R^3$ move with velocity $v \in \R^3$. It was introduced by Maxwell~\cite{Maxwell} and Boltzmann~\cite{Boltzmann} and reads
\begin{equation}\label{eq:Boltzmann0}
\partial F +  v \cdot \nabla_x F = \frac{1}{\eps} Q(F,F),
\end{equation}
which is complemented with an initial data $F_{|t=0} = F_0$ and where $\eps \in (0,1]$ is the Knudsen number, which corresponds to the ratio between the mean-free path and the macroscopic length scale.

The Boltzmann collision operator $Q$ is a bilinear operator acting only on the velocity variable $v \in \R^3$, which means that collisions are local in space, and it is given by
\begin{equation}\label{eq:operator_Boltzmann}
Q(G,F)(v)=\int_{\R^3}\int_{\S^2}B(v-v_*,\sigma)(G'_*F'-G_*F)\, \d \sigma \, \d v_*, 
\end{equation}
where here and below we use the standard short-hand notation $F=F(v)$, $G_*=G(v_*)$, $F'=F(v')$, and $G'_*=G(v'_*)$, and where the pre- and post-collision velocities $(v',v'_*)$ and $(v,v_*)$ are related through
\begin{equation}\label{eq:def-v'-v'*}
v' = \frac{v+v_*}{2} + \frac{|v-v_*|}{2}\sigma 
\quad\text{and}\quad
v'_* = \frac{v+v_*}{2} - \frac{|v-v_*|}{2}\sigma ,
\end{equation}
where $\sigma \in \S^2$.
The above formula is one possible parametrization of the set of solutions of  an elastic collision with the physical laws of conservation (momentum and energy) 
$$
v+v_*=v'+v'_* \quad\text{and}\quad  |v|^2+|v_*|^2=|v'|^2+|v'_*|^2.
$$
The function $B(v-v_*,\sigma)$ appearing in \eqref{eq:operator_Boltzmann}, called the collision kernel, is supposed to be nonnegative and to depend only on the relative velocity $|v-v_*|$ and the deviation angle $\theta$ through $\cos\theta:= \frac{v-v_*}{|v-v_*|}\cdot\sigma$. As it is customary, we may suppose without loss of generality that $\theta \in [0,\pi/2]$,  for otherwise $B$ can be replaced by its symmetrized form.

In this paper we shall consider the case of \emph{non-cutoff potentials}  that we describe now. 
The collision kernel $B$ takes the form
$$
B(v-v_*,\sigma)=|v-v_*|^\gamma b(\cos \theta),
$$
for some nonnegative function $b$, called the angular kernel, and some parameter $\gamma \in (-3,1]$. We assume that the angular kernel $b$ is a locally smooth implicit function which is not locally integrable, more precisely that it satisfies
$$
\mathcal{K}\theta^{-1-2s}\leq \sin\theta \, b(\cos\theta)\leq \mathcal{K}^{-1}\theta^{-1-2s} \quad \mbox{with} \quad 0<s<1,
$$
for some constant $\mathcal{K}>0$. Moreover the parameters satisfy the condition 
\begin{equation}\label{eq:cond_gamma_s}
\max\left\{ -3 , - \frac{3}{2} - 2s \right\} <\gamma  \le 1, \quad 0<s<1, \quad  \gamma+2s>-1.
\end{equation}
We shall consider in this paper the full range of parameters $\gamma$ and $s$ satisfying \eqref{eq:cond_gamma_s}, and we classify them into two cases: When $\gamma + 2s \ge 0$ we speak of \emph{hard potentials}, and when $\gamma+2s < 0$ of \emph{soft potentials}.
We also mention that \emph{cutoff kernels} correspond to the case in which we remove the singularity of the angular kernel $b$ and assume that $b$ is integrable.

\begin{rem}
When particles interact via a repulsive inverse-power law potential $\phi(r) = r^{-(p-1)}$ with $p>2$, then it holds (see \cite{Maxwell,Cercignani}) that $\gamma = \frac {p-5} {p-1}$ and $s = \frac 1 {p-1}$. It is easy to check that $\gamma + 4s = 1$ which means the above assumption is satisfied for the full range of the inverse power law model. 
\end{rem}

Formally if $F$ is a solution to equation \eqref{eq:Boltzmann0} with the initial data $F_0$, then it enjoys the conservation of mass, momentum and the energy, that is,
\[ 
\frac {d}{dt}\int_{\Omega_x  \times{\R}^3 } F(t, x , v)\varphi(v) \, \d v \, \d x= 0, \quad \varphi(v)=1,v,|v|^2,
\]
which is a consequence of the collision invariants of the Boltzmann operator
\begin{equation}\label{eq:collision_invariants}
\int_{\R^3} Q(F,F)(v) \varphi(v) \, \d v = 0, \quad \varphi(v)=1,v,|v|^2.
\end{equation}
Moreover the Boltzmann H-theorem asserts on the one hand that the entropy
\[
H(F) = \int_{\Omega_x \times \R^3}  F \log F \, \d v \, \d x ,
\]
is non-increasing in time. Indeed, at least formally, since $(x-y) (\log x -\log y)$ is nonnegative, we have the following inequality for the entropy dissipation $D(f)$:
\begin{align*}
D(f) 
&= -\frac{\d}{\d t} H(F) = - \int_{\Omega_x \times \R^3} Q(F,F) \, \d v \, \d x \\
&= \frac 1 4\int_{\Omega_x \times \R^3 \times \R^3 \times \S^2} B(v-v_*, \sigma) (F' F'_* - F_* F)  \log \left( \frac{F' F_*'}{F F_*} \right) \, \d \sigma \, \d v_* \, \d v \,  \d x \ge 0.
\end{align*}
On the other hand, the second part of the H-theorem asserts that local equilibria of the Boltzmann equation are local Maxwellian distributions in velocity, more precisely that
$$
D(F) = 0 \quad \Leftrightarrow \quad Q(F,F)=0 \quad \Leftrightarrow \quad F(t,x,v) = \frac{\rho(t,x)}{(2 \pi \theta(t,x))^{3/2}} \exp\left( - \frac{|v-u(t,x)|^2}{2 \theta(t,x)} \right),
$$
with $\rho (t,x) >0$, $u(t,x) \in \R^3$ and $\theta(t,x) >0$. In what follows, we denote by $\mu=\mu(v)$ the global Maxwellian
$$
\mu = (2\pi)^{-3/2} e^{-|v|^2/2}.
$$

Observing that the effect of collisions are enhanced when taking small parameter $\eps \in (0,1]$, one can expect from the above H-Theorem that, at least formally, in the limit $\eps \to 0$ the solution $F$ approaches a local Maxwellian equilibrium. One therefore considers, see for instance in~\cite{BGL1}, a rescaling of the solution $F$ of \eqref{eq:Boltzmann0} in which an additional dilatation of the macroscopic time scale has been performed in order to be able to reach the Navier-Stokes equation in the limit. This procedure gives us the following rescaled Boltzmann equation for the new unknown $F^\eps = F^\eps(t,x,v)$:
\begin{equation}\label{eq:Boltzmann}
\partial F^\eps +  \frac{1}{\eps} v \cdot \nabla_x F^\eps = \frac{1}{\eps^2} Q(F^\eps,F^\eps),
\end{equation}
with initial data $F^\eps_{|t=0} = F^\eps_0$. 

In the torus case $\Omega_x = \T^3$ (normalized as $|\T^3|=1$), we shall always assume, thanks to the conservation laws, that the initial datum $F^\eps_0$ satisfies the normalization
\begin{equation}\label{eq:normalization0}
\int_{\T^3}\int_{\R^3} F^\eps_0(x,v) [1,v,|v|^2] \, \d v \, \d x = [1,0,3],
\end{equation}
that is, the initial data $F^\eps_0$ has the same mass, momentum and energy as $\mu$, and the Maxwellian $\mu$ is the unique global equilibrium to \eqref{eq:Boltzmann}.

In order to relate the above rescaled Boltzmann equation \eqref{eq:Boltzmann} to the expected incompressible Navier-Stokes-Fourier system (described below in \eqref{Navier-Stokes-Fourier system0}) in the limit $\eps \to 0$, we are going to work with the perturbation $f^\eps$ defined by
\begin{equation}\label{eq:def:feps}
F^\eps = \mu + \eps \sqrt{\mu} f^\eps ,
\end{equation}
which then satisfies the equation
\begin{equation}\label{eq:feps_intro}
\partial_t f^\eps + \frac{1}{\eps} v \cdot \nabla_x f^\eps = \frac{1}{\eps^2}L f^\eps + \frac{1}{\eps} \Gamma(f^\eps , f^\eps), \quad
\end{equation}
with initial data $f^\eps_0 = \frac{F^\eps_0-\mu}{\eps \sqrt{\mu}}$, and where we denote 
\begin{equation}\label{eq:def:Gamma}
\Gamma (f,g) = \mu^{-1/2} Q (\sqrt{\mu} f , \sqrt{\mu} g),
\end{equation}
and
\begin{equation}\label{eq:def:L}
L f = \Gamma(\sqrt{\mu} , f) + \Gamma(f,\sqrt{\mu}).
\end{equation}
We already remark that thanks to the collision invariants in \eqref{eq:collision_invariants}, we have
\begin{equation}\label{eq:Gamma_invariants}
\int_{\R^3} \Gamma(f,f) [1,v,|v|^2] \sqrt{\mu} \, \d v = 0.
\end{equation}

In the case of the torus $\Omega_x = \T^3$, we observe from \eqref{eq:normalization0} that $f^\eps_0$ satisfies
\begin{equation}\label{eq:normalization1}
\int_{\T^3}\int_{\R^3} f^\eps_0(x,v) [1,v,|v|^2]\sqrt{\mu}(v) \, \d v \, \d x = 0,
\end{equation}
and from the conservation laws recalled above that, for all $t \ge 0$,
\begin{equation}\label{eq:normalization2}
\int_{\T^3}\int_{\R^3} f^\eps(t,x,v) [1,v,|v|^2]\sqrt{\mu}(v) \, \d v \, \d x = 0.
\end{equation}

\subsection{The Navier-Stokes-Fourier system}\label{sec:intro_nsf}

We recall the Navier-Stokes-Fourier system associated with the Boussinesq equation which writes
\begin{equation}
\label{Navier-Stokes-Fourier system0}
\left\{ \begin{aligned}& \partial_t u + u \cdot \nabla_x u + \nabla_x p- \nu_1 \Delta_x u = 0,
\\
& \partial_t \theta + u \cdot \nabla_x \theta - \nu_2 \Delta_x \theta = 0,
\\
& \Div_x u = 0,
\\
& \nabla_x(\rho +\theta) =0,
\end{aligned}
\right.
\end{equation}
with positive viscosity coefficients $\nu_1, \nu_2>0$.
In this system, the temperature $\theta = \theta(t,x) : \R_+ \times \Omega_x \to \R$ of the fluid, the density $\rho = \rho(t,x) : \R_+ \times \Omega_x \to \R$ of the fluid, and the pressure $p = p(t,x): \R_+ \times \Omega_x \to \R$ of the fluid are scalar unknowns, whereas the velocity $u=u(t,x) : \R_+ \times \Omega_x \to \R^3$ of the fluid is an unknown vector field. 
The pressure $p$ can actually be eliminated from the equation by applying to the first equation in~\eqref{Navier-Stokes-Fourier system0} the Leray projector $\mathbb{P}$ onto the space of divergence-free vector fields. In other words, for $u$ we have
$$
\partial_t u  -\nu_1 \Delta_x u  =Q_{\mathrm{NS}}  (u, u),
$$
where the bilinear operator $Q_{\mathrm{NS}} $ is defined by
\begin{equation}\label{eq:def:QNS}
Q_{\mathrm{NS}}   (v, u)    = -  \frac 1 2\mathbb{P} (\Div (v \otimes u ) + \Div (u \otimes v ) ), \quad \Div (v \otimes u )^j :=\sum_{k=1}^3 \partial_k(v^ju^k) = \Div (v^j u), 
\end{equation}
and the Leray projector $\mathbb P$ on divergence-free vector fields is as follows, for $1 \le j \le 3$ and all $\xi \in \Omega'_\xi$, 
\[
\mathcal{F}_x (\mathbb{P} f)^j (\xi)  =   \mathcal{F}_x (f^j)(\xi) -\frac 1 {|\xi|^2}    \sum_{k=1}^3 \xi_j \xi_k \mathcal{F}_x (f^k) (\xi) =\sum_{k=1}^3 (\delta_{j, k} -1) \frac {\xi_j \xi_k} {|\xi|^2} \mathcal{F}_x (f^k) (\xi) ,
\]
where $\FF_x$ denotes the Fourier transform in the spatial variable $x \in \Omega_x$, see for instance \cite[Section 5.1]{BCD}.

We therefore consider the system
\begin{equation}
\label{Navier-Stokes-Fourier system}
\left\{ \begin{aligned}& 
\partial_t u  -\nu_1 \Delta_x u  =Q_{\mathrm{NS}}  (u, u)   ,
\\
& \partial_t \theta + u \cdot \nabla_x \theta - \nu_2 \Delta_x \theta = 0
\\
& \Div_x u = 0,
\\
& \nabla_x(\rho +\theta) =0,
\end{aligned}
\right.
\end{equation}
for the unknown $(\rho,u,\theta)$, which is complemented with a initial data $(\rho_0 , u_0 , \theta_0)$ that we shall always suppose to verify
\begin{equation}\label{eq:rho0_u0_theta0}
\Div_x u_0 = 0 , \quad \nabla_x(\rho_0 +\theta_0) =0.
\end{equation}
In the case of the torus $\Omega_x = \T^3$, we suppose moreover that the initial data is mean-free, namely
$$
\int_{\T^3} \rho_0(x) \, \d x = 
\int_{\T^3} u_0(x) \, \d x = 
\int_{\T^3} \theta_0(x) \, \d x = 0,
$$
which then implies that the associated solution $(\rho,u,\theta)$ also is mean-free for all $t \ge 0$
\begin{equation}\label{eq:meanfree}
\int_{\T^3} \rho(t,x) \, \d x = 
\int_{\T^3} u(t,x) \, \d x = 
\int_{\T^3} \theta(t,x) \, \d x = 0.
\end{equation}

\section{Main results}\label{sec:results}

Our main result establishes a strong convergence in the hydrodynamic limit from solutions to the rescaled Boltzmann equation~\eqref{eq:feps_intro} towards solution to the incompressible Nabier-Stokes-Fourier equation~\eqref{Navier-Stokes-Fourier system} (see Theorem~\ref{theo:hydro_limit}). In order to do so, we first need to provide a well-posedness theory for the 
Boltzmann equation~\eqref{eq:feps_intro} (see Theorem~\ref{theo:boltzmann}) as well as a well-posedness theory for the incompressible Nabier-Stokes-Fourier equation~\eqref{Navier-Stokes-Fourier system} (see Theorem~\ref{theo:NSF}), in such a way that the functional frameworks are compatible for being able to compare solutions and then to tackle the hydrodynamic limit problem.

\medskip

Before stating our results we introduce some notation. 
Given a function $f=f(x, v)$ we denote $\widehat f (\xi, v) = \FF_x ( f(\cdot, v)) (\xi)$ the Fourier transform in the space variable, for $\xi \in \Omega'_\xi= \Z^3$ (if $\Omega_x = \T^3$) or $ \Omega'_\xi = \R^3$ (if $\Omega_x = \R^3$), more precisely 
$$
\widehat f(\xi, v) = \frac{1}{(2 \pi)^{3/2}}\int_{\R^3} e^{-\mathrm{i} x \cdot \xi} f(x, v) \, \d x.
$$
In particular, we observe that if $f$ satisfies \eqref{eq:feps_intro}, then for all $\xi \in \Omega'_\xi$, its Fourier transform in space $\widehat f^\eps (\xi)$ satisfies the equation
\begin{equation}\label{eq:feps_fourier_intro}
\partial_t \widehat f^\eps (\xi)  = \frac{1}{\eps^2} ( L - \mathrm{i} \eps v \cdot \xi ) \widehat f^\eps(\xi) + \frac{1}{\eps} \widehat{\Gamma}(f^\eps , f^\eps)(\xi),
\end{equation}
where
$$
\widehat{\Gamma}( f , g)(\xi) = \sum_{\eta \in \Z^3} \Gamma \left( \widehat f(\xi-\eta) , \widehat g (\eta) \right) \quad\text{if}\quad \Omega_x=\T^3,
$$
or
$$
\widehat{\Gamma}( f , g)(\xi) = \int_{\R^3} \Gamma \left( \widehat f(\xi-\eta) , \widehat g (\eta) \right) \d \eta \quad\text{if}\quad \Omega_x=\R^3.
$$

For functions $f=f(x, v)$ we write the \emph{micro-macro decomposition}
\begin{equation}\label{eq:f_micro_macro}
f = \P^\perp f + \P f, \quad \P^\perp = I - \P,
\end{equation}
where $\P$ is the orthogonal projection onto $\Ker (L) = \{\sqrt{\mu} , v \sqrt{\mu}, |v|^2 \sqrt{\mu}\}$ given by
\begin{equation}\label{eq:def:Pf}
\P f (x, v) = \left\{ \rho[f](x)  + u[f](x) \cdot v  + \theta[f](x) \frac{(|v|^2-3)}{2} \right\} \sqrt{\mu}(v),
\end{equation}
where
\begin{equation}\label{eq:def:rho_u_theta}
\begin{aligned}
\rho[f](x) &= \int_{\R^3} f(x, v) \sqrt{\mu}(v) \, \d v ,\\
u[f](x) &= \int_{\R^3}  f(x, v) v \sqrt{\mu}(v) \, \d v ,\\
\theta[f](x) &= \int_{\R^3} f(x, v) \frac{(|v|^2-3)}{3} \sqrt{\mu}(v) \, \d v .
\end{aligned}
\end{equation}
The function $\P^\perp f$ is called the \emph{microscopic part} of $f$, whereas $\P f$ is the \emph{macroscopic part} of $f$.

We now introduce the functional spaces we work with.
For every $\ell \ge 0$ we denote by $L^2_v(\la v \ra^\ell)$ the weighted Lebesgue space associated to the inner product
$$
\la f ,g \ra_{L^2_v(\la v \ra^\ell)} := \la \la v \ra^{\ell} f , \la v \ra^{\ell} g \ra_{L^2_v} = \int_{\R^3} f g \la v \ra^{2 \ell} \, \d v,
$$
and the norm
$$
\| f \|_{L^2_v(\la v \ra^\ell)} := \| \la v \ra^\ell f \|_{L^2_v},
$$
where $L^2_v = L^2 (\R^3_v)$ is the standard Lebesgue space.
We denote by $H^{s,*}_v$ the Sobolev-type space associated to the dissipation of the linearized operator $L$ defined in \cite{AMUXY2} (see also \cite{GS} for the definition of a different but equivalent anisotropic norm), more precisely we denote
\begin{equation}\label{eq:def:Hs*v0}
\| f \|_{H^{s,*}_v (\la v \ra^\ell)} := \| \la v \ra^\ell f \|_{H^{s,*}_v } ,
\end{equation}
where
\begin{equation}\label{eq:def:Hs*v}
\begin{aligned}
\| f \|_{H^{s,*}_v }^2 
&:=  \int_{\R^3}\int_{\R^3}\int_{\S^2} b(\cos \theta) |v-v_*|^\gamma \mu(v_*)[f(v') - f(v)]^2 \, \d \sigma \, \d v_* \, \d v \\
&\quad 
+  \int_{\R^3}\int_{\R^3}\int_{\S^2} b(\cos \theta) |v-v_*|^\gamma  f(v_*)^2 [\sqrt{\mu}(v') - \sqrt{\mu}(v) ]^2 \, \d \sigma \, \d v_* \, \d v,
\end{aligned}
\end{equation}
which verifies, see \cite{AMUXY2,GS}, 
$$
\| \la v \ra^{\gamma/2+s} f \|_{L^2_v(\la v \ra^\ell)}
+ \| \la v \ra^{\gamma/2} f \|_{H^{s}_v(\la v \ra^\ell)} 
\lesssim \| f \|_{H^{s,*}_v(\la v \ra^\ell)} 
\lesssim \| \la v \ra^{\gamma/2+s} f \|_{H^{s}_v(\la v \ra^\ell)}.
$$
We also define the space $(H^{s,*}_v)'$ as the dual space of $H^{s,*}_v$ endowed with the norm
\begin{equation}\label{eq:def:Hs*v'}
\| f \|_{(H^{s,*}_v)'} := \sup_{\| \phi \|_{H^{s,*}_v}\le 1 } \la f , \phi \ra.
\end{equation}
For functions depending on space and velocity variables, we shall also use a variant of the quantity $ \| \cdot \|_{H^{s,*}_v (\la v \ra^\ell)}$ defined above in \eqref{eq:def:Hs*v} that also depends on the spatial variable. More precisely, for $f=f(x,v)$ we define the quantity
\begin{equation}\label{eq:def:Hs**v}
\| f \|_{H^{s,**}_v (\la v \ra^\ell)}^2 
:= \| \P^\perp f \|_{H^{s,*}_v (\la v \ra^\ell)}^2 + \| a(D_x) \P f \|_{L^2_v}^2,
\end{equation}
where $a(D_x)$ is the Fourier multiplier $a(\xi) = \frac{|\xi|}{\la \xi \ra}$, which gives, in Fourier variable,
\begin{equation}\label{eq:def:Hs**v_fourier}
\| \widehat f(\xi) \|_{H^{s,**}_v (\la v \ra^\ell)}^2 
= \| \P^\perp \widehat f (\xi) \|_{H^{s,*}_v (\la v \ra^\ell)}^2 + \frac{|\xi|^2}{\la \xi \ra^2}\| \P \widehat f (\xi) \|_{L^2_v}^2.
\end{equation}
Finally, given a functional space $X$ in the variables $(t,\xi, v)$, we shall denote by $\FF^{-1}_x ( X)$ the Fourier-based space defined as
$$
\FF^{-1}_x ( X) := \left\{ f = f(t, x, v)  \mid \widehat f \in X \right\} .
$$
Hereafter, in order to deal with the torus case $\Omega_x = \T^3$ and the whole space case $\Omega_x = \R^3$ simultaneously, we denote $L^p_\xi = \ell^p(\Z^3)$ in the torus case and $L^p_\xi = L^p(\R^3)$ in the whole space case, moreover we abuse notation and write
$$
\int_{\Omega'_\xi} \phi(\xi) \, \d \xi := 
\left\{
\begin{aligned}
& \sum_{\xi \in \Z^3} \phi(\xi) \quad &\text{if}\quad \Omega'_\xi = \Z^3 , \\
& \int_{\R^3} \phi(\xi) \, \d \xi \quad &\text{if}\quad \Omega'_\xi = \R^3 .
\end{aligned}
\right.
$$
In particular, we shall consider below functional spaces of the type $\FF^{-1}_x(L^p_\xi L^\infty_t L^2_v (\la v \ra^\ell))$ and $\FF^{-1}_x( L^p_\xi L^2_t H^{s,*}_v (\la v \ra^\ell) )$ (or $\FF^{-1}_x( L^p_\xi L^2_t H^{s,**}_v (\la v \ra^\ell) )$) and the respective norms, for $f=f(t,x,v)$,
$$
\| \widehat f \|_{L^p_\xi L^\infty_t L^2_v(\la v \ra^\ell)} := \left(\int_{\Omega'_\xi} \sup_{t \ge 0} \| \widehat f(t,\xi, \cdot) \|_{L^2_v(\la v \ra^\ell)}^p  \, \d \xi\right)^{1/p} \quad \text{for} \quad p\in[1,+\infty),
$$
and 
$$
\| \widehat f \|_{L^p_\xi L^2_t H^{s,*}_v(\la v \ra^\ell)} := \left(\int_{\Omega'_\xi} \left\{ \int_{0}^\infty \| \widehat f(t,\xi, \cdot) \|_{H^{s,*}_v(\la v \ra^\ell)}^2 \, \d t \right\}^{p/2}  \, \d \xi \right)^{1/p} \quad \text{for} \quad p\in[1,+\infty),
$$
with the usual modification for $p=+\infty$.

\subsection{Well-posedness for the rescaled Boltzmann equation}
Our first result concerns the global well-posedness, regularization and decay for equation \eqref{eq:feps_intro} for small initial data.

\begin{theo}[Global well-posedness and decay for the Boltzmann equation]\label{theo:boltzmann}
Let $\ell=0$ in the hard potentials case $\gamma+2s \ge 0$, and $\ell \ge 0$ in the soft potentials case $\gamma+2s<0$.
There is $\eta_0>0$ small enough such that for all $\eps \in (0,1]$ the following holds:

\begin{enumerate}

\item Torus case $\Omega_x = \T^3$: For any initial data $f^\eps_0 \in \FF^{-1}_x(L^1_\xi L^2_v (\la v \ra^\ell))$ satisfying \eqref{eq:normalization2} and $ \| \widehat f^\eps_0 \|_{L^1_\xi L^2_v(\la v \ra^\ell)} \le \eta_0$, there exists a unique global mild solution $f^\eps \in \FF^{-1}_x(L^1_\xi L^\infty_t L^2_v(\la v \ra^\ell) \cap L^1_\xi L^2_t H^{s,*}_v (\la v \ra^\ell))$ to \eqref{eq:feps_intro} satisfying \eqref{eq:normalization2} and the energy estimate
\begin{equation}\label{eq:theo1:existence}
\| \widehat f^\eps \|_{L^1_\xi L^\infty_t L^2_v(\la v \ra^\ell)} + \frac{1}{\eps} \|  \P^\perp \widehat f^\eps \|_{L^1_\xi L^2_t H^{s,*}_v(\la v \ra^\ell)} +  \|  \P \widehat f^\eps \|_{L^1_\xi L^2_t L^2_v} \lesssim \| \widehat f^\eps_0 \|_{L^1_\xi L^2_v(\la v \ra^\ell)}.
\end{equation}

Moreover we have the following decay estimates: In the hard potentials case $\gamma+2s \ge 0$, there exists $\lambda >0$ such that \begin{equation}\label{eq:theo1:decay_hard}
\| \mathrm{e}_\lambda \widehat f^\eps \|_{L^1_\xi L^\infty_t L^2_v} + \frac{1}{\eps} \| \mathrm{e}_\lambda \P^\perp \widehat f^\eps \|_{L^1_\xi L^2_t H^{s,*}_v} +  \| \mathrm{e}_\lambda \P \widehat f^\eps \|_{L^1_\xi L^2_t L^2_v} \lesssim \| \widehat f^\eps_0 \|_{L^1_\xi L^2_v(\la v \ra^\ell)},
\end{equation}
where we denote $ \mathrm{e}_\lambda : t \mapsto e^{\lambda t}$. In the soft potentials case $\gamma + 2 s <0$, if $\ell>0$ then for any $0 < \omega < \frac{\ell}{|\gamma+2s|}$ there holds
\begin{equation}\label{eq:theo1:decay_soft}
\| \mathrm{p}_\omega \widehat f^\eps \|_{L^1_\xi L^\infty_t L^2_v} + \frac{1}{\eps} \| \mathrm{p}_\omega \P^\perp \widehat f^\eps \|_{L^1_\xi L^2_t H^{s,*}_v} +  \| \mathrm{p}_\omega \P \widehat f^\eps \|_{L^1_\xi L^2_t L^2_v} \lesssim \| \widehat f^\eps_0 \|_{L^1_\xi L^2_v (\la v \ra^\ell) },
\end{equation}
where we denote $\mathrm{p}_\omega : t \mapsto (1+t)^{\omega}$.

\item Whole space case $\Omega_x = \R^3$: Let $p \in (3/2,\infty]$. For any initial data $f^\eps_0 \in \FF^{-1}_x(L^1_\xi L^2_v(\la v \ra^\ell)  \cap L^p_\xi L^2_v(\la v \ra^\ell)  )$ satisfying $ \| \widehat f^\eps_0 \|_{L^1_\xi L^2_v(\la v \ra^\ell)} + \| \widehat f^\eps_0 \|_{L^p_\xi L^2_v(\la v \ra^\ell)} \le \eta_0$, there exists a unique global mild solution 
$
f^\eps \in \FF^{-1}_x(L^1_\xi L^\infty_t L^2_v(\la v \ra^\ell) \cap L^1_\xi L^2_t H^{s,**}_v(\la v \ra^\ell) ) \cap  \FF^{-1}_x(L^p_\xi L^\infty_t L^2_v(\la v \ra^\ell) \cap L^p_\xi L^2_t H^{s,**}_v(\la v \ra^\ell) )
$ 
to \eqref{eq:feps_intro} satisfying the energy estimate
\begin{equation}\label{eq:theo1bis:existence}
\begin{aligned}
\| \widehat f^\eps \|_{L^1_\xi L^\infty_t L^2_v(\la v \ra^\ell)} 
+ \frac{1}{\eps} \|  \P^\perp \widehat f^\eps \|_{L^1_\xi L^2_t H^{s,*}_v(\la v \ra^\ell)} 
+&  \left\| \frac{|\xi|}{\la \xi \ra} \P \widehat f^\eps \right\|_{L^1_\xi L^2_t L^2_v} \\ 
+\| \widehat f^\eps \|_{L^p_\xi L^\infty_t L^2_v(\la v \ra^\ell)} 
+ \frac{1}{\eps} \|  \P^\perp \widehat f^\eps \|_{L^p_\xi L^2_t H^{s,*}_v(\la v \ra^\ell)} 
&+  \left\| \frac{|\xi|}{\la \xi \ra} \P \widehat f ^\eps\right\|_{L^p_\xi L^2_t L^2_v} 
\lesssim \| \widehat f^\eps_0 \|_{L^1_\xi L^2_v(\la v \ra^\ell)} + \| \widehat f^\eps_0 \|_{L^p_\xi L^2_v(\la v \ra^\ell)}.
\end{aligned}
\end{equation}

Moreover we have the following decay estimates: In the hard potentials case $\gamma+2s \ge 0$, for any $0< \vartheta < \frac{3}{2}(1-\frac{1}{p})$ there holds
\begin{equation}\label{eq:theo1bis:decay_hard}
\begin{aligned}
\| \mathrm{p}_\vartheta \widehat f^\eps \|_{L^1_\xi L^\infty_t L^2_v} 
+ \frac{1}{\eps} \| \mathrm{p}_\vartheta \P^\perp \widehat f^\eps \|_{L^1_\xi L^2_t H^{s,*}_v} 
&+  \left\| \mathrm{p}_\vartheta \frac{|\xi|}{\la \xi \ra} \P \widehat f^\eps \right\|_{L^1_\xi L^2_t L^2_v} \\
&\lesssim \| \widehat f^\eps_0 \|_{L^1_\xi L^2_v} + \| \widehat f^\eps_0 \|_{L^p_\xi L^2_v}.
\end{aligned}
\end{equation}
where we denote $\mathrm{p}_\vartheta : t \mapsto (1+t)^{\vartheta}$. In the soft potentials case $\gamma + 2 s <0$, if $0 < \vartheta <  \frac{3}{2}(1-\frac{1}{p})$ and $\ell > \vartheta |\gamma+2s|$ there holds
\begin{equation}\label{eq:theo1bis:decay_soft}
\begin{aligned}
\| \mathrm{p}_\vartheta \widehat f^\eps \|_{L^1_\xi L^\infty_t L^2_v} 
+ \frac{1}{\eps} \| \mathrm{p}_\vartheta \P^\perp \widehat f^\eps \|_{L^1_\xi L^2_t H^{s,*}_v} 
&+  \left\| \mathrm{p}_\vartheta \frac{|\xi|}{\la \xi \ra} \P \widehat f^\eps \right\|_{L^1_\xi L^2_t L^2_v} \\
&\lesssim \| \widehat f^\eps_0 \|_{L^1_\xi L^2_v(\la v \ra^\ell)} + \| \widehat f^\eps_0 \|_{L^p_\xi L^2_v(\la v \ra^\ell)}.
\end{aligned}
\end{equation}

\end{enumerate}

\end{theo}

\begin{rem}\begin{enumerate}[label=(\roman*)]

\item We observe that in the soft-potentials case $\gamma+2s<0$ we can take $\ell=0$ for the well-posedness result. We only need a well-posedess theory with $\ell>0$ in order to obtain the decay estimates (\eqref{eq:theo1:decay_soft} and \eqref{eq:theo1bis:decay_soft}).

\item We observe that the functional spaces are different when working on the torus or the whole space. In the torus we have a solution in the space $\FF^{-1}_x (L^1_\xi L^2_t H^{s,*}_v (\la v \ra^\ell))$, whereas in the whole space the solution belongs to $\FF^{-1}_x (L^1_\xi L^2_t H^{s,**}_v (\la v \ra^\ell))$, with clearly $\| \cdot \|_{H^{s,**}_v (\la v \ra^\ell)} \le \| \cdot \|_{H^{s,*}_v (\la v \ra^\ell)}$. This comes from the hypocoercive-type estimate for the linearized operator (see Proposition~\ref{prop:hypocoercivity}). 

\item Another difference between the torus and the whole space appears when dealing with low frequencies $|\xi|<1$. When working on the torus the only low frequency is $\xi=0$, which is controlled thanks to the conservation laws. On the other hand, in the whole space, the gain estimate for the linearized operator in $\FF^{-1}_x (L^1_\xi L^2_t H^{s,**}_v (\la v \ra^\ell))$ is not enough to control low frequencies in the nonlinear estimates. This is why we also need to work in $\FF^{-1}_x (L^p_\xi)$-type spaces with $p\in(3/2,\infty]$.

\end{enumerate}
\end{rem}

\medskip

The Cauchy theory and the large time behavior for Boltzmann equation for $\eps=1$ have been extensively studied. Concerning the theory for large data, we only mention the global existence of renormalized solutions \cite{DPL} for the cutoff Boltzmann equation, and  the global existence of renormalized solutions with defect measure \cite{AV} for the non-cutoff Boltzmann equation.

We now give a very brief review for solutions to the Boltzmann equation in a perturbative framework, that is, for solutions near the Maxwellian. For the case of cutoff potentials, we refer to the works \cite{G1965,U1974,U1976,C1980_2,AU1982} as well as the more recent \cite{SW2021,D2022} for global solutions in spaces of the form $L^\infty_v H^N_x$ ; and to \cite{K1990,LYY2004,G2,SG,D2008} for solutions in $H^N_{x,v}$ or $H^N_x L^2_v$. On the other hand, for the non-cutoff Boltzmann equation, we refer to \cite{GS, GS2} in the torus case and to \cite{AMUXY2, AMUXY3, AMUXY4} in the whole space case, for the first global solutions in spaces of the form $H^N_{x, v}$ by working with anisotropic norms (see \eqref{eq:def:Hs*v}). The optimal time-decay was obtained in \cite{S} for the whole space, and recently \cite{D2022_2} constructed global solutions in the whole space.

All the above results concern solutions with Gaussian decay in velocity, that is, they hold in functional spaces of the type $H^N_{x, v}$ for the perturbation $f$ defined in \eqref{eq:def:feps}, which means that $F-\mu \in H^N_{x, v}(\mu^{-1/2})$. By developing decay estimates on the resolvents and semigroups of non-symmetric operators in Banach spaces, Gualdani-Mischler-Mouhot~\cite{GMM} proved nonlinear stability for the cutoff Boltzmann equation with hard potentials in $L^1_v L^\infty_x ( \la v \ra^k \mu^{1/2} ), k > 2$, that is, in spaces with polynomial decay in velocity ($f \in L^1_v L^\infty_x ( \la v \ra^k \mu^{1/2} )$ means $F-\mu \in L^1_v L^\infty_x ( \la v \ra^k)$). In the same framework, the case of non-cutoff hard potentials was treated in \cite{HTT, AMSY}, and
that of non-cutoff soft potentials in \cite{CHJ}.

The aforementioned results were obtained in Sobolev-type spaces, very recently Duan, Liu, Sakamoto and Strain~\cite{DLSS} obtained the well-posedness of the Boltzmann equation in Fourier-based spaces $L^1_{\xi}L^\infty_t L^2_x$ in the torus case, which was then extended to the whole space case by Duan, Sakamoto and Ueda in~\cite{DSU2022}, see also \cite{CG} for the whole space case in polynomial weighted spaces. We also refer to the works \cite{AMSY2,Cao} for recent results on the well-posedness for non-cutoff Boltzmann using De Giorgi arguments.

\smallskip

In our paper, we establish uniform in $\eps$ estimates for the rescaled non-cutoff Boltzmann equation \eqref{eq:feps_intro}.  Our result in Theorem~\ref{theo:boltzmann} is similar to the ones in \cite{DLSS, DSU2022}, but the proof is quite different. Indeed, thanks to new \emph{integrated-in-time regularization estimates}, we are able to prove the well-posedness of \eqref{eq:feps_intro} using a contraction fixed-point argument in a suitable functional space that takes into account these regularization estimates, which is the main novelty in Theorem~\ref{theo:boltzmann}. More precisely, we first investigate the semigroup $U^\eps$ associated to the linearized operator $\frac{1}{\eps^2} ( L - \eps v \cdot \nabla_x ) $ appearing in \eqref{eq:feps_intro}. We provide boundedness and integrated-in-time regularization estimates for $U^\eps$ (see  Proposition~\ref{prop:estimate_Ueps}), as well as for its integral in time against a source $\int_0^t U^\eps(t-\tau) S(\tau) \, \d \tau$ (see Proposition~\ref{prop:estimate_Ueps_regularization}). 
Together with nonlinear estimates for $\Gamma$ (see Lemma~\ref{lem:nonlinear}), we are then able to take $S$ equal to the nonlinear term $\Gamma(f, f)$ and prove the global well-posedness of mild solutions of \eqref{eq:feps_intro}, namely 
$$
f^\eps(t) = U^\eps(t) f^\eps_0 + \frac{1}{\eps}\int_0^t U^\eps (t-\tau) \Gamma(f^\eps(\tau) , f^\eps(\tau)) \, \d \tau,
$$
by applying a contraction fixed-point argument. The decay estimate is then obtained as a consequence of decay estimates for $U^\eps$ (see  Propositions~\ref{prop:estimate_Ueps_decay_hard_torus} and \ref{prop:estimate_Ueps_decay_hard_wholespace}) and for $\int_0^t U^\eps(t-\tau) S(\tau) \, \d \tau$ (see Propositions~\ref{prop:estimate_Ueps_regularization_hard_torus} and \ref{prop:estimate_Ueps_regularization_hard_wholespace}).
It is important to notice that the fixed-point takes place in the space $\FF^{-1}_x(L^1_\xi L^\infty_t L^2_v (\la v \ra^\ell) \cap L^1_\xi L^2_t H^{s,*}_v (\la v \ra^\ell))$ for the torus case, and in $\FF^{-1}_x(L^1_\xi L^\infty_t L^2_v(\la v \ra^\ell) \cap L^1_\xi L^2_t H^{s,**}_v(\la v \ra^\ell) ) \cap  \FF^{-1}_x(L^p_\xi L^\infty_t L^2_v(\la v \ra^\ell) \cap L^p_\xi L^2_t H^{s,**}_v(\la v \ra^\ell) )$ for the whole space, that is, the integrated-in-time regularization appears in the functional space.

It is worth mentioning that the integrated-in-time regularization estimates as well as the estimates for $\int_0^t U^\eps(t-\tau) S(\tau) \, \d \tau$ are the key ingredient of our method. On the one hand, they are the main novelty that allows us apply a contraction fixed-point argument as explained above. On the other hand, they are also crucial for establishing the strong convergence in the proof of the hydrodynamic limit established below in Theorem~\ref{theo:hydro_limit}.

\subsection{Well-posedness for the Navier-Stokes-Fourier system}

Our second result concerns the global well-posedness of the incompressible Navier-Stokes-Fourier system~\eqref{Navier-Stokes-Fourier system} for small initial data.

\begin{theo}[Global well-posedness for the Navier-Stokes-Fourier system]\label{theo:NSF}
There exists $\eta_1>0$ small enough such that the following holds:
\begin{enumerate}

\item Torus case $\Omega_x = \T^3$: For any initial data $(\rho_0 , u_0,\theta_0) \in \FF_x^{-1}(L^1_\xi)$ satisfying \eqref{eq:meanfree} and $ \| (\widehat \rho_0 , \widehat u_0, \widehat \theta_0) \|_{L^1_\xi} \le \eta_1$, there exists a unique global mild solution $(\rho,u,\theta) \in \FF^{-1}_x ( L^1_\xi L^\infty_t \cap L^1_\xi( \la \xi \ra) L^2_t)$ to the Navier-Stokes-Fourier system~\eqref{Navier-Stokes-Fourier system} satisfying \eqref{eq:meanfree} and the energy estimate
$$
\| (\widehat \rho, \widehat u, \widehat \theta) \|_{L^1_\xi L^\infty_t}
+\| \la \xi \ra (\widehat \rho, \widehat u, \widehat \theta) \|_{L^1_\xi L^2_t}
\lesssim \| (\widehat \rho_0 , \widehat u_0, \widehat \theta_0) \|_{L^1_\xi}.
$$

\item Whole space case $\Omega_x = \R^3$: Let $p \in (3/2, \infty]$. For any initial data $(\rho_0 , u_0,\theta_0) \in \FF_x^{-1}(L^1_\xi \cap L^p_\xi)$ satisfying $ \| (\widehat \rho_0 , \widehat u_0, \widehat \theta_0) \|_{L^1_\xi}  + \| (\widehat \rho_0 , \widehat u_0, \widehat \theta_0) \|_{L^p_\xi} \le \eta_1$, there exists a unique global mild solution $(\rho,u,\theta) \in \FF^{-1}_x ( L^1_\xi L^\infty_t \cap L^1_\xi(|\xi|) L^2_t \cap L^p_\xi L^\infty_t \cap L^p_\xi(|\xi|) L^2_t)$ to the Navier-Stokes-Fourier system~\eqref{Navier-Stokes-Fourier system} satisfying the energy estimate
$$
\begin{aligned}
\| (\widehat \rho, \widehat u, \widehat \theta) \|_{L^1_\xi L^\infty_t}
+\| |\xi| (\widehat \rho, \widehat u, \widehat \theta) \|_{L^1_\xi L^2_t} 
+\| (\widehat \rho, \widehat u, \widehat \theta) \|_{L^p_\xi L^\infty_t}
+\| |\xi| (\widehat \rho, \widehat u, \widehat \theta) \|_{L^p_\xi L^2_t} \\
\lesssim 
\| (\widehat \rho_0 , \widehat u_0, \widehat \theta_0) \|_{L^1_\xi}
+ \| (\widehat \rho_0 , \widehat u_0, \widehat \theta_0) \|_{L^p_\xi}.
\end{aligned}
$$
\end{enumerate}

\end{theo}

The incompressible Navier-Stokes equation, that is, the first equation in \eqref{Navier-Stokes-Fourier system}, possesses a vast literature so we only mention a few works in the three dimensional case below, and we refer the reader to the monographs \cite{LR,BCD} and the references therein for more details. On the one hand, global weak solutions for large initial data were obtained in the pioneering work \cite{Leray} (see also \cite{Hopf}). On the other hand, global mild solutions for small initial data were obtained in \cite{FK,Kato,Chemin,Calderon,Cannone,FLRT} in different Lebesgue and Sobolev spaces, and we refer again to the book~\cite{LR} for results in Besov and Morrey spaces. We mention in particular the work of Lei and Lin~\cite{LeiLin} where global mild solutions in the whole space $\R^3$ were constructed in the Fourier-based space $L^1_\xi (|\xi|^{-1}) L^\infty_t$.

Our results in Theorem~\ref{theo:NSF} are maybe not completely new, but we do not have a reference for this precise functional setting (observe that the functional spaces in Theorem~\ref{theo:NSF} correspond exactly to the same functional setting as in the global well-posedness for the Boltzmann equation in Theorem~\ref{theo:boltzmann}). Therefore, and also for the sake of completeness, we shall provide a complete proof of them in Section~\ref{section Navier-Stokes-Fourier}.

Our strategy for obtaining the global solution $u$ for the incompressible Navier-Stokes equation follows a standard fixed-point argument. As in the proof of Theorem~\ref{theo:boltzmann}, we first obtain boundedness and integrated-in-time regularization estimates for the semigroup $V$ associated to the operator $\nu_1 \Delta_x$ (see Proposition~\ref{prop:estimate_V}), as well as for its integral in time against a source $\int_0^t V(t-\tau) S(\tau) \, \d \tau$ (see Proposition~\ref{prop:estimate_V_regularization}). We then combine this with estimates for the nonlinear term $Q_{\mathrm{NS}}$ (see Lemma~\ref{lem:estimate_QNS}) to obtain, thanks to a fixed-point argument, the global well-posedness of mild solutions of the first equation in \eqref{Navier-Stokes-Fourier system}, namely 
$$
u(t) = V(t) u_0 + \int_0^t V (t-\tau) Q_{\mathrm{NS}} (u(\tau) , u(\tau)) \, \d \tau .
$$
Once the solution $u$ is constructed, we can obtain in a similar (and even easier) way the well-posedness of mild solutions of the second equation in \eqref{Navier-Stokes-Fourier system} for the temperature $\theta$. Finally we easily obtain the result for the density $\rho$ thanks to the last equation in \eqref{Navier-Stokes-Fourier system}.

\subsection{Hydrodynamic limit}

Our third result regards the hydrodynamic limit of the rescaled Boltzmann equation, that is, we are interested in the behavior of solutions $(f^\eps)_{\eps \in (0,1]}$ to~\eqref{eq:feps_intro} in the limit $\eps \to 0$.

Let $(\rho_0, u_0 , \theta_0)$ be an initial data veryfying \eqref{eq:rho0_u0_theta0} (and also \eqref{eq:meanfree} in the torus case) and consider the associated global solution $(\rho, u , \theta)$ to the incompressible Navier-Stokes-Fourier system~\eqref{Navier-Stokes-Fourier system} given by Theorem~\ref{theo:NSF}, where the viscosity coefficients $\nu_1, \nu_2 >0$ are given as follows (see \cite{BGL1}): Let us introduce the two unique functions $\Phi$ (which is a matrix-valued function) and $\Psi$ (which is a vector-valued function) orthogonal to $\Ker L$ such that
\[
\frac 1 {\sqrt{\mu}} L(\sqrt{\mu} \Phi) = \frac {|v|^2} 3 I_{3 \times 3} -  v \otimes v , \quad \frac 1 {\sqrt{\mu}} L(\sqrt{\mu} \Psi)  = \frac {5-|v|^2} 2  v ,
\]
then the viscosity coefficients are defined by
\[
\nu_1 = \frac 1 {10}  \int_{\R^3} L(\sqrt{\mu}   \Phi) \Phi \sqrt{\mu} \, \d v ,\quad \nu_2 = \frac 2 {15} \int_{\R^3} \Psi \cdot L(\sqrt{\mu}  \Psi) \sqrt{\mu} \, \d v
\]
We define the initial kinetic distribution $g_0 \in \Ker L$ associated to $(\rho_0, u_0 , \theta_0)$ by
\begin{equation}\label{eq:g_0}
g_0(x,v) = \P g_0 (x,v) = \left[ \rho_0(x) + u_0(x) \cdot v + \theta_0(x) \frac{(|v|^2-3)}{2} 	\right] \sqrt{\mu} (v), 
\end{equation}
and we suppose that $g_0$ is well-prepared in the sense
\begin{equation}\label{eq:well_prepared}
\nabla_x \cdot u_0 = 0 \quad\text{and} \quad \rho_0+\theta_0=0.
\end{equation}
We then consider the kinetic distribution $g(t) \in \Ker L$ associated to $(\rho(t), u(t) , \theta(t))$ by
\begin{equation}\label{eq:g(t)}
g(t,x,v) = \P g(t,x,v) = \left[ \rho(t,x) + u(t,x) \cdot v + \theta(t,x) \frac{(|v|^2-3)}{2} 	\right] \sqrt{\mu} (v).
\end{equation}

\begin{theo}[Hydrodynamic limit]\label{theo:hydro_limit}
Let $(f^\eps_0)_{\eps \in (0,1]}$ satisfy the hypotheses of Theorem~\ref{theo:boltzmann} and consider the associated global unique mild solution $(f^\eps)_{\eps \in (0,1]}$ to \eqref{eq:feps_intro}. 
Let also $(\rho_0, u_0 , \theta_0)$ satisfy the hypotheses of Theorem~\ref{theo:NSF} as well as \eqref{eq:well_prepared}, and consider the associated global unique mild solution $(\rho,u,\theta)$ to \eqref{Navier-Stokes-Fourier system}. Finally, let $g_0 = \P g_0$ be defined by \eqref{eq:g_0} and $g = \P g$ by \eqref{eq:g(t)}.
There exists $0 < \eta_2 < \min(\eta_0,\eta_1)$ such that if 
$$
\begin{aligned}
\max \left( \| \widehat f^\eps_0 \|_{L^1_\xi L^2_v} , \| \widehat g_0 \|_{L^1_\xi L^2_v} \right) &\le \eta_2  \quad \text{in the case} \quad \Omega_x = \T^3,\\
\max \left( \| \widehat f^\eps_0 \|_{L^1_\xi L^2_v} + \| \widehat f^\eps_0 \|_{L^p_\xi L^2_v} , \| \widehat g_0 \|_{L^1_\xi L^2_v} + \| \widehat g_0 \|_{L^p_\xi L^2_v} \right) &\le \eta_2  \quad \text{in the case} \quad \Omega_x = \R^3,
\end{aligned}
$$ 
for all $\eps \in (0,1]$ and
$$
\lim_{\eps \to 0} \| \widehat f^\eps_0 - \widehat g_0 \|_{L^1_\xi L^2_v} = 0,
$$
then there holds
\begin{equation}
\label{main hydrodynamical result}
\lim_{\eps \to 0} \| \widehat f^\eps - \widehat g \|_{L^1_\xi L^\infty_t L^2_v} = 0.
\end{equation}
\end{theo}

\begin{rem}\begin{enumerate}[label=(\roman*)]

\item One can get a explicit rate of convergence in \eqref{main hydrodynamical result} if we suppose that the initial data $g_0$
has some additional regularity in $x$, namely a rate of $\eps^\delta$ if the initial data $g_0$ satisfies $ \Vert \langle \xi \rangle^\delta \widehat g_0 \Vert_{L^1_\xi L^2_v} <\infty$ for $\delta \in (0, 1]$. We refer to \eqref{final hydrodynamical result1} and \eqref{final hydrodynamical result2} for a quantitative version of this result.

\item Our methods can also be applied to the Landau equation with Coulomb potential, and we obtain similar results as in Theorem~\ref{theo:boltzmann} and in Theorem~\ref{theo:hydro_limit}.

\item Our result concerns well-prepared data for the fluid equation, namely $(\rho_0, u_0, \theta_0)$ associated to the initial kinetic distribution $g_0$ satisfies \eqref{eq:well_prepared}. In the whole space, fluid initial data that are not well-prepared could be handled as in \cite{GT} by using dispersive estimates. In the case of the torus, we refer to \cite{Jiang-Xiong} who handle the initial fluid layers for fluid initial data that are not well-prepared. 

\end{enumerate}
\end{rem}

Before giving some comments on the above result and its strategy, we start by providing a short overview of the existing literature on the problem of deriving incompressible Navier-Stokes fluid equations from the kinetic Boltzmann one, and we refer to the book by Saint-Raymond \cite{SR} for a thorough presentation of the topic including other hydrodynamic limits. The first justifications of the link between kinetic and fluid equations were formal and based on asymptotic expansions by Hilbert, Chapman, Cowling and Grad (see \cite{H2, CE, Grad2}). The first rigorous convergence proofs based also on asymptotic expansions were given by Caflisch \cite{C} (see also \cite{LN} and \cite{ELM}). In those papers, the limit is justified up to the first singular time for the fluid equation. Guo \cite{Guo} has justified the limit towards the Navier-Stokes equation and beyond in Hilbert's expansion for the cutoff Boltzmann and Landau equations.

In the framework of large data solutions, the weak convergence of global renormalized solutions of the cutoff Boltzmann equation of \cite{DPL} towards global weak solution to the fluid system were obtained in \cite{BGL1, BGL2, GSR1,GSR2, LM1, LM2, SR}. Moreover, for the case of non-cutoff kernels, 
we refer to \cite{Arsenio} who proved the hydrodynamic limit from global renormalized solutions with defect measure of \cite{AV}.

We now discuss results in the framework of perturbative solutions, that is, solutions near the Maxwellian. Based on the spectral analysis of the linearized cutoff Boltzmann operator performed in \cite{Nicolaenko,CIP, EP}, some hydrodynamic results were obtained in \cite{Nishida,BU,GT}, see also \cite{CRT} for the Landau equation. Moreover, for the non-cutoff Boltzmann equation, we refer to \cite{JXZ} where the authors obtained a result of weak-$*$ convergence in $L^\infty_t (H^2_{x, v})$ towards the fluid system by proving uniform in $\eps$ estimates. Up to our knowledge, our paper is the first to prove a strong convergence towards the incompressible Navier-Stokes-Fourier system for the non-cutoff Boltzmann equation. We also note here that, compared to former hydrodynamical limit results, in our work we do not need any derivative assumption on the initial data.

We now describe our strategy in order to obtain strong convergence results. Our approach is inspired by the one used in \cite{BU} for the cutoff Boltzmann equation, which was also used more recently in \cite{B, GT} still for cutoff kernels and in \cite{CRT} for the Landau equation. Indeed, as in \cite{GT, CRT}, using the spectral analysis performed in \cite{EP,YY, YY2}, in order to prove our main convergence result, we reformulate the fluid equation in a kinetic fashion and we then study the equation satisfied by the difference between the kinetic and the fluid solutions. More precisely, we denote the kinetic solution by
$$
f^\eps(t) = U^\eps(t) f^\eps_0 + \Psi^\eps[f^\eps,f^\eps] (t),
$$
and we observe, thanks to \cite{BU}, that the kinetic distribution $g$ associated to the fluid solution $(\rho, u ,\theta)$ through \eqref{eq:g(t)} satisfies
$$
g(t) = U(t) g_0 + \Psi[g,g](t),
$$
where $U$ is obtained as the limit of $U^\eps$ and $\Psi$ as the limit of $\Psi^\eps$ when $\eps \to 0$. The idea is then to compute the norm of the difference $f^\eps - g$ by using convergence estimates from $U^\eps$ to $U$ (see Lemma~\ref{convergence U epsilon U}) and from $\Psi^\eps$ to $\Psi$ (see Lemma~\ref{decomposition Psi epsilon}), which are based on the spectral study of \cite{YY, YY2}, together with uniform in $\eps$ estimates for the kinetic solution $f^\eps$ from Theorem~\ref{theo:boltzmann}.
This was achieved in \cite{GT} for the cutoff Boltzmann equation by applying a fixed-point method, however, as explained in \cite{CRT}, this can not be directly applied to the non-cutoff Boltzmann and Landau equations due to the anisotropic loss of regularity in the nonlinear collision operator $\Gamma$. To overcome this difficulty for the Landau equation, the authors in \cite{CRT} proved new pointwise-in-time regularization estimates not only for the semigroup $U^\eps$ but also for the solution to the nonlinear rescaled kinetic equation, which were then used to close the estimates and obtain a result of strong convergence.

In our work, we propose a new method in order to obtain strong convergence in the hydrodynamic limit using only the integrated-in-time regularization estimates (as opposed to pointwise-in-time regularization estimates as in \cite{CRT}) for the semigroup $U^\eps$ as well as for $\int_0^t U^\eps(t-\tau) S(\tau) \, \d \tau$. More precisely, the fixed-point argument in the space $\FF^{-1}_x(L^1_\xi L^\infty_t L^2_v \cap L^1_\xi L^2_t H^{s,*}_v )$ for the torus case, or in $\FF^{-1}_x(L^1_\xi L^\infty_t L^2_v \cap L^1_\xi L^2_t H^{s,**}_v ) \cap  \FF^{-1}_x(L^p_\xi L^\infty_t L^2_v \cap L^p_\xi L^2_t H^{s,**}_v )$ for the whole space, used for the global well-posedness in Theorem~\ref{theo:boltzmann} above together with the corresponding energy estimates are sufficient to estimate the $\FF^{-1}_x(L^1_\xi L^\infty_t L^2_v)$-norm of the difference $f^\eps - g$ and obtain strong convergence.

\subsection{Organization of the paper}
In Section~\ref{section linearized Boltzmann}, we first establish basic properties for the rescaled linearized non-cutoff Boltzmann collision operator and then compute the basic estimates for the associated semigroup. In Section~\ref{section well-posedness Boltzmann} we prove the  well-posedness for the rescaled non-cutoff Boltzmann equation. We establish well-posedness for the Navier-Stokes-Fourier system in Section~\ref{section Navier-Stokes-Fourier}. Finally we obtain the hydrodynamical limit result in Section~\ref{section hydrodynamical limit}.

\section{Linearized Boltzmann operator}\label{section linearized Boltzmann}

It is well-known, see for instance \cite{MS} and the references therein, that the linearized Boltzmann collision operator $L$, defined in~\eqref{eq:def:L}, satisfies the following coercive-type inequality
\begin{equation}
\la L f , f \ra_{L^2_v} \le - \lambda \| \P^\perp f \|_{H^{s,*}_v}^2.
\end{equation}
where we recall that $\P^\perp = I - \P$ and $\P$ is the orthogonal projection onto $\Ker L$ given by \eqref{eq:def:Pf}.
For all $\eps \in (0,1]$ and all $\xi \in \Omega'_\xi$, we denote by $\Lambda^\eps (\xi)$ the Fourier transform in space of the full linearized operator $\frac{1}{\eps^2} L - \frac{1}{\eps} v \cdot \nabla_x$, namely 
\begin{equation}\label{eq:def:Lambdaeps_xi}
\Lambda^\eps (\xi) :=  \frac{1}{\eps^2} ( L - \mathrm{i} \eps v \cdot \xi ).
\end{equation}

We first gather dissipativity results for the operator $\Lambda^\eps (\xi)$ obtained for instance in \cite{Strain}, that we reformulate below as in \cite{CG} and inspired from \cite{CRT,BCMT} in order to take into account the different scales related to the parameter $\eps \in (0,1]$. For every $\xi \in \Omega'_\xi$ we define
\begin{equation*}
	\begin{aligned}
		B[ f ,  g ] (\xi)
		&:=  \frac{\delta_1 \mathrm{i}  }{\la \xi\ra^2}  \xi \theta[\widehat f(\xi)]  \cdot M[ \P^\perp \widehat g(\xi)]  
		+  \frac{\delta_1 \mathrm{i} }{\la \xi\ra^2}   \xi \theta[\widehat g(\xi)] \cdot M[ \P^\perp \widehat f (\xi)] \\
		&\quad 
		+  \frac{\delta_2\mathrm{i}}{\la \xi\ra^2} (\xi \otimes u[\widehat f(\xi)] )^{\mathrm{sym}} : \left\{ \Theta[\P^\perp \widehat g(\xi)] + \theta[\widehat g(\xi)] I \right\}\\
		&\quad 
		+  \frac{\delta_2\mathrm{i}}{\la \xi\ra^2} (\xi \otimes u[\widehat g(\xi)] )^{\mathrm{sym}} : \left\{\Theta[\P^\perp \widehat f(\xi)] + \theta[\widehat f(\xi)] I \right\}\\
		&\quad 
		+  \frac{\delta_3\mathrm{i} }{\la \xi\ra^2} \xi \rho[\widehat f(\xi)] \cdot  u[\widehat g(\xi)] 
		+  \frac{\delta_3\mathrm{i} }{\la \xi\ra^2} \xi \rho[\widehat g(\xi)] \cdot  u[\widehat f(\xi)] ,
	\end{aligned}
\end{equation*}
with constants $0 < \delta_3 \ll \delta_2 \ll \delta_1 \ll 1$, where $I$ is the $3 \times 3$ identity matrix and the moments $M$ and $\Theta$ are defined by
$$
M [f] = \int_{\R^3} f v (|v|^2-5) \sqrt{\mu}(v) \, \d v, \qquad 
\Theta[f] = \int_{\R^3} f \left(v \otimes v - I\right) \sqrt{\mu}(v) \, \d v,
$$
and where for vectors $a, b \in \R^3$ and matrices $A, B \in \R^{3 \times 3}$, we denote
$$
(a \otimes b)^{\mathrm{sym}} = \frac{1}{2} (a_j b_k + a_kb_j)_{1 \le j,k \le 3}, \qquad 
A : B = \sum_{j,k=1}^3 A_{jk} B_{jk}.
$$

We then define the inner product $\la\!\la \cdot , \cdot \ra\!\ra_{L^2_v}$ on $L^2_v$ (depending on $\xi$) by
\begin{equation}\label{eq:new_inner_product}
\la\!\la \widehat f(\xi), \widehat g (\xi) \ra\!\ra_{L^2_v} 
:=  \la \widehat f(\xi), \widehat g (\xi) \ra_{L^2_v} + \eps  B[f, g](\xi) ,
\end{equation}
and the associated norm
\begin{equation}\label{eq:new_norm}
\Nt \widehat f(\xi) \Nt_{L^2_v}^2 :=  \la\!\la \widehat f(\xi), \widehat f (\xi) \ra\!\ra_{L^2_v} .
\end{equation}
In a similar fashion, for any $\ell>0$, we define the inner product $\la\!\la \cdot , \cdot \ra\!\ra_{L^2_v (\la v \ra^\ell)}$ on $L^2_v(\la v \ra^\ell)$ (depending on $\xi$) by
\begin{equation}\label{eq:new_inner_product_ell}
\begin{aligned}
\la\!\la \widehat f(\xi), \widehat g (\xi) \ra\!\ra_{L^2_v(\la v \ra^\ell)} 
&:=  \la \widehat f(\xi), \widehat g (\xi) \ra_{L^2_v} 
+\delta_0 \la \P^\perp \widehat f(\xi), \P^\perp \widehat g (\xi) \ra_{L^2_v (\la v \ra^\ell)} \\
&\quad + \eps  B[f, g](\xi) ,
\end{aligned}
\end{equation}
with $\delta_1 \ll \delta_0 \ll 1$, and the associated norm
\begin{equation}\label{eq:new_norm_ell}
\Nt \widehat f(\xi) \Nt_{L^2_v(\la v \ra^\ell)}^2 :=  \la\!\la \widehat f(\xi), \widehat f (\xi) \ra\!\ra_{L^2_v(\la v \ra^\ell)} .
\end{equation}
It is important to notice the factor $\eps$ in front of the last term in the right-hand side of \eqref{eq:new_inner_product} and \eqref{eq:new_inner_product_ell}.

Arguing as in \cite{Strain}, the main difference being the factor $\eps$ at the second term of \eqref{eq:new_inner_product} and \eqref{eq:new_inner_product_ell}, we obtain the following dissipativity result.

\begin{prop}\label{prop:hypocoercivity}
We can choose $0 < \delta_3 \ll \delta_2 \ll \delta_1 \ll \delta_0 \ll 1$ appropriately such that:
\begin{enumerate}

\item The new norm $\Nt \cdot \Nt_{L^2_v( \la v \ra^\ell)}$ is equivalent to the usual norm $\| \cdot \|_{L^2_v( \la v \ra^\ell)}$ on $L^2_v( \la v \ra^\ell)$ with bounds that are independent of $\xi$ and $\eps$.

\item If $\Omega_x = \T^3$, for every $f$ satisfying \eqref{eq:normalization2} we have, for all $\xi \in \Z^3$,
$$
\re \la\!\la \Lambda^\eps (\xi) \widehat f (\xi) , \widehat f (\xi) \ra\!\ra_{L^2_v (\la v \ra^{\ell})} \le - \lambda_0 \left( \frac{1}{\eps^2} \| \P^\perp \widehat f (\xi) \|_{H^{s,*}_v(\la v \ra^{\ell})}^2 + \| \P \widehat f (\xi) \|_{L^2_v}^2 \right),
$$
for some constant $\lambda_0>0$.

\item If $\Omega_x = \R^3$, for every $f$ we have, for all $\xi \in \R^3$,
$$
\re \la\!\la \Lambda^\eps (\xi) \widehat f (\xi) , \widehat f (\xi) \ra\!\ra_{L^2_v (\la v \ra^{\ell})} \le - \lambda_0 \left( \frac{1}{\eps^2} \| \P^\perp \widehat f (\xi) \|_{H^{s,*}_v(\la v \ra^{\ell})}^2 + \frac{|\xi|^2}{\la \xi \ra^2}\| \P \widehat f (\xi) \|_{L^2_v}^2 \right),
$$
for some constant $\lambda_0>0$.

\end{enumerate}

\end{prop}

The aim of this remainder section is to obtain, using the dissipativity result of Proposition~\ref{prop:hypocoercivity}, decay and regularization estimates for the semigroup associated to the linearized operator $\widehat \Lambda^\eps$. We denote in the sequel by
\begin{equation}\label{eq:def:hatUeps}
\widehat U^\eps (t,\xi) = e^{t \Lambda^\eps (\xi)},
\end{equation}
the semigroup associated to $\Lambda^\eps (\xi)$, and by 
\begin{equation}\label{eq:def:Ueps}
U^\eps(t) = \FF_x^{-1} \widehat U^\eps (t) \FF_x ,
\end{equation}
the semigroup associated to $\frac{1}{\eps^2} ( L - \eps v \cdot \nabla_x ) $.

\subsection{Boundedness and regularization estimates}

We first provide boundedness and integrated-in-time regularization estimates for the semigroup $U^\eps$ (see Proposition~\ref{prop:estimate_Ueps}) as well as its integral in time against a source $\int_0^t U^\eps(t-\tau) S(\tau) \, \d \tau$ (see Proposition~\ref{prop:estimate_Ueps_regularization}). These are the key estimates we shall use later in order to prove the well-posedness results for the rescaled Boltzmann equation~\eqref{eq:feps_intro} in Theorem~\ref{theo:boltzmann}. They are also crucial for establishing the convergence of some of the terms in the proof of the hydrodynamic limit in Theorem~\ref{theo:hydro_limit}.

\begin{prop}\label{prop:estimate_Ueps}
Let $\ell \ge 0$ and $p \in [1,\infty]$.
Let $\widehat f_0 \in L^1_\xi L^2_v (\la v \ra^\ell)$ and suppose moreover that $f_0$ verifies \eqref{eq:normalization2} in the torus case $\Omega_x = \T^3$. Then
\begin{multline*}
\|  \widehat U^\eps (\cdot) \widehat f_0 \|_{L^p_\xi L^\infty_t L^2_v(\la v \ra^\ell) } 
+\frac{1}{\eps} \| \P^\perp \widehat U^\eps (\cdot) \widehat f_0 \|_{L^p_\xi L^2_t H^{s,*}_v (\la v \ra^\ell)}
+ \left\|  \frac{|\xi|}{\la \xi \ra} \P \widehat U^\eps (\cdot) \widehat f_0 \right\|_{L^p_\xi L^2_t L^2_v } 
\lesssim \| \widehat f_0 \|_{L^p_\xi L^2_v (\la v \ra^\ell)} ,
\end{multline*}
and moreover, in the torus case, we also have that $U^\eps (t) f_0$ verifies~\eqref{eq:normalization2} for all $t \ge 0$.
\end{prop}

\begin{rem}\label{rem:prop:estimate_Ueps}
Observe that, in the torus case $\Omega_x = \T^3$, one can replace the term $ \frac{|\xi|}{\la \xi \ra} \P \widehat U^\eps (\cdot) \widehat f_0$ in above estimate by $\P \widehat U^\eps (\cdot) \widehat f_0$ since $U^\eps (t) f_0$ verifies~\eqref{eq:normalization2}.
\end{rem}

\begin{proof}
Let $f(t) = U^\eps(t) f_0$ for all $t \ge 0$, which satisfies the equation
\begin{equation}\label{eq:Ueps_f0}
\partial_t f = \frac{1}{\eps^2} (L - \eps v \cdot \nabla_x) f, \quad 
f_{| t=0} = f_0.
\end{equation}
We already observe that in the case of the torus, $f(t)$ verifies \eqref{eq:normalization2} thanks to the properties of $L$.
Moreover, for all $\xi \in \Z^3$ (if $\Omega_x = \T^3$) or all $ \xi \in \R^3$ (if $\Omega_x = \R^3)$, the Fourier transform in space $\widehat f$ satisfies
\begin{equation}\label{eq:hatUeps_hatf0}
\partial_t \widehat f(\xi) = \Lambda^\eps (\xi) \widehat f(\xi), \quad \widehat f(\xi)_{|t=0} = \widehat f_0(\xi).
\end{equation}

Using Proposition~\ref{prop:hypocoercivity} we have, for all $t \ge 0$,
$$
\begin{aligned}
\frac{1}{2} \frac{\d}{\d t} \Nt \widehat f(\xi) \Nt_{L^2_v(\la v \ra^\ell) }^2
&= \re \la\!\la \Lambda^\eps (\xi) \widehat f (\xi) , \widehat f (\xi) \ra\!\ra_{L^2_v (\la v \ra^\ell)} \\
&\le - \lambda_0 \left( \frac{1}{\eps^2} \| \P^\perp \widehat f (\xi) \|_{H^{s,*}_v(\la v \ra^\ell)}^2 + \frac{|\xi|^2}{\la \xi \ra^2}\| \P \widehat f (\xi) \|_{L^2_v}^2 \right),
\end{aligned}
$$
which implies, for all $t \ge 0$,
$$
\begin{aligned}
\| \widehat f(t,\xi) \|_{L^2_v(\la v \ra^\ell)}^2
+\frac{1}{\eps^2}\int_0^t \| \P^\perp \widehat f(\tau,\xi) \|_{H^{s,*}_v(\la v \ra^\ell)}^2 \, \d \tau
+\int_0^t  \frac{|\xi|^2}{\la \xi \ra^2} \| \P \widehat f(\tau,\xi) \|_{L^2_v}^2 \, \d \tau 
\lesssim \| \widehat f_0 (\xi)\|_{L^2_v(\la v \ra^\ell)}^2.
\end{aligned}
$$
where we have used that $\Nt \widehat f(\xi) \Nt_{L^2_v}$ is equivalent to $\| \widehat f(\xi) \|_{L^2_v}$ independently of $\xi$ and $\eps$.
Taking the supremum in time and then taking the square-root of previous estimate yields
$$
\begin{aligned}
\| \widehat f(\xi) \|_{L^\infty_t L^2_v(\la v \ra^\ell)}
+\frac{1}{\eps}\| \P^\perp \widehat f(\xi) \|_{L^2_t H^{s,*}_v(\la v \ra^\ell)}
+\left\|  \frac{|\xi|}{\la \xi \ra} \P \widehat f(\xi) \right\|_{L^2_t L^2_v}
\lesssim \| \widehat f_0 (\xi) \|_{L^2_v},
\end{aligned}
$$
and we conclude by taking the $L^p_\xi$ norm.
\end{proof}

\begin{prop}\label{prop:estimate_Ueps_regularization}
Let $\ell \ge 0$ and $p \in [1,\infty]$.
Let $S=S(t,x,v)$ verify $\P S = 0$ and $ \la v \ra^\ell \widehat S \in L^p_\xi L^2_t (H^{s,*}_v)'$, and denote
$$
g_S(t) = \int_0^t  U^\eps (t-\tau) S(\tau) \, \d \tau.
$$
Then
$$
\begin{aligned}
\left\|  \widehat g_S \right\|_{L^p_\xi L^\infty_t L^2_v(\la v \ra^\ell)} 
&+\frac{1}{\eps}\|  \P^\perp \widehat g_S \|_{L^p_\xi L^2_t H^{s,*}_v(\la v \ra^\ell)} 
+\left\| \frac{|\xi|}{\la \xi \ra} \P \widehat g_S  \right\|_{L^p_\xi L^2_t L^2_v} 
\lesssim \eps \| \la v \ra^\ell \widehat S \|_{L^p_\xi L^2_t (H^{s,*}_v)'} .
\end{aligned}
$$
\end{prop}

\begin{rem}
As in Remark~\ref{rem:prop:estimate_Ueps}, we observe that in the torus case $\Omega_x = \T^3$ one can replace the term $ \frac{|\xi|}{\la \xi \ra} \P \widehat g_S$ in above estimate by $\P \widehat g_S$.
\end{rem}

\begin{proof}
We first observe that $g_S$ satisfies the equation
\begin{equation}\label{eq:gS}
\partial_t  g_S = \frac{1}{\eps^2} ( L - \eps v \cdot \nabla_x )g_S +  S, \quad  g_{|t=0} = 0,
\end{equation}
thus, for all  for all $\xi \in \Z^3$ (if $\Omega_x = \T^3$) or all $ \xi \in \R^3$ (if $\Omega_x = \R^3)$,
\begin{equation}\label{eq:hatgS}
\partial_t \widehat g_S(\xi) = \Lambda^\eps (\xi) \widehat g_S(\xi) + \widehat S(\xi), \quad \widehat g(\xi)_{|t=0} = 0,
\end{equation}
that is, for all $t \ge 0$ ,
\begin{equation}\label{eq:hatgS_integrated}
\widehat g_S(t,\xi) =\int_0^t  \widehat U^\eps (t-\tau,\xi) \widehat S(\tau,\xi) \, \d \tau.
\end{equation}

We remark from \eqref{eq:new_inner_product} (if $\ell=0$) or \eqref{eq:new_inner_product_ell} (if $\ell >0$) and the fact that $\P S = 0$ that
$$
\begin{aligned}
\la\!\la \widehat S (\xi) ,  \widehat g_S (\xi) \ra\!\ra_{L^2_v(\la v \ra^{\ell})}
&= \la \widehat S (\xi) ,   \widehat g_S (\xi) \ra_{L^2_v} 
+ \delta_0 \la \P^\perp \widehat S (\xi) ,   \P^\perp \widehat g_S (\xi) \ra_{L^2_v (\la v \ra^\ell)} 
+ \eps  B [S , g_S ](\xi) \\
&= \la \widehat S (\xi) ,  \P^\perp \widehat g_S (\xi) \ra_{L^2_v} 
+ \delta_0 \la \P^\perp \widehat S (\xi) ,   \P^\perp \widehat g_S (\xi) \ra_{L^2_v (\la v \ra^\ell)} 
+ \eps  B [S , g_S ](\xi) .
\end{aligned}
$$
Using again that $\P S = 0$, so that $\rho[S] = u[S] = \theta[S]=0$, we have
$$
\begin{aligned}
B [ S, g_S](\xi)
= \frac{\delta_1 \mathrm{i} }{1+|\xi|^2}   \xi \theta[\widehat g_S(\xi)] \cdot M[ \P^\perp \widehat S (\xi)]
+  \frac{\delta_2\mathrm{i}}{1+|\xi|^2} (\xi \otimes u[\widehat g_S(\xi)] )^{\mathrm{sym}} : \Theta[\P^\perp \widehat S(\xi)] , 
\end{aligned}
$$
therefore observing that for any polynomial $p=p(v)$ there holds
$$
\left|\int_{\R^3} \widehat S(\xi) p(v) \sqrt{\mu}(v) \, \d v  \right|
\lesssim \| \widehat S(\xi) \|_{(H^{s,*}_v)'},
$$
we get
$$
|B [S , g_S ](\xi)|
\lesssim \|  \P^\perp \widehat S(\xi) \|_{(H^{s,*}_v)'} \frac{|\xi|}{\la \xi \ra^2}\| \P \widehat g_S(\xi) \|_{L^2_v}
\lesssim \|  \P^\perp \widehat S(\xi) \|_{(H^{s,*}_v)'} \frac{|\xi|}{\la \xi \ra}\| \P \widehat g_S(\xi) \|_{L^2_v}.
$$
Moreover
$$
\la \widehat S (\xi) ,  \P^\perp \widehat g_S (\xi) \ra_{L^2_v} 
\lesssim \| \widehat S(\xi) \|_{(H^{s,*}_v)'}  \| \P^\perp \widehat g_S (\xi) \|_{H^{s,*}_v},
$$
and
$$
\begin{aligned}
\la \widehat S (\xi) ,   \P^\perp \widehat g_S (\xi) \ra_{L^2_v (\la v \ra^\ell)} 
&= \la \la v \ra^\ell \widehat S (\xi) ,   \la v \ra^\ell \P^\perp \widehat g_S (\xi) \ra_{L^2_v} \\
&\lesssim \| \la v \ra^\ell \widehat S(\xi) \|_{(H^{s,*}_v)'}  \| \la v \ra^\ell\P^\perp \widehat g_S (\xi) \|_{H^{s,*}_v} ,
\end{aligned}
$$
Gathering previous estimates yields
\begin{equation}\label{eq:bound_S}
\begin{aligned}
\la\!\la \widehat S (\xi) ,  \widehat g_S (\xi) \ra\!\ra_{L^2_v (\la v \ra^{\ell})}
&\lesssim \| \la v \ra^{\ell} \widehat S(\xi) \|_{(H^{s,*}_v)'} \left( \| \P^\perp \widehat g_S (\xi)\|_{H^{s,*}_v(\la v \ra^{\ell})} + \eps \frac{|\xi|}{\la \xi \ra} \| \P \widehat g_S (\xi)\|_{L^2_v} \right).
\end{aligned}
\end{equation}

Using Proposition~\ref{prop:hypocoercivity} and arguing as in Proposition~\ref{prop:estimate_Ueps} we have, for all $t \ge 0$ and all $\xi \in \Omega'_\xi$,
\begin{equation}\label{eq:dt_hatgS}
\begin{aligned}
\frac{1}{2} \frac{\d}{\d t} \Nt \widehat g_S(\xi) \Nt_{L^2_v(\la v \ra^{\ell})}^2
&\le - \lambda_0   \left( \frac{1}{\eps^2} \| \P^\perp \widehat g_S (\xi) \|_{H^{s,*}_v(\la v \ra^{\ell})}^2 + \frac{|\xi|^2}{\la \xi \ra^2}\| \P \widehat g_S (\xi) \|_{L^2_v}^2 \right) \\
&\quad
+ C \| \la v \ra^{\ell} \widehat S(\xi) \|_{(H^{s,*}_v)'} \left( \| \P^\perp \widehat g_S (\xi)\|_{H^{s,*}_v(\la v \ra^{\ell})} + \eps \frac{|\xi|}{\la \xi \ra} \| \P \widehat g_S (\xi)\|_{L^2_v} \right)\\
&\le - \frac{\lambda_0}{2}  \left( \frac{1}{\eps^2} \| \P^\perp \widehat g_S (\xi) \|_{H^{s,*}_v(\la v \ra^{\ell})}^2 + \frac{|\xi|^2}{\la \xi \ra^2} \| \P \widehat g_S (\xi) \|_{L^2_v}^2 \right) \\
&\quad
+ C \eps^2 \| \la v \ra^{\ell} \widehat S(\xi) \|_{(H^{s,*}_v)'}^2 ,
\end{aligned}
\end{equation}
where we have used Young's inequality in last line, which implies
$$
\begin{aligned}
\| \widehat g_S(t,\xi) \|_{L^2_v(\la v \ra^{\ell})}^2
+\frac{1}{\eps^2}\int_0^t \| \P^\perp \widehat g_S(\tau,\xi) \|_{H^{s,*}_v(\la v \ra^{\ell})}^2 \, \d \tau
&+\int_0^t  \frac{|\xi|^2}{\la \xi \ra^2} \| \P \widehat g_S(\tau,\xi) \|_{L^2_v}^2 \, \d \tau \\
&\quad
\lesssim \eps^2 \int_0^t \|  \la v \ra^{\ell} \widehat S(\tau,\xi) \|_{(H^{s,*}_v)'}^2 \, \d \tau.
\end{aligned}
$$
Taking the supremum in time and then taking the square-root of previous estimate yields
$$
\begin{aligned}
\|  \widehat g_S(\xi) \|_{L^\infty_t L^2_v(\la v \ra^{\ell})}
+\frac{1}{\eps}\|  \P^\perp \widehat g_S(\xi) \|_{L^2_t H^{s,*}_v(\la v \ra^{\ell})}
+\left\|  \frac{|\xi|}{\la \xi \ra} \P \widehat g_S(\xi) \right\|_{L^2_t L^2_v}
\lesssim \eps \|  \la v \ra^{\ell} \widehat S(\xi) \|_{L^2_t (H^{s,*}_v)'},
\end{aligned}
$$
and we conclude by taking the $L^p_\xi$ norm.
\end{proof}

\subsection{Decay estimates: Hard potentials in the torus}\label{sec:linear:HP}

In this subsection we shall always assume $\gamma + 2s \ge 0$ and $\Omega_x = \T^3$, and we shall obtain decay estimates for the semigroup $U^\eps$ (see Proposition~\ref{prop:estimate_Ueps_decay_hard_torus}) as well as its integral in time against a source $\int_0^t U^\eps(t-\tau) S(\tau) \, \d \tau$ (see Proposition~\ref{prop:estimate_Ueps_regularization_hard_torus}). We recall that given any real number $\lambda \in \R$ we denote $\mathrm{e}_\lambda : t \mapsto e^{\lambda t}$.

\begin{prop}\label{prop:estimate_Ueps_decay_hard_torus}
Let $\ell \ge 0$. Let $\widehat f_0 \in L^1_\xi L^2_v(\la v \ra^{\ell}) $, then
\begin{multline*}
\| \mathrm{e}_\lambda \widehat U^\eps (\cdot) \widehat f_0 \|_{L^1_\xi L^\infty_t L^2_v (\la v \ra^{\ell})} 
+\frac{1}{\eps} \| \mathrm{e}_\lambda \P^\perp \widehat U^\eps (\cdot) \widehat f_0 \|_{L^1_\xi L^2_t H^{s,*}_v(\la v \ra^{\ell}) }
+ \| \mathrm{e}_\lambda  \P \widehat U^\eps (\cdot) \widehat f_0 \|_{L^1_\xi L^2_t L^2_v } \\
\lesssim \| \widehat f_0 \|_{L^1_\xi L^2_v(\la v \ra^{\ell}) } ,
\end{multline*}
for some $\lambda >0$ (depending on $\lambda_0$ of Proposition~\ref{prop:hypocoercivity}).
\end{prop}

\begin{proof}
Let $f(t) = U^\eps(t) f_0$ for all $t \ge 0$ which satisfies \eqref{eq:Ueps_f0}, so that $\widehat f(t,\xi) = \widehat U^\eps(t,\xi) \widehat f_0(\xi)$ satisfies \eqref{eq:hatUeps_hatf0} for all $\xi \in \Z^3$. 
Using Proposition~\ref{prop:hypocoercivity} we have, for all $t \ge 0$ and some $\lambda_0>0$,
$$
\begin{aligned}
\frac{1}{2} \frac{\d}{\d t} \Nt \widehat f(\xi) \Nt_{L^2_v (\la v \ra^{\ell})}^2
&\le - \lambda_0 \left( \frac{1}{\eps^2} \| \P^\perp \widehat f (\xi) \|_{H^{s,*}_v(\la v \ra^{\ell})}^2 + \| \P \widehat f (\xi) \|_{L^2_v}^2 \right),
\end{aligned}
$$
which implies, since $\| \cdot \|_{H^{s,*}_v (\la v \ra^{\ell})} \ge \| \la v \ra^{\gamma/2+s} \cdot \|_{L^2_v (\la v \ra^{\ell}) } \ge  \| \cdot \|_{L^2_v(\la v \ra^{\ell}) }$ and the fact that $\Nt \widehat f(\xi) \Nt_{L^2_v(\la v \ra^{\ell})}$ is equivalent to $\| \widehat f(\xi) \|_{L^2_v(\la v \ra^{\ell}) }$  independently of $\xi$ and $\eps$, that
$$
\begin{aligned}
\frac{1}{2} \frac{\d}{\d t} \Nt \widehat f(\xi) \Nt_{L^2_v(\la v \ra^{\ell})}^2
\le - \lambda \Nt \widehat f (\xi) \Nt_{L^2_v(\la v \ra^{\ell})}^2- \sigma  \left( \frac{1}{\eps^2} \| \P^\perp \widehat f (\xi) \|_{H^{s,*}_v(\la v \ra^{\ell})}^2 + \| \P \widehat f (\xi) \|_{L^2_v}^2 \right),
\end{aligned}
$$
for some positive constants $\lambda , \sigma >0$ depending only on the implicit constants in Proposition~\ref{prop:hypocoercivity}--(1) and on $\lambda_0>0$ appearing in Proposition~\ref{prop:hypocoercivity}--(2).
We therefore deduce
$$
\begin{aligned}
\frac{\d}{\d t} \left\{ e^{2\lambda t} \Nt \widehat f(\xi) \Nt_{L^2_v(\la v \ra^{\ell})}^2 \right\}
\le - \sigma e^{2\lambda t}  \left( \frac{1}{\eps^2} \| \P^\perp \widehat f (\xi) \|_{H^{s,*}_v(\la v \ra^{\ell})}^2 + \| \P \widehat f (\xi) \|_{L^2_v}^2 \right),
\end{aligned}
$$
which implies, for all $t \ge 0$,
$$
\begin{aligned}
e^{2\lambda t}\| \widehat f(t,\xi) \|_{L^2_v(\la v \ra^{\ell})}^2
+\frac{1}{\eps^2}\int_0^t e^{2\lambda s}\| \P^\perp \widehat f(\tau,\xi) \|_{H^{s,*}_v(\la v \ra^{\ell})}^2 \, \d \tau
+\int_0^t e^{2\lambda s}  \| \P \widehat f(\tau,\xi) \|_{L^2_v}^2 \, \d \tau \\
\lesssim \| \widehat f_0 (\xi)\|_{L^2_v(\la v \ra^{\ell})}^2.
\end{aligned}
$$
where we have used again that $\Nt \widehat f(\xi) \Nt_{L^2_v (\la v \ra^{\ell}) }$ is equivalent to $\| \widehat f(\xi) \|_{L^2_v (\la v \ra^{\ell}) }$ independently of $\xi$ and $\eps$.
Taking the supremum in time and then taking the square-root of previous estimate yields
$$
\begin{aligned}
\| \mathrm{e}_\lambda\widehat f(\xi) \|_{L^\infty_t L^2_v(\la v \ra^{\ell})}
+\frac{1}{\eps}\| \mathrm{e}_\lambda \P^\perp \widehat f(\xi) \|_{L^2_t H^{s,*}_v(\la v \ra^{\ell})}
+\| \mathrm{e}_\lambda \P \widehat f(\xi) \|_{L^2_t L^2_v}
\lesssim \| \widehat f_0 (\xi) \|_{L^2_v(\la v \ra^{\ell})},
\end{aligned}
$$
and we conclude by taking the $L^1_\xi$ norm.
\end{proof}

\begin{prop}\label{prop:estimate_Ueps_regularization_hard_torus}
Let $\ell \ge 0$. Let $\lambda>0$ be given in Proposition~\ref{prop:estimate_Ueps_decay_hard_torus}.
Let $S=S(t,x,v)$ verify $\P S = 0$ and $\mathrm{e}_{\lambda} \la v \ra^{\ell}  \widehat S \in L^1_\xi L^2_t (H^{s,*}_v)'$, and denote
$$
g_S(t) = \int_0^t  U^\eps (t-\tau) S(\tau) \, \d \tau.
$$
Then
$$
\begin{aligned}
\left\| \mathrm{e}_{\lambda} \widehat g_S \right\|_{L^1_\xi L^\infty_t L^2_v(\la v \ra^{\ell})} 
&+\frac{1}{\eps}\| \mathrm{e}_{\lambda} \P^\perp \widehat g_S \|_{L^1_\xi L^2_t H^{s,*}_v(\la v \ra^{\ell})} 
+\left\|\mathrm{e}_{\lambda}  \P \widehat g_S  \right\|_{L^1_\xi L^2_t L^2_v} 
\lesssim \eps \| \mathrm{e}_{\lambda} \la v \ra^{\ell} \widehat S \|_{L^1_\xi L^2_t (H^{s,*}_v)'} .
\end{aligned}
$$
\end{prop}

\begin{proof}
Recall that $g_S$ satisfies equation~\eqref{eq:gS} and $\widehat g$ verifies~\eqref{eq:hatgS} for all $\xi \in \Z^3$ as well as \eqref{eq:hatgS_integrated}.
Thanks to \eqref{eq:dt_hatgS} and using that  $\| \cdot \|_{H^{s,*}_v (\la v \ra^{\ell})} \ge \| \la v \ra^{\gamma/2+s} \cdot \|_{L^2_v (\la v \ra^{\ell})} \ge  \| \cdot \|_{L^2_v (\la v \ra^{\ell})}$ as in the proof of Proposition~\ref{prop:estimate_Ueps_decay_hard_torus}, we get for all $t \ge 0$
$$
\begin{aligned}
\frac{1}{2} \frac{\d}{\d t} \Nt \widehat g_S(\xi) \Nt_{L^2_v(\la v \ra^{\ell})}^2
&\le - \lambda \| \widehat g_S (\xi) \|_{L^2_v(\la v \ra^{\ell})}^2- \sigma \left( \frac{1}{\eps^2} \| \P^\perp \widehat g_S (\xi) \|_{H^{s,*}_v(\la v \ra^{\ell})}^2 +  \| \P \widehat g_S (\xi) \|_{L^2_v}^2 \right) \\
&\quad
+ C \eps^2 \| \la v \ra^{\ell} \widehat S(\xi) \|_{(H^{s,*}_v)'}^2,
\end{aligned}
$$
for some constants $\lambda,\sigma,C>0$. We therefore deduce
$$
\begin{aligned}
\frac{\d}{\d t} \left\{ e^{2\lambda t} \Nt \widehat g_S(\xi) \Nt_{L^2_v(\la v \ra^{\ell})}^2 \right\}
&\le - \sigma e^{2\lambda t}  \left( \frac{1}{\eps^2} \| \P^\perp \widehat g_S (\xi) \|_{H^{s,*}_v(\la v \ra^{\ell})}^2 +\| \P \widehat g_S (\xi) \|_{L^2_v}^2 \right)\\
&\quad
+ C \eps^2 e^{2\lambda t}\| \la v \ra^{\ell} \widehat S(\xi) \|_{(H^{s,*}_v)'}^2,
\end{aligned}
$$
which implies, for all $t \ge 0$,
$$
\begin{aligned}
e^{2\lambda t}\| \widehat g_S(t,\xi) \|_{L^2_v(\la v \ra^{\ell})}^2
&+\frac{1}{\eps^2}\int_0^t e^{2\lambda s}\| \P^\perp \widehat g_S(\tau,\xi) \|_{H^{s,*}_v(\la v \ra^{\ell})}^2 \, \d \tau \\
&+\int_0^t e^{2\lambda s} \| \P \widehat g_S(\tau,\xi) \|_{L^2_v}^2 \, \d \tau
\lesssim \eps^2 \int_0^t e^{2\lambda s}\| \la v \ra^{\ell} \widehat S(\tau,\xi) \|_{(H^{s,*}_v)'}^2 \, \d \tau.
\end{aligned}
$$
Taking the supremum in time and then taking the square-root of previous estimate yields
$$
\begin{aligned}
\| \mathrm{e}_{\lambda} \widehat g_S(\xi) \|_{L^\infty_t L^2_v(\la v \ra^{\ell})}
+\frac{1}{\eps}\| \mathrm{e}_{\lambda} \P^\perp \widehat g_S(\xi) \|_{L^2_t H^{s,*}_v(\la v \ra^{\ell})}
+\left\| \mathrm{e}_{\lambda}  \P \widehat g_S(\xi) \right\|_{L^2_t L^2_v}
\lesssim \eps \| \mathrm{e}_{\lambda}  \la v \ra^{\ell} \widehat S(\xi) \|_{L^2_t (H^{s,*}_v)'},
\end{aligned}
$$
and we conclude by taking the $L^1_\xi$ norm.
\end{proof}

\subsection{Decay estimates: Soft potentials in the torus}\label{sec:linear:SP}

In this subsection we shall always assume $\gamma + 2s < 0$ and $\Omega_x = \T^3$, and we shall obtain decay estimates for the semigroup $U^\eps$ (see Proposition~\ref{prop:estimate_Ueps_decay_soft_torus}) as well as its integral in time against a source $\int_0^t U^\eps(t-\tau) S(\tau) \, \d \tau$ (see Proposition~\ref{prop:estimate_Ueps_regularization_soft_torus}). We recall that given any real number $\omega \in \R$ we denote $\mathrm{p}_\omega : t \mapsto (1+t)^{\omega}$.

\begin{prop}\label{prop:estimate_Ueps_decay_soft_torus}
Let $\ell > 0$ and $\widehat f_0 \in L^1_\xi L^2_v(\la v \ra^{\ell})$, then for any $0< \omega < \frac{\ell}{|\gamma+2s|}$ we have
$$
\begin{aligned}
\| \mathrm{p}_\omega \widehat U^\eps (\cdot) \widehat f_0 \|_{L^1_\xi L^\infty_t L^2_v} 
+\frac{1}{\eps} \| \mathrm{p}_\omega \P^\perp \widehat U^\eps (\cdot) \widehat f_0 \|_{L^1_\xi L^2_t H^{s,*}_v}
+ \| \mathrm{p}_\omega \P (\widehat U^\eps (\cdot) \widehat f_0) \|_{L^1_\xi L^2_t L^2_v} \\
\lesssim \| \widehat f_0 \|_{L^1_\xi L^2_v} 
+ \| \widehat U^\eps (\cdot) \widehat f_0 \|_{L^1_\xi L^\infty_t L^2_v(\la v \ra^\ell)} .
\end{aligned}
$$
\end{prop}

\begin{proof}
Arguing as in the proof of Proposition~\ref{prop:estimate_Ueps_decay_hard_torus}, denoting $f(t) = U^\eps(t) f_0$ and using that $\| \cdot \|_{H^{s,*}} \ge \| \la v \ra^{\gamma/2 + s} \cdot \|_{L^2_v}$, we obtain
\begin{equation}\label{eq:dt_hatf_soft}
\begin{aligned}
\frac{1}{2} \frac{\d}{\d t} \Nt \widehat f(\xi) \Nt_{L^2_v}^2
\le - \lambda \| \la v \ra^{\gamma/2 +s} \widehat f (\xi) \|_{L^2_v}^2- \sigma  \left( \frac{1}{\eps^2} \| \P^\perp \widehat f (\xi) \|_{H^{s,*}_v}^2 + \| \P \widehat f (\xi) \|_{L^2_v}^2 \right),
\end{aligned}
\end{equation}
for some positive constants $\lambda , \sigma >0$.

We now observe the following interpolation inequality: for any $R>0$ there holds
\begin{equation}\label{eq:interpolation}
\Nt \widehat f(\xi) \Nt_{L^2_v}^2 \lesssim \la R \ra^{|\gamma+2s|} \| \la v \ra^{\gamma/2+s} \widehat f(\xi) \|_{L^2_v}^2 + \la R \ra^{-2 \ell} \| \widehat f(\xi) \|_{L^2_v(\la v \ra^\ell)}^2.
\end{equation}
Therefore coming back to \eqref{eq:dt_hatf_soft} and choosing $\la R \ra = [(\lambda/\omega) (1+t) ]^{1/|\gamma+2s|}$ yields
$$
\begin{aligned}
\frac{1}{2} \frac{\d}{\d t} \Nt \widehat f(\xi) \Nt_{L^2_v}^2
&\le - \omega (1+t)^{-1}\Nt \widehat f(\xi) \Nt_{L^2_v}^2 - \sigma \left( \frac{1}{\eps^2} \| \P^\perp \widehat f (\xi) \|_{H^{s,*}_v}^2 + \| \P \widehat f (\xi) \|_{L^2_v}^2 \right) \\
&\quad
+ C (1+t)^{-1 - \frac{2 \ell}{|\gamma+2s|}} \| \widehat f (\xi) \|_{L^2_v(\la v \ra^\ell)}^2,
\end{aligned}
$$
for some constant $C>0$ (independent of $\xi$ and $\eps)$. Multiplying both sides by $(1+t)^{2\omega}$ gives
$$
\begin{aligned}
\frac{1}{2} \frac{\d}{\d t} \left\{ (1+t)^{2\omega}\Nt \widehat f(\xi) \Nt_{L^2_v}^2 \right\}
&\le  - \sigma (1+t)^{2\omega} \left( \frac{1}{\eps^2} \| \P^\perp \widehat f (\xi) \|_{H^{s,*}_v}^2 + \| \P \widehat f (\xi) \|_{L^2_v}^2 \right) \\
&\quad
+ C (1+t)^{2\omega-1 - \frac{2 \ell}{|\gamma+2s|}} \| \widehat f (\xi) \|_{L^2_v(\la v \ra^\ell)}^2.
\end{aligned}
$$
Integrating last estimate in time gives, for all $t \ge 0$,
\begin{multline*}
(1+t)^{2\omega} \| \widehat f (t,\xi) \|_{L^2_v}^2 
+ \frac{1}{\eps^2}\int_0^t (1+\tau)^{2\omega} \| \P^\perp \widehat f (\tau,\xi) \|_{H^{s,*}_v}^2 \, \d \tau
+ \int_0^t (1+\tau)^{2\omega} \| \P \widehat f (\tau,\xi) \|_{L^2_v}^2 \, \d \tau
\\ 
\lesssim \| \widehat f_0 (\xi) \|_{L^2_v}^2 
+ \sup_{\tau \in [0,t]} \| \widehat f (\tau,\xi) \|_{L^2_v (\la v \ra^\ell)}^2 \int_0^t (1+\tau)^{2\omega-1 - \frac{2 \ell}{|\gamma+2s|}} \, \d \tau,
\end{multline*}
where we have used again that $\Nt \cdot \Nt_{L^2_v}$ is equivalent to $\| \cdot \|_{L^2_v}$ independently of $\xi$ and $\eps$.
Observing that $(1+t)^{2\omega-1 - \frac{2 \ell}{|\gamma+2s|}} $ is integrable since $0 < \omega < \frac{\ell}{|\gamma+2s|}$, we can take the supremum in time in last estimate and then its square-root to obtain
$$
\| \mathrm{p}_\omega \widehat f (\xi) \|_{L^\infty_t L^2_v} 
+\frac{1}{\eps}  \| \mathrm{p}_\omega \P^\perp \widehat f (\xi) \|_{L^2_t H^{s,*}_v} 
+  \| \mathrm{p}_\omega \P \widehat f (s,\xi) \|_{L^2_t L^2_v }
\lesssim \| \widehat f_0 (\xi) \|_{L^2_v} + \| \widehat f (\xi) \|_{L^\infty_t L^2_v (\la v \ra^\ell)},
$$
and we conclude the proof by taking the $L^1_\xi$ norm.
\end{proof}

\begin{prop}\label{prop:estimate_Ueps_regularization_soft_torus}
Let $S=S(t,x,v)$ verify $\P S = 0$ and $\mathrm{p}_{\omega} \widehat S  \in L^1_\xi L^2_t (H^{s,*}_v)'$ for some $0< \omega < \frac{\ell}{|\gamma+2s|}$ and $\ell >0$, and denote
$$
g_S(t) = \int_0^t  U^\eps (t-\tau) S(\tau) \, \d \tau.
$$
Assume that $\widehat g_S \in L^1_\xi L^\infty_t L^2_v (\la v \ra^\ell)$, then we have
$$
\begin{aligned}
\left\| \mathrm{p}_{\omega} \widehat g_S \right\|_{L^1_\xi L^\infty_t L^2_v} 
&+\frac{1}{\eps}\| \mathrm{p}_{\omega} \P^\perp \widehat g_S\|_{L^1_\xi L^2_t H^{s,*}_v} 
+\|\mathrm{p}_{\omega} \P \widehat g_S  \|_{L^1_\xi L^2_t L^2_v} \\
&\qquad
\lesssim \eps \| \mathrm{p}_{\omega} \widehat S \|_{L^1_\xi L^2_t (H^{s,*}_v)'} 
+ \| \widehat g_S  \|_{L^1_\xi L^\infty_t L^2_v (\la v \ra^\ell)}.
\end{aligned}
$$

\end{prop}

\begin{proof}
Arguing as in the proof of Proposition~\ref{prop:estimate_Ueps_regularization_hard_torus}, but using now that  $\| \cdot \|_{H^{s,*}(\la v \ra^\ell)} \ge \| \la v \ra^{\gamma/2 + s} \cdot \|_{L^2_v(\la v \ra^\ell)}$ as in Proposition~\ref{prop:estimate_Ueps_decay_soft_torus}, we have
\begin{equation}\label{eq:dt_gS_soft}
\begin{aligned}
\frac{1}{2} \frac{\d}{\d t} \Nt \widehat g_S(\xi) \Nt_{L^2_v}^2
&\le - \lambda \| \la v \ra^{\gamma/2 + s}\widehat g_S (\xi) \|_{L^2_v}^2- \sigma  \left( \frac{1}{\eps^2} \| \P^\perp \widehat g_S (\xi) \|_{H^{s,*}_v}^2 + \| \P \widehat g_S (\xi) \|_{L^2_v}^2 \right) \\
&\quad
+ C \eps^2 \| \widehat S(\xi) \|_{(H^{s,*}_v)'}^2.
\end{aligned}
\end{equation}
for some constants $\lambda,\sigma,C>0$.
Using the interpolation~\eqref{eq:interpolation} as in the proof of Proposition~\ref{prop:estimate_Ueps_decay_soft_torus}, we obtain
$$
\begin{aligned}
\frac{1}{2} \frac{\d}{\d t} \Nt \widehat g_S(\xi) \Nt_{L^2_v}^2
&\le - \omega (1+t)^{-1}\Nt \widehat g_S(\xi) \Nt_{L^2_v}^2 - \sigma \left( \frac{1}{\eps^2} \| \P^\perp \widehat g_S (\xi) \|_{H^{s,*}_v}^2 + \| \P \widehat g_S (\xi) \|_{L^2_v}^2 \right) \\
&\quad
+  C \eps^2 \| \widehat S(\xi) \|_{(H^{s,*}_v)'}^2
+ C (1+t)^{-1 - \frac{2 \ell}{|\gamma+2s|}} \| \widehat g_S (\xi) \|_{L^2_v(\la v \ra^\ell)}^2,
\end{aligned}
$$
for some constant $C>0$ (independent of $\xi$ and $\eps)$. We can then conclude exactly as in the proof of Proposition~\ref{prop:estimate_Ueps_decay_soft_torus}.
\end{proof}

\subsection{Decay estimates: Hard potentials in the whole space}\label{sec:linear:HP_wholespace}

In this subsection we shall always assume $\gamma + 2s \ge 0$ and $\Omega_x = \R^3$, and we shall obtain decay estimates for the semigroup $U^\eps$ (see Proposition~\ref{prop:estimate_Ueps_decay_hard_wholespace}) as well as its integral in time against a source $\int_0^t U^\eps(t-\tau) S(\tau) \, \d \tau$ (see Proposition~\ref{prop:estimate_Ueps_regularization_hard_wholespace}). We recall that given any real number $\omega \in \R$ we denote $\mathrm{p}_\omega : t \mapsto (1+t)^{\omega}$.

\begin{prop}\label{prop:estimate_Ueps_decay_hard_wholespace}
Let $\ell \ge 0$, $p \in (3/2,\infty]$ and $0<\vartheta< \frac{3}{2}(1-\frac{1}{p})$.
Let $\widehat f_0 \in L^1_\xi L^2_v (\la v \ra^{\ell})$, then
\begin{multline*}
\| \mathrm{p}_\vartheta \widehat U^\eps (\cdot) \widehat f_0 \|_{L^1_\xi L^\infty_t L^2_v (\la v \ra^{\ell})} 
+\frac{1}{\eps} \|\mathrm{p}_\vartheta \P^\perp \widehat U^\eps (\cdot) \widehat f_0 \|_{L^1_\xi L^2_t H^{s,*}_v (\la v \ra^{\ell})}
+ \left\| \mathrm{p}_\vartheta \frac{|\xi|}{\la \xi \ra} \P \widehat U^\eps (\cdot) \widehat f_0 \right\|_{L^1_\xi L^2_t L^2_v (\la v \ra^{\ell})} \\
\lesssim \| \widehat f_0 \|_{L^1_\xi L^2_v (\la v \ra^{\ell})} +  \| \widehat U^\eps (\cdot) \widehat f_0 \|_{L^p_\xi L^\infty_t L^2_v (\la v \ra^{\ell})} .
\end{multline*}
\end{prop}

\begin{proof}
Let $f(t) = U^\eps(t) f_0$ for all $t \ge 0$ which satisfies \eqref{eq:Ueps_f0}, so that $\widehat f(t,\xi) = \widehat U^\eps(t,\xi) \widehat f(\xi)$ satisfies \eqref{eq:hatUeps_hatf0} for all $\xi \in \R^3$. 
Using Proposition~\ref{prop:hypocoercivity} we have, for all $t \ge 0$ and some $\lambda_0>0$,
$$
\begin{aligned}
\frac{1}{2} \frac{\d}{\d t} \Nt \widehat f(\xi) \Nt_{L^2_v (\la v \ra^{\ell})}^2
&\le - \lambda_0 \left( \frac{1}{\eps^2} \| \P^\perp \widehat f (\xi) \|_{H^{s,*}_v(\la v \ra^{\ell})}^2 + \frac{|\xi|^2}{\la \xi \ra^2}\| \P \widehat f (\xi) \|_{L^2_v}^2 \right),
\end{aligned}
$$
and we already observe that, using  $\| \cdot \|_{H^{s,*}_v (\la v \ra^{\ell})} \ge \| \la v \ra^{\gamma/2+s} \cdot \|_{L^2_v (\la v \ra^{\ell})} \ge  \| \cdot \|_{L^2_v (\la v \ra^{\ell})}$ and $\eps \in (0,1]$,
$$
\frac{1}{\eps^2} \| \P^\perp \widehat f (\xi) \|_{H^{s,*}_v(\la v \ra^{\ell})}^2 + \frac{|\xi|^2}{\la \xi \ra^2}\| \P \widehat f (\xi) \|_{L^2_v}^2 \gtrsim  \frac{|\xi|^2}{\la \xi \ra^2} \Nt \widehat f (\xi) \Nt_{L^2_v (\la v \ra^{\ell})}^2,
$$
where we have used that  $\Nt \widehat f (\xi) \Nt_{L^2_v(\la v \ra^{\ell})}$ is equivalent to $\| \widehat f(\xi) \|_{L^2_v(\la v \ra^{\ell})}$ independently of $\xi$ and $\eps$. Therefore it follows
\begin{equation}\label{eq:dt_hatf_hard_wholespace}
\begin{aligned}
\frac{1}{2} \frac{\d}{\d t} \Nt \widehat f(\xi) \Nt_{L^2_v (\la v \ra^{\ell})}^2
&\le - 2\lambda \frac{|\xi|^2}{\la \xi \ra^2} \Nt \widehat f (\xi) \Nt_{L^2_v (\la v \ra^{\ell})}^2 - \sigma \left( \frac{1}{\eps^2} \| \P^\perp \widehat f (\xi) \|_{H^{s,*}_v(\la v \ra^{\ell})}^2 + \frac{|\xi|^2}{\la \xi \ra^2}\| \P \widehat f (\xi) \|_{L^2_v}^2 \right),
\end{aligned}
\end{equation}
for some constants $\lambda,\sigma >0$.
We now split our analysis into two cases: high frequencies $|\xi| \ge 1$ and low frequencies $|\xi|<1$.

For high frequencies $|\xi| \ge 1$ we remark that $\frac{|\xi|^2}{\la \xi \ra^2} \ge \frac{1}{2}$, hence we obtain 
$$
\begin{aligned}
\frac{1}{2} \frac{\d}{\d t} \mathbf{1}_{|\xi| \ge 1} \Nt \widehat f(\xi) \Nt_{L^2_v (\la v \ra^{\ell})}^2
&\le - \lambda \mathbf{1}_{|\xi| \ge 1} \Nt \widehat f(\xi) \Nt_{L^2_v (\la v \ra^{\ell})}^2 \\
&\quad
- \frac{\sigma}{2} \left( \frac{1}{\eps^2} \mathbf{1}_{|\xi| \ge 1}\| \P^\perp \widehat f (\xi) \|_{H^{s,*}_v(\la v \ra^{\ell})}^2 + \mathbf{1}_{|\xi| \ge 1}\| \P \widehat f (\xi) \|_{L^2_v}^2 \right).
\end{aligned}
$$
Arguing as in the proof of Proposition~\ref{prop:estimate_Ueps_decay_hard_torus} we hence deduce
\begin{equation}\label{eq:high_freq}
\begin{aligned}
\mathbf{1}_{|\xi| \ge 1}\| \mathrm{e}_\lambda\widehat f(\xi) \|_{L^\infty_t L^2_v(\la v \ra^{\ell})}
+\frac{1}{\eps} \mathbf{1}_{|\xi| \ge 1}\| \mathrm{e}_\lambda \P^\perp \widehat f(\xi) \|_{L^2_t H^{s,*}_v(\la v \ra^{\ell})}
+ \mathbf{1}_{|\xi| \ge 1} \| \mathrm{e}_\lambda \P \widehat f(\xi) \|_{L^2_t L^2_v} \\
\lesssim \mathbf{1}_{|\xi| \ge 1} \| \widehat f_0 (\xi) \|_{L^2_v(\la v \ra^{\ell})}.
\end{aligned}
\end{equation}

We now investigate the case of low frequencies $|\xi|<1$. We denote by $p'$ the conjugate exponent of $p$, that is $1/p+1/p' = 1$ with the convention $p'=1$ if $p=\infty$, and consider a real number $r$ verifying $1+p'/3 < r < 1+ 1/(2 \vartheta)$, which we observe is possible thanks to the conditions on $p$ and $\vartheta$.
Remarking that $|\xi|^2 \le 2 |\xi|^2/\la \xi \ra^2$ if $|\xi|<1$, by Young's inequality we get: for any $\delta>0$ there is $C_\delta>0$ such that, for all $|\xi|<1$ and $t \ge 0$, we have
\begin{equation}\label{eq:Young_wholespace}
1 \le \delta (1+t) \frac{|\xi|^2}{\la \xi \ra^2} + C_\delta (1+t)^{-\frac{1}{r-1}} |\xi|^{-\frac{2}{r-1}}.
\end{equation}
We therefore obtain, coming back to~\eqref{eq:dt_hatf_hard_wholespace} and choosing $\delta>0$ appropriately,
$$
\begin{aligned}
\frac{1}{2} \frac{\d}{\d t}  \mathbf{1}_{|\xi| < 1} \Nt \widehat f(\xi) \Nt_{L^2_v (\la v \ra^\ell)}^2
&\le - \sigma \left( \frac{1}{\eps^2} \mathbf{1}_{|\xi| < 1}\| \P^\perp \widehat f (\xi) \|_{H^{s,*}_v(\la v \ra^\ell)}^2 + \mathbf{1}_{|\xi| < 1} \frac{|\xi|^2}{\la \xi \ra^2}\| \P \widehat f (\xi) \|_{L^2_v}^2 \right) \\
&\quad
- \vartheta (1+t)^{-1} \mathbf{1}_{|\xi| < 1}\Nt \widehat f (\xi) \Nt_{L^2_v(\la v \ra^\ell)}^2\\
&\quad
+ C (1+t)^{-1-\frac{1}{r-1}} |\xi|^{-\frac{2}{r-1}} \mathbf{1}_{|\xi| < 1} \| \widehat f (\xi) \|_{L^2_v(\la v \ra^\ell)}^2.
\end{aligned}
$$
for some constant $C>0$. Multiplying both sides by $(1+t)^{2\vartheta}$ gives
$$
\begin{aligned}
&\frac{1}{2} \frac{\d}{\d t} \left\{ (1+t)^{2\vartheta} \mathbf{1}_{|\xi| < 1} \Nt \widehat f(\xi) \Nt_{L^2_v(\la v \ra^\ell) }^2 \right\} \\
&\qquad
\le - \sigma  (1+t)^{2\vartheta} \left( \frac{1}{\eps^2} \mathbf{1}_{|\xi| < 1}\| \P^\perp \widehat f (\xi) \|_{H^{s,*}_v(\la v \ra^\ell)}^2 + \mathbf{1}_{|\xi| < 1} \frac{|\xi|^2}{\la \xi \ra^2}\| \P \widehat f (\xi) \|_{L^2_v}^2 \right) \\
&\qquad\quad
+ C (1+t)^{2\vartheta-1-\frac{1}{r-1}} |\xi|^{-\frac{2}{r-1}} \mathbf{1}_{|\xi| < 1} \| \widehat f (\xi) \|_{L^2_v(\la v \ra^\ell)}^2.
\end{aligned}
$$
Integrating in time implies, for all $t\ge 0$,
$$
\begin{aligned}
(1+t)^{2\vartheta} \mathbf{1}_{|\xi| < 1} \| \widehat f(t,\xi) \|_{L^2_v(\la v \ra^\ell) }^2
&+ \frac{1}{\eps^2} \int_0^t (1+\tau)^{2 \vartheta} \mathbf{1}_{|\xi| < 1} \| \P^\perp \widehat f (\tau,\xi) \|_{H^{s,*}_v(\la v \ra^\ell)}^2 \, \d \tau \\
&+ \int_0^t (1+\tau)^{2 \vartheta} \mathbf{1}_{|\xi| < 1} \frac{|\xi|^2}{\la \xi \ra^2}\| \P \widehat f (\tau,\xi) \|_{L^2_v}^2 \, \d \tau \\
&\lesssim \mathbf{1}_{|\xi| < 1} \| \widehat f_0 (\xi) \|_{L^2_v(\la v \ra^\ell)}^2 + \mathbf{1}_{|\xi| < 1} |\xi|^{-\frac{2}{r-1}}  \| \widehat f(\xi) \|_{L^\infty_t L^2_v (\la v \ra^\ell)}^2,
\end{aligned}
$$
where we have used that $(1+t)^{2\vartheta-1-\frac{1}{r-1}}$ is integrable since $r<1 + 1/(2 \vartheta)$.
We now take the supremum in time and finally the square-root of the resulting estimate, which gives
\begin{equation}\label{eq:low_freq}
\begin{aligned}
\mathbf{1}_{|\xi| < 1}\| \mathrm{p}_\vartheta \widehat f(\xi) \|_{L^\infty_t L^2_v(\la v \ra^\ell)}
+\frac{1}{\eps} \mathbf{1}_{|\xi| < 1}\| \mathrm{p}_\vartheta \P^\perp \widehat f(\xi) \|_{L^2_t H^{s,*}_v(\la v \ra^\ell)}
+ \mathbf{1}_{|\xi| < 1} \left\| \mathrm{p}_\vartheta \frac{|\xi|}{\la \xi \ra}  \P \widehat f(\xi) \right\|_{L^2_t L^2_v} \\
\lesssim \mathbf{1}_{|\xi| < 1} \| \widehat f_0 (\xi) \|_{L^2_v(\la v \ra^\ell)} + \mathbf{1}_{|\xi| < 1} |\xi|^{-\frac{1}{r-1}}   \| \widehat f(\xi) \|_{L^\infty_t L^2_v (\la v \ra^\ell)} .
\end{aligned}
\end{equation}

Gathering the estimate for high frequencies~\eqref{eq:high_freq} together with the one for low frequencies~\eqref{eq:low_freq}, it follows
$$
\begin{aligned}
\| \mathrm{p}_\vartheta \widehat f(\xi) \|_{L^\infty_t L^2_v(\la v \ra^\ell)}
+\frac{1}{\eps} \| \mathrm{p}_\vartheta \P^\perp \widehat f(\xi) \|_{L^2_t H^{s,*}_v(\la v \ra^\ell)}
+ \left\| \mathrm{p}_\vartheta \frac{|\xi|}{\la \xi \ra}  \P \widehat f(\xi) \right\|_{L^2_t L^2_v}\\
\lesssim  \| \widehat f_0 (\xi) \|_{L^2_v(\la v \ra^\ell)} + \mathbf{1}_{|\xi| < 1} |\xi|^{-\frac{1}{r-1}}   \| \widehat f(\xi) \|_{L^\infty_t L^2_v (\la v \ra^\ell)} .
\end{aligned}
$$
Taking the $L^1_\xi$ norm above, we use H\"older's inequality to obtain
$$
\begin{aligned}
\int_{\R^3} \mathbf{1}_{|\xi| < 1} |\xi|^{-\frac{1}{r-1}}   \| \widehat f(\xi) \|_{L^\infty_t L^2_v(\la v \ra^\ell) } \, \d \xi 
&\lesssim \left( \int_{\R^3} \mathbf{1}_{|\xi| < 1} |\xi|^{-\frac{p'}{r-1}} \, \d \xi \right)^{1/p'} \| \widehat f \|_{L^p_\xi L^\infty_t L^2_v(\la v \ra^\ell) } \\
&\lesssim \| \widehat f \|_{L^p_\xi L^\infty_t L^2_v(\la v \ra^\ell) },
\end{aligned}
$$
since $r > 1+p'/3$, which implies 
$$
\begin{aligned}
\| \mathrm{p}_\vartheta \widehat f \|_{L^1_\xi L^\infty_t L^2_v(\la v \ra^\ell)}
+\frac{1}{\eps} \| \mathrm{p}_\vartheta \P^\perp \widehat f \|_{L^1_\xi L^2_t H^{s,*}_v(\la v \ra^\ell)}
&+ \left\| \mathrm{p}_\vartheta \frac{|\xi|}{\la \xi \ra}  \P \widehat f \right\|_{L^1_\xi L^2_t L^2_v} \\
&\lesssim  \| \widehat f_0  \|_{L^1_\xi L^2_v(\la v \ra^\ell)} +  \| \widehat f \|_{L^p_\xi L^\infty_t L^2_v (\la v \ra^\ell)} ,
\end{aligned}
$$
and concludes the proof.
\end{proof}

\begin{prop}\label{prop:estimate_Ueps_regularization_hard_wholespace}
Let $\ell \ge 0$, $p \in (3/2,\infty]$ and $0<\vartheta< \frac{3}{2}(1-\frac{1}{p})$.
Let $S=S(t,x,v)$ verify $\P S = 0$ and $\mathrm{p}_{\vartheta}  \la v \ra^{\ell} \widehat S \in L^1_\xi L^2_t (H^{s,*}_v)'$, and denote
$$
g_S(t) = \int_0^t  U^\eps (t-\tau) S(\tau) \, \d \tau.
$$
Assume that $\widehat g_S \in L^p_\xi L^\infty_t L^2_v (\la v \ra^\ell)$, then
$$
\begin{aligned}
\left\| \mathrm{p}_{\vartheta} \widehat g_S \right\|_{L^1_\xi L^\infty_t L^2_v(\la v \ra^{\ell})} 
&+\frac{1}{\eps}\| \mathrm{p}_{\vartheta} \P^\perp \widehat g_S \|_{L^1_\xi L^2_t H^{s,*}_v(\la v \ra^{\ell})} 
+\left\| \mathrm{p}_{\vartheta} \frac{|\xi|}{\la \xi \ra} \P \widehat g_S  \right\|_{L^1_\xi L^2_t L^2_v} \\
&\lesssim \eps \| \mathrm{p}_{\vartheta} \la v \ra^{\ell} \widehat S \|_{L^1_\xi L^2_t (H^{s,*}_v)'} 
+ \left\| \widehat g_S \right\|_{L^p_\xi L^\infty_t L^2_v(\la v \ra^{\ell})} .
\end{aligned}
$$
\end{prop}

\begin{proof}
Recalling that $\widehat g_S$ satisfies \eqref{eq:hatgS}, we can argue as for obtaining \eqref{eq:dt_hatf_hard_wholespace} to get
$$
\begin{aligned}
\frac{1}{2} \frac{\d}{\d t} \Nt \widehat g_S(\xi) \Nt_{L^2_v(\la v \ra^{\ell})}^2
&\le -2 \lambda \frac{|\xi|^2}{\la \xi \ra^2} \Nt \widehat g_S(\xi) \Nt_{L^2_v(\la v \ra^{\ell})}^2 - \sigma \left( \frac{1}{\eps^2} \| \P^\perp \widehat g_S (\xi) \|_{H^{s,*}_v(\la v \ra^{\ell})}^2 + \frac{|\xi|^2}{\la \xi \ra^2} \| \P \widehat g_S (\xi) \|_{L^2_v}^2 \right) \\
&\quad
+ C \eps^2 \| \la v \ra^{\ell} \widehat S(\xi) \|_{(H^{s,*}_v)'}^2 ,
\end{aligned}
$$
for some constants $\lambda,\sigma,C>0$. By separating the cases of high and low frequencies, we can conclude exactly as in the proof of Proposition~\ref{prop:estimate_Ueps_decay_hard_wholespace}.
\end{proof}

\subsection{Decay estimates: Soft potentials in the whole space}\label{sec:linear:SP_wholespace}

In this subsection we shall always assume $\gamma + 2s < 0$ and $\Omega_x = \R^3$, and we shall obtain decay estimates for the semigroup $U^\eps$ (see Proposition~\ref{prop:estimate_Ueps_decay_soft_wholespace}) as well as its integral in time against a source $\int_0^t U^\eps(t-\tau) S(\tau) \, \d \tau$ (see Proposition~\ref{prop:estimate_Ueps_regularization_soft_wholespace}). We recall that given any real number $\omega \in \R$ we denote $\mathrm{p}_\omega : t \mapsto (1+t)^{\omega}$.

\begin{prop}\label{prop:estimate_Ueps_decay_soft_wholespace}
Let $p \in (3/2,\infty]$ and $0<\vartheta< \frac{3}{2}(1-\frac{1}{p})$. Let $f_0 \in \FF_x^{-1} ( L^1_\xi L^2_v (\la v \ra^\ell) \cap L^p_\xi L^2_v)$ with $\ell>  \vartheta |\gamma+2s|$, then we have
$$
\begin{aligned}
\| \mathrm{p}_\vartheta \widehat U^\eps (\cdot) \widehat f_0 \|_{L^1_\xi L^\infty_t L^2_v} 
+\frac{1}{\eps} \| \mathrm{p}_\vartheta \P^\perp \widehat U^\eps (\cdot) \widehat f_0 \|_{L^1_\xi L^2_t H^{s,*}_v}
+ \left\| \mathrm{p}_\vartheta \frac{|\xi|}{\la \xi \ra} \P (\widehat U^\eps (\cdot) \widehat f_0) \right\|_{L^1_\xi L^2_t L^2_v} \\
\lesssim \| \widehat f_0 \|_{L^1_\xi L^2_v} 
+\| \widehat U^\eps (\cdot) \widehat f_0 \|_{L^1_\xi L^\infty_t L^2_v(\la v \ra^\ell)} 
+ \| \widehat U^\eps (\cdot) \widehat f_0 \|_{L^p_\xi L^\infty_t L^2_v} .
\end{aligned}
$$
\end{prop}

\begin{proof}
Arguing as in the proof of Proposition~\ref{prop:estimate_Ueps_decay_hard_wholespace}, denoting $f(t) = U^\eps(t) f_0$ and using that $\| \cdot \|_{H^{s,*}(\la v \ra^\ell)} \ge \| \la v \ra^{\gamma/2 + s} \cdot \|_{L^2_v(\la v \ra^\ell)}$, we first obtain 
\begin{equation}\label{eq:dt_hatf_soft_bis}
\begin{aligned}
\frac{1}{2} \frac{\d}{\d t} \Nt \widehat f(\xi) \Nt_{L^2_v}^2
&\le - \lambda \| \la v \ra^{\gamma/2 +s} \P^\perp \widehat f (\xi) \|_{L^2_v}^2 - \lambda \frac{|\xi|^2}{\la \xi \ra^2} \| \P \widehat f (\xi) \|_{L^2_v}^2 \\
&\quad
- \sigma  \left( \frac{1}{\eps^2} \| \P^\perp \widehat f (\xi) \|_{H^{s,*}_v}^2 + \frac{|\xi|^2}{\la \xi \ra^2}\| \P \widehat f (\xi) \|_{L^2_v}^2 \right),
\end{aligned}
\end{equation}
for some positive constants $\lambda , \sigma >0$. We now split the analysis into high frequencies and low frequencies.

For high frequencies $|\xi| \ge 1$ we observe that $\frac{|\xi|^2}{\la \xi \ra^2} \ge \frac{1}{2}$, which yields 
$$
\begin{aligned}
\frac{1}{2} \frac{\d}{\d t} \mathbf{1}_{|\xi| \ge 1}\Nt \widehat f(\xi) \Nt_{L^2_v}^2
&\le - \lambda \mathbf{1}_{|\xi| \ge 1}\| \la v \ra^{\gamma/2 +s} \P^\perp \widehat f (\xi) \|_{L^2_v}^2 - \lambda  \| \P \widehat f (\xi) \|_{L^2_v}^2 \\
&\quad
- \sigma  \left( \frac{1}{\eps^2} \mathbf{1}_{|\xi| \ge 1} \| \P^\perp \widehat f (\xi) \|_{H^{s,*}_v}^2 +  \mathbf{1}_{|\xi| \ge 1}\| \P \widehat f (\xi) \|_{L^2_v}^2 \right) \\
&\le - \lambda \mathbf{1}_{|\xi| \ge 1}\| \la v \ra^{\gamma/2 +s}  \widehat f (\xi) \|_{L^2_v}^2  \\
&\quad
- \sigma  \left( \frac{1}{\eps^2} \mathbf{1}_{|\xi| \ge 1} \| \P^\perp \widehat f (\xi) \|_{H^{s,*}_v}^2 + \mathbf{1}_{|\xi| \ge 1}\| \P \widehat f (\xi) \|_{L^2_v}^2 \right),
\end{aligned}
$$
for some other constants $\lambda , \sigma >0$. Thanks to the interpolation inequality~\eqref{eq:interpolation} of the proof of Proposition~\ref{prop:estimate_Ueps_decay_soft_torus}, we hence deduce
$$
\begin{aligned}
\frac{1}{2} \frac{\d}{\d t}  \mathbf{1}_{|\xi| \ge 1}\Nt \widehat f(\xi) \Nt_{L^2_v }^2
&\le - \omega (1+t)^{-1}  \mathbf{1}_{|\xi| \ge 1} \Nt \widehat f(\xi) \Nt_{L^2_v}^2  \\
&\quad
- \sigma  \left( \frac{1}{\eps^2} \mathbf{1}_{|\xi| \ge 1} \| \P^\perp \widehat f (\xi) \|_{H^{s,*}_v }^2 +  \mathbf{1}_{|\xi| \ge 1} \| \P \widehat f (\xi) \|_{L^2_v}^2 \right) \\
&\quad
+ C (1+t)^{-1 - \frac{2 \ell}{|\gamma+2s|}} \| \widehat f(\xi) \|_{L^2_v (\la v \ra^{\ell})}^2,
\end{aligned}
$$
for any $ \vartheta < \omega < \frac{\ell}{|\gamma+2s|}$ and some constant $C>0$. With this inequality we can thus argue as in the proof of Proposition~\ref{prop:estimate_Ueps_decay_soft_torus}, which gives, recalling that $(1+t)^{2\omega-1 - \frac{2 \ell}{|\gamma+2s|}}$ is integrable since $0 < \omega < \frac{\ell}{|\gamma+2s|}$, 
\begin{equation}\label{eq:high_freq_soft}
\begin{aligned}
\mathbf{1}_{|\xi| \ge 1}\| \mathrm{p}_\omega \widehat f(\xi) \|_{L^\infty_t L^2_v }
+\frac{1}{\eps} \mathbf{1}_{|\xi| \ge 1}\| \mathrm{p}_\omega \P^\perp \widehat f(\xi) \|_{L^2_t H^{s,*}_v}
+ \mathbf{1}_{|\xi| \ge 1} \| \mathrm{p}_\omega \P \widehat f(\xi) \|_{L^2_t L^2_v} \\
\lesssim \mathbf{1}_{|\xi| \ge 1} \| \widehat f_0 (\xi) \|_{L^2_v}
+ \mathbf{1}_{|\xi| \ge 1} \| \widehat f (\xi) \|_{L^\infty_t L^2_v(\la v \ra^{\ell})}.
\end{aligned}
\end{equation}

We now turn our attention to the low frequencies case $|\xi|<1$ . First of all, from \eqref{eq:dt_hatf_soft_bis}, we use the interpolation inequality~\eqref{eq:interpolation} of the proof of Proposition~\ref{prop:estimate_Ueps_decay_soft_torus} to deduce
$$
\begin{aligned}
\frac{1}{2} \frac{\d}{\d t} \mathbf{1}_{|\xi| < 1}\Nt \widehat f(\xi) \Nt_{L^2_v}^2
&\le -\omega (1+t)^{-1}  \mathbf{1}_{|\xi| < 1} \Nt \P^\perp \widehat f(\xi) \Nt_{L^2_v}^2 - \lambda \mathbf{1}_{|\xi| < 1} \frac{|\xi|^2}{\la \xi \ra^2} \| \P \widehat f (\xi) \|_{L^2_v}^2 \\
&\quad
- \sigma  \left( \frac{1}{\eps^2} \mathbf{1}_{|\xi| < 1} \| \P^\perp \widehat f (\xi) \|_{H^{s,*}_v}^2 + \mathbf{1}_{|\xi| < 1} \frac{|\xi|^2}{\la \xi \ra^2}\| \P \widehat f (\xi) \|_{L^2_v}^2 \right) \\
&\quad+ C (1+t)^{-1 - \frac{2 \ell}{|\gamma+2s|}} \mathbf{1}_{|\xi| < 1}\| \P^\perp \widehat f(\xi) \|_{L^2_v (\la v \ra^{\ell})}^2,
\end{aligned}
$$
for any $ \vartheta < \omega < \frac{\ell}{|\gamma+2s|}$ and some constant $C>0$.
As in the proof of Proposition~\ref{prop:estimate_Ueps_decay_hard_wholespace}, we denote by $p'$ the conjugate exponent of $p$, and consider a real number $r$ verifying $1+p'/3 < r < 1+ 1/(2 \vartheta)$. Using inequality~\eqref{eq:Young_wholespace} we hence deduce
$$
\begin{aligned}
\frac{1}{2} \frac{\d}{\d t}  \mathbf{1}_{|\xi| < 1} \Nt \widehat f(\xi) \Nt_{L^2_v}^2
&\le 
- \vartheta (1+t)^{-1} \mathbf{1}_{|\xi| < 1}\Nt \widehat f (\xi) \Nt_{L^2_v}^2\\
&\quad
- \sigma \left( \frac{1}{\eps^2} \mathbf{1}_{|\xi| < 1}\| \P^\perp \widehat f (\xi) \|_{H^{s,*}_v}^2 + \mathbf{1}_{|\xi| < 1} \frac{|\xi|^2}{\la \xi \ra^2}\| \P \widehat f (\xi) \|_{L^2_v}^2 \right) \\
&\quad
+ C (1+t)^{-1-\frac{2\ell}{|\gamma+2s|}} \mathbf{1}_{|\xi| < 1}\| \P^\perp \widehat f (\xi) \|_{L^2_v (\la v \ra^\ell)}^2 \\
&\quad
+ C (1+t)^{-1-\frac{1}{r-1}} |\xi|^{-\frac{2}{r-1}} \mathbf{1}_{|\xi| < 1} \| \P \widehat f (\xi) \|_{L^2_v}^2,
\end{aligned}
$$
for some constant $C>0$. Multiplying both sides by $(1+t)^{2\vartheta}$ gives
$$
\begin{aligned}
\frac{1}{2} \frac{\d}{\d t}  \left\{ (1+t)^{2\vartheta} \mathbf{1}_{|\xi| < 1} \Nt \widehat f(\xi) \Nt_{L^2_v}^2 \right\}
&\le 
- \sigma (1+t)^{2\vartheta} \left( \frac{1}{\eps^2} \mathbf{1}_{|\xi| < 1}\| \P^\perp \widehat f (\xi) \|_{H^{s,*}_v}^2 + \mathbf{1}_{|\xi| < 1} \frac{|\xi|^2}{\la \xi \ra^2}\| \P \widehat f (\xi) \|_{L^2_v}^2 \right) \\
&\quad
+ C (1+t)^{2 \vartheta-1-\frac{2\ell}{|\gamma+2s|}} \mathbf{1}_{|\xi| < 1} \| \P^\perp \widehat f (\xi) \|_{L^2_v (\la v \ra^\ell)}^2 \\
&\quad
+ C (1+t)^{2 \vartheta-1-\frac{1}{r-1}} |\xi|^{-\frac{2}{r-1}} \mathbf{1}_{|\xi| < 1} \|\P \widehat f (\xi) \|_{L^2_v}^2.
\end{aligned}
$$
Integrating in time implies, for all $t\ge 0$,
$$
\begin{aligned}
(1+t)^{2\vartheta} \mathbf{1}_{|\xi| < 1} \| \widehat f(t,\xi) \|_{L^2_v}^2
&+ \frac{1}{\eps^2} \int_0^t (1+\tau)^{2 \vartheta} \mathbf{1}_{|\xi| < 1} \| \P^\perp \widehat f (\tau,\xi) \|_{H^{s,*}_v}^2 \, \d \tau \\
&+ \int_0^t (1+\tau)^{2 \vartheta} \mathbf{1}_{|\xi| < 1} \frac{|\xi|^2}{\la \xi \ra^2}\| \P \widehat f (\tau,\xi) \|_{L^2_v}^2 \, \d \tau \\
&\lesssim \mathbf{1}_{|\xi| < 1} \| \widehat f_0 (\xi) \|_{L^2_v}^2 
+ \mathbf{1}_{|\xi| < 1}\| \widehat f (\xi) \|_{L^\infty_t L^2_v (\la v \ra^\ell)}^2
+ \mathbf{1}_{|\xi| < 1} |\xi|^{-\frac{2}{r-1}}  \| \widehat f(\xi) \|_{L^\infty_t L^2_v }^2,
\end{aligned}
$$
where we have used that $(1+t)^{2 \vartheta-1-\frac{2\ell}{|\gamma+2s|}}$ and $(1+t)^{2\vartheta-1-\frac{1}{r-1}}$ are integrable since $0< \vartheta < \omega < \frac{\ell}{|\gamma + 2s|}$ and $r<1 + 1/(2 \vartheta)$, respectively.
We can now take the supremum in time and then the square-root of the resulting estimate, which gives
\begin{equation}\label{eq:low_freq_soft}
\begin{aligned}
\mathbf{1}_{|\xi| < 1}\| \mathrm{p}_\vartheta \widehat f(\xi) \|_{L^\infty_t L^2_v}
+\frac{1}{\eps} \mathbf{1}_{|\xi| < 1}\| \mathrm{p}_\vartheta \P^\perp \widehat f(\xi) \|_{L^2_t H^{s,*}_v}
+ \mathbf{1}_{|\xi| < 1} \left\| \mathrm{p}_\vartheta \frac{|\xi|}{\la \xi \ra}  \P \widehat f(\xi) \right\|_{L^2_t L^2_v} \\
\lesssim \mathbf{1}_{|\xi| < 1} \| \widehat f_0 (\xi) \|_{L^2_v} 
+ \mathbf{1}_{|\xi| < 1} \| \widehat f (\xi) \|_{L^\infty_t L^2_v(\la v \ra^{\ell})} 
+ \mathbf{1}_{|\xi| < 1} |\xi|^{-\frac{1}{r-1}}   \| \widehat f(\xi) \|_{L^\infty_t L^2_v} .
\end{aligned}
\end{equation}

Gathering the estimate for high frequencies~\eqref{eq:high_freq_soft} together with the one for low frequencies~\eqref{eq:low_freq_soft} and observing that $\vartheta < \omega$, it follows
$$
\begin{aligned}
\| \mathrm{p}_\vartheta \widehat f(\xi) \|_{L^\infty_t L^2_v(\la v \ra^{\ell})}
+\frac{1}{\eps} \| \mathrm{p}_\vartheta \P^\perp \widehat f(\xi) \|_{L^2_t H^{s,*}_v(\la v \ra^{\ell})}
+ \left\| \mathrm{p}_\vartheta \frac{|\xi|}{\la \xi \ra}  \P \widehat f(\xi) \right\|_{L^2_t L^2_v}\\
\lesssim  \| \widehat f_0 (\xi) \|_{L^2_v} 
+ \| \widehat f (\xi) \|_{L^\infty_t L^2_v(\la v \ra^{\ell})} 
+ \mathbf{1}_{|\xi| < 1} |\xi|^{-\frac{1}{r-1}}   \| \widehat f(\xi) \|_{L^\infty_t L^2_v} .
\end{aligned}
$$
Taking the $L^1_\xi$ norm above, we use H\"older's inequality to control the last term in the right-hand side as in the proof of Proposition~\ref{prop:estimate_Ueps_decay_hard_wholespace}, to obtain
$$
\begin{aligned}
\int_{\R^3} \mathbf{1}_{|\xi| < 1} |\xi|^{-\frac{1}{r-1}}   \| \widehat f(\xi) \|_{L^\infty_t L^2_v (\la v \ra^{\ell})} \, \d \xi 
&\lesssim \| \widehat f \|_{L^p_\xi L^\infty_t L^2_v (\la v \ra^{\ell})},
\end{aligned}
$$
since $r > 1+p'/3$, which implies 
$$
\begin{aligned}
\| \mathrm{p}_\vartheta \widehat f \|_{L^1_\xi L^\infty_t L^2_v(\la v \ra^{\ell})}
+\frac{1}{\eps} \| \mathrm{p}_\vartheta \P^\perp \widehat f \|_{L^1_\xi L^2_t H^{s,*}_v(\la v \ra^{\ell})}
+ \left\| \mathrm{p}_\vartheta \frac{|\xi|}{\la \xi \ra}  \P \widehat f \right\|_{L^1_\xi L^2_t L^2_v}\\
\lesssim  \| \widehat f_0  \|_{L^1_\xi L^2_v} 
+ \| \widehat f \|_{L^1_\xi L^\infty_t L^2_v(\la v \ra^{\ell})}
+  \| \widehat f \|_{L^p_\xi L^\infty_t L^2_v (\la v \ra^{\ell})} 
\end{aligned}
$$
and concludes the proof.
\end{proof}

\begin{prop}\label{prop:estimate_Ueps_regularization_soft_wholespace}
Let $p \in (3/2,\infty]$ and $0<\vartheta< \frac{3}{2}(1-\frac{1}{p})$. Let $S=S(t,x,v)$ verify $\P S = 0$ and $ \mathrm{p}_{\vartheta}\widehat S \in L^1_\xi L^2_t (H^{s,*}_v )'$, and denote
$$
g_S(t) = \int_0^t  U^\eps (t-\tau) S(\tau) \, \d \tau.
$$
Assume that $g_S \in \FF_x^{-1} ( L^1_\xi L^2_v (\la v \ra^\ell) \cap L^p_\xi L^2_v)$ with $\ell> \vartheta |\gamma+2s|$, then
$$
\begin{aligned}
\left\| \mathrm{p}_{\vartheta} \widehat g_S \right\|_{L^1_\xi L^\infty_t L^2_v} 
&+\frac{1}{\eps}\| \mathrm{p}_{\vartheta} \P^\perp \widehat g_S \|_{L^1_\xi L^2_t H^{s,*}_v} 
+\left\| \mathrm{p}_{\vartheta} \frac{|\xi|}{\la \xi \ra} \P \widehat g_S  \right\|_{L^1_\xi L^2_t L^2_v} \\
&\lesssim \eps \| \mathrm{p}_{\vartheta} \widehat S \|_{L^1_\xi L^2_t (H^{s,*}_v)'} 
+ \| \widehat g_S \|_{L^1_\xi L^\infty_t L^2_v(\la v \ra^{\ell})} 
+ \| \widehat g_S \|_{L^p_\xi L^\infty_t L^2_v} .
\end{aligned}
$$
\end{prop}

\begin{proof}
Recalling that $\widehat g_S$ satisfies \eqref{eq:hatgS}, we can argue as for obtaining \eqref{eq:dt_hatf_soft_bis} to get
$$
\begin{aligned}
\frac{1}{2} \frac{\d}{\d t} \Nt \widehat g_S(\xi) \Nt_{L^2_v}^2
&\le -\lambda \| \langle v \rangle^{\gamma/2+s} \P^\perp \widehat g_S(\xi) \|_{L^2_v}^2
-\lambda \frac{|\xi|^2}{\la \xi \ra^2} \| \P \widehat g_S(\xi) \|_{L^2_v}^2
\\
&\quad 
- \sigma \left( \frac{1}{\eps^2} \| \P^\perp \widehat g_S (\xi) \|_{H^{s,*}_v}^2 + \frac{|\xi|^2}{\la \xi \ra^2} \| \P \widehat g_S (\xi) \|_{L^2_v}^2 \right) + C \eps^2 \| \widehat S(\xi) \|_{(H^{s,*}_v)'}^2 ,
\end{aligned}
$$
for some constants $\lambda,\sigma,C>0$. By separating the cases of high and low frequencies, we can conclude exactly as in the proof of Proposition~\ref{prop:estimate_Ueps_decay_soft_wholespace}.
\end{proof}

\section{Well-posedness and regularization for the rescaled Boltzmann equation}\label{section well-posedness Boltzmann}

Consider the equation \eqref{eq:feps_intro} that we rewrite here
$$
\left\{
\begin{aligned}
& \partial_t  f^\eps   = \frac{1}{\eps^2}(L-\eps v \cdot \nabla_x) f^\eps + \frac{1}{\eps} \Gamma(f^\eps , f^\eps) \\
& f^\eps_{t=0}=f^\eps_0.
\end{aligned}
\right.
$$
We shall consider mild solutions of \eqref{eq:feps_intro}, that is, we shall  prove the well-posedness of a solution $f^\eps$ to \eqref{eq:feps_intro} in Duhamel's form
\begin{equation}\label{eq:feps}
f^\eps (t) = U^\eps(t) f^\eps_0 + \frac{1}{\eps}\int_0^t U^\eps(t-\tau) \Gamma (f^\eps(\tau) , f^\eps(\tau)) \, \d \tau.
\end{equation}
Taking the Fourier transform in space of \eqref{eq:feps_intro}, we have
\begin{equation}\label{eq:feps_intro_fourier}
\left\{
\begin{aligned}
& \partial_t \widehat f^\eps (\xi)  = \Lambda^\eps(\xi) \widehat f^\eps(\xi) + \frac{1}{\eps} \widehat{\Gamma}(f^\eps , f^\eps)(\xi)  \\
& \widehat f^\eps (\xi)_{t=0}= \widehat f^\eps_0 (\xi),
\end{aligned}
\right.
\end{equation}
and by Duhamel's formula
\begin{equation}\label{eq:feps_fourier}
\widehat f^\eps(t,\xi) = \widehat U^\eps(t,\xi) \widehat f^\eps_0(\xi) + \frac{1}{\eps}\int_0^t \widehat U^\eps(t-\tau,\xi) \widehat \Gamma ( f^\eps(\tau) , f^\eps(\tau))(\xi) \, \d \tau.
\end{equation}

\subsection{Nonlinear estimates}

We start by recalling some well-known trilinear estimates on the collision operator $\Gamma$ established in \cite{GS,AMUXY,AMUXY5}.
We start with estimates without velocity weight. 
From \cite{GS,AMUXY5}, for the hard potentials case $\gamma+2s \ge 0$ there holds
\begin{equation}\label{eq:nonlinear_L2v_hard}
\left|\la \Gamma(f,g) , h \ra_{L^2_v } \right| 
\lesssim  \| f \|_{L^2_v  } \| g \|_{H^{s,*}_v }  \|  h \|_{H^{s,*}_v}.
\end{equation}
Moreover from \cite{AMUXY}, for the soft potentials case $\gamma+2s<0$ one has
\begin{equation}\label{eq:nonlinear_L2v_soft}
\begin{aligned}
&\left| \la \Gamma(f,g) , h \ra_{L^2_v} \right| \\
&\qquad
\lesssim \left( \| \la v \ra^{\gamma/2+s} f \|_{L^2_v} \| g \|_{H^{s,*}_v } +  \| f \|_{H^{s,*}_v } \| \la v \ra^{\gamma/2+s} g \|_{L^2_v} \right) \|  h \|_{H^{s,*}_v } \\
&\qquad\quad
+ \min\left\{ \| \la v \ra^{\gamma/2+s} f \|_{L^2_v  } \| g \|_{L^2_v} , \|  f \|_{L^2_v  } \| \la v \ra^{\gamma/2+s} g \|_{L^2_v}  \right\} \|  h \|_{H^{s,*}_v}.
\end{aligned}
\end{equation}
From these estimates we already obtain
\begin{equation}\label{eq:nonlinear_Hs*v'}
\begin{aligned}
\| \Gamma(f,g) \|_{(H^{s,*}_v)'} 
&= \sup_{\| \phi \|_{H^{s,*}_v } \le 1} \la \Gamma(f,g), \phi \ra_{L^2_v} \\
&\lesssim  \| \la v \ra^{-(\gamma/2+s)_{-}} f \|_{L^2_v} \| g \|_{H^{s,*}_v } +  \| f \|_{H^{s,*}_v } \| \la v \ra^{-(\gamma/2+s)_{-}} g \|_{L^2_v} \\
&\quad
+ \min\left\{ \| \la v \ra^{-(\gamma/2+s)_{-}} f \|_{L^2_v  } \| g \|_{L^2_v} , \|  f \|_{L^2_v  } \| \la v \ra^{-(\gamma/2+s)_{-}} g \|_{L^2_v}  \right\},
\end{aligned}
\end{equation}
where we denote $a_- = -\min(-a,0) $, which holds for both hard and soft potentials.

For the soft potentials case, we shall also need estimates when adding velocity weight $\la v \ra^\ell$. From \eqref{eq:nonlinear_L2v_soft} together with the commutator estimate of \cite[Proposition~3.13]{AMUXY}, there holds
\begin{equation}\label{eq:nonlinear_L2v_soft_weighted}
\begin{aligned}
&\left|\la \Gamma(f,g) , h \ra_{L^2_v (\la v \ra^\ell)} \right| \\
&\qquad
\lesssim 
\left( \| \la v \ra^{\gamma/2+s} f \|_{L^2_v} \| g \|_{H^{s,*}_v(\la v \ra^\ell) } +  \| f \|_{H^{s,*}_v } \| \la v \ra^{\gamma/2+s} g \|_{L^2_v (\la v \ra^\ell)} \right) \|  h \|_{H^{s,*}_v (\la v \ra^\ell) } \\
&\qquad\quad
+ \min\left\{ \| \la v \ra^{\gamma/2+s} f \|_{L^2_v  } \| g \|_{L^2_v (\la v \ra^\ell)} , \|  f \|_{L^2_v  } \| \la v \ra^{\gamma/2+s} g \|_{L^2_v (\la v \ra^\ell)}  \right\} \|  h \|_{H^{s,*}_v (\la v \ra^\ell)}.
\end{aligned}
\end{equation}
Therefore we also deduce
\begin{equation}\label{eq:nonlinear_Hs*v'_soft_weighted}
\begin{aligned}
\| \la v \ra^{\ell} \Gamma(f,g) \|_{(H^{s,*}_v)'} 
&= \sup_{\| \phi \|_{H^{s,*}_v (\la v \ra^\ell)} \le 1} \la \Gamma(f,g), \phi \ra_{L^2_v(\la v \ra^\ell)} \\
&\lesssim  \| \la v \ra^{\gamma/2+s} f \|_{L^2_v} \| g \|_{H^{s,*}_v(\la v \ra^\ell) } +   \| f \|_{H^{s,*}_v } \| \la v \ra^{\gamma/2+s} g \|_{L^2_v (\la v \ra^\ell)}  \\
&\quad
+ \min\left\{ \| \la v \ra^{\gamma/2+s} f \|_{L^2_v  } \| g \|_{L^2_v (\la v \ra^\ell)} , \|  f \|_{L^2_v  } \| \la v \ra^{\gamma/2+s} g \|_{L^2_v (\la v \ra^\ell)}  \right\} ,
\end{aligned}
\end{equation}
for the soft potentials case.

Thanks to \eqref{eq:nonlinear_Hs*v'} we deduce our main nonlinear estimate without weight.

\begin{lem}\label{lem:nonlinear}
Let $p \in [1,\infty]$. For any smooth enough functions $f,g$ there holds
$$
\| \widehat \Gamma (f,g) \|_{L^p_\xi L^2_t (H^{s,*}_v)'} 
\lesssim \Gamma_1 + \Gamma_2 + \min\left\{ \Gamma_3 , \Gamma_4 \right\}
$$
where
$$
\begin{aligned}
\Gamma_1 
= \min\Big\{ 
&\| \la v \ra^{-(\gamma/2+s)_{-}}  \widehat f \|_{L^p_\xi L^\infty_t L^2_v} \| \widehat g \|_{ L^1_\xi L^2_t H^{s,*}_v} , \| \la v \ra^{-(\gamma/2+s)_{-}}  \widehat f \|_{L^p_\xi L^2_t L^2_v} \| \widehat g \|_{ L^1_\xi L^\infty_t H^{s,*}_v}  ,\\
&\| \la v \ra^{-(\gamma/2+s)_{-}}  \widehat f \|_{L^1_\xi L^\infty_t L^2_v} \| \widehat g \|_{ L^p_\xi L^2_t H^{s,*}_v} , \| \la v \ra^{-(\gamma/2+s)_{-}}  \widehat f \|_{L^1_\xi L^2_t L^2_v} \| \widehat g \|_{ L^p_\xi L^\infty_t H^{s,*}_v}
\Big\},
\end{aligned}
$$
$$
\begin{aligned}
\Gamma_2 = \min\Big\{ 
&\| \widehat f \|_{L^p_\xi L^\infty_t H^{s,*}_v} \|\la v \ra^{-(\gamma/2+s)_{-}}  \widehat g \|_{ L^1_\xi L^2_t L^2_v} , \|  \widehat f \|_{L^p_\xi L^2_t H^{s,*}_v} \| \la v \ra^{-(\gamma/2+s)_{-}} \widehat g \|_{ L^1_\xi L^\infty_t L^2_v} ,\\
& \| \widehat f \|_{L^1_\xi L^\infty_t H^{s,*}_v} \|\la v \ra^{-(\gamma/2+s)_{-}}  \widehat g \|_{ L^p_\xi L^2_t L^2_v} ,
\|  \widehat f \|_{L^1_\xi L^2_t H^{s,*}_v} \| \la v \ra^{-(\gamma/2+s)_{-}} \widehat g \|_{ L^p_\xi L^\infty_t L^2_v} 
\Big\},
\end{aligned}
$$
$$
\begin{aligned}
\Gamma_3 = \min\Big\{ 
&\| \la v \ra^{-(\gamma/2+s)_{-}}  \widehat f \|_{L^p_\xi L^\infty_t L^2_v} \| \widehat g \|_{ L^1_\xi L^2_t L^2_v} , \| \la v \ra^{-(\gamma/2+s)_{-}}  \widehat f \|_{L^p_\xi L^2_t L^2_v} \| \widehat g \|_{ L^1_\xi L^\infty_t L^2_v} , \\
&\| \la v \ra^{-(\gamma/2+s)_{-}}  \widehat f \|_{L^1_\xi L^\infty_t L^2_v} \| \widehat g \|_{L^p_\xi L^2_t L^2_v} , \| \la v \ra^{-(\gamma/2+s)_{-}}  \widehat f \|_{L^1_\xi L^2_t L^2_v} \| \widehat g \|_{L^p_\xi L^\infty_t L^2_v}
\Big\},
\end{aligned}
$$
and
$$
\begin{aligned}
\Gamma_4 = \min\Big\{ 
&  \| \widehat f \|_{L^p_\xi L^\infty_t L^2_v} \|\la v \ra^{-(\gamma/2+s)_{-}}  \widehat g \|_{ L^1_\xi L^2_t L^2_v} , 
\|  \widehat f \|_{L^p_\xi L^2_t L^2_v} \| \la v \ra^{-(\gamma/2+s)_{-}} \widehat g \|_{ L^1_\xi L^\infty_t L^2_v} ,\\
&  \| \widehat f \|_{L^1_\xi L^\infty_t L^2_v} \|\la v \ra^{-(\gamma/2+s)_{-}}  \widehat g \|_{ L^p_\xi L^2_t L^2_v} ,
\|  \widehat f \|_{L^1_\xi L^2_t L^2_v} \| \la v \ra^{-(\gamma/2+s)_{-}} \widehat g \|_{ L^p_\xi L^\infty_t L^2_v} 
\Big\}.
\end{aligned}
$$
\end{lem}

\begin{proof}
Using \eqref{eq:nonlinear_Hs*v'} we write
$$
\begin{aligned}
\left\{ \int_0^\infty \| \la v \ra^{\ell} \widehat \Gamma(f(t),g(t))(\xi) \|_{(H^{s,*}_v)'}^2 \, \d t \right\}^{1/2}
&\lesssim I_1 + I_2 + \min\{ I_3,I_4 \}
\end{aligned}
$$
with
$$
I_1 = \left\{  \int_0^\infty \left( \int_{\Omega'_\eta} \| \la v \ra^{-(\gamma/2+s)_{-}} f(t,\xi-\eta) \|_{L^2_v} \| \widehat g (t,\eta) \|_{H^{s,*}_v} \, \d \eta \right)^2 \, \d t \right\}^{1/2},
$$
$$
I_2 = \left\{  \int_0^\infty \left( \int_{\Omega'_\eta} \|  f(t,\xi-\eta) \|_{H^{s,*}_v} \| \la v \ra^{-(\gamma/2+s)_{-}}\widehat g (t,\eta) \|_{L^2_v} \, \d \eta \right)^2 \, \d t \right\}^{1/2},
$$
$$
I_3 = \left\{  \int_0^\infty \left( \int_{\Omega'_\eta} \| \la v \ra^{-(\gamma/2+s)_{-}} f(t,\xi-\eta) \|_{L^2_v} \| \widehat g (t,\eta) \|_{L^2_v} \, \d \eta \right)^2 \, \d t \right\}^{1/2},
$$
and
$$
I_4 = \left\{  \int_0^\infty \left( \int_{\Omega'_\eta} \|  f(t,\xi-\eta) \|_{L^2_v} \| \la v \ra^{-(\gamma/2+s)_{-}}\widehat g (t,\eta) \|_{L^2_v} \, \d \eta \right)^2 \, \d t \right\}^{1/2}.
$$

We now investigate the term $I_1$. Thanks to Minkowski and H\"older inequalities we then obtain
$$
\begin{aligned}
I_1
&\lesssim  \int_{\Omega'_\eta} \left( \int_0^\infty \|    \la v \ra^{-(\gamma/2+s)_{-}}  \widehat f(t,\xi-\eta) \|_{L^2_v}^2 \| \widehat g (t,\eta) \|_{H^{s,*}_v}^2 \, \d t \right)^{1/2}  \d \eta \\
&\lesssim \int_{\Omega'_\eta} \|   \la v \ra^{-(\gamma/2+s)_{-}}   \widehat f(\xi-\eta) \|_{L^\infty_t L^2_v}  \| \widehat g (\eta) \|_{L^2_t H^{s,*}_v} \, \d \eta.
\end{aligned}
$$
Taking the $L^p_\xi$ norm in above estimate and using Young's inequality for convolution we first obtain
$$
I_1 \lesssim \| \la v \ra^{-(\gamma/2+s)_{-}}  \widehat f \|_{L^p_\xi L^\infty_t L^2_v} \| \widehat g \|_{ L^1_\xi L^2_t H^{s,*}_v}
\quad\text{and}\quad
I_1 \lesssim \| \la v \ra^{-(\gamma/2+s)_{-}}  \widehat f \|_{L^1_\xi L^\infty_t L^2_v} \| \widehat g \|_{ L^p_\xi L^2_t H^{s,*}_v}.
$$
Arguing exactly as above but exchanging the role of $f$ and $g$ when performing H\"older's inequality, we also obtain
$$
I_1 \lesssim \| \la v \ra^{-(\gamma/2+s)_{-}}  \widehat f \|_{L^p_\xi L^2_t L^2_v} \| \widehat g \|_{ L^1_\xi L^\infty_t H^{s,*}_v}  
\quad\text{and}\quad
I_1 \lesssim \| \la v \ra^{-(\gamma/2+s)_{-}}  \widehat f \|_{L^1_\xi L^2_t L^2_v} \| \widehat g \|_{ L^p_\xi L^\infty_t H^{s,*}_v},
$$
that is $I_1 \lesssim \Gamma_1$.

The estimates for the other terms $I_2$, $I_3$ and $I_4$ can be obtained exactly as for $I_1$, so we omit it.
\end{proof}

Arguing exactly as in the proof of Lemma~\ref{lem:nonlinear} but using the weighted estimate \eqref{eq:nonlinear_Hs*v'_soft_weighted}, we also obtain the main weighted nonlinear estimate for soft potentials below, the proof of which we omit for simplicity.

\begin{lem}\label{lem:nonlinear_weighted}
Let $\ell>0$, $\gamma+2s<0$ and $p \in [1,\infty]$. For any smooth enough functions $f,g$ there holds
$$
\| \la v \ra^{\ell} \widehat \Gamma (f, g) \|_{L^p_\xi L^2_t (H^{s,*}_v)'} 
\lesssim \widetilde \Gamma_1 + \widetilde \Gamma_2 + \min\left\{ \widetilde \Gamma_3, \widetilde \Gamma_4 \right\}
$$
where
$$
\begin{aligned}
\widetilde \Gamma_1 
= \min\Big\{ 
&\| \la v \ra^{\gamma/2+s} \widehat f \|_{L^p_\xi L^\infty_t L^2_v } \| \widehat g \|_{ L^1_\xi L^2_t H^{s,*}_v (\la v \ra^\ell)} , \| \la v \ra^{\gamma/2+s}   \widehat f \|_{L^p_\xi L^2_t L^2_v } \| \widehat g \|_{ L^1_\xi L^\infty_t H^{s,*}_v (\la v \ra^\ell)}, \\
& \|  \la v \ra^{\gamma/2+s} \widehat f \|_{L^1_\xi L^\infty_t L^2_v } \| \widehat g \|_{ L^p_\xi L^2_t H^{s,*}_v (\la v \ra^\ell)} , \|  \la v \ra^{\gamma/2+s} \widehat f \|_{L^1_\xi L^2_t L^2_v } \| \widehat g \|_{ L^p_\xi L^\infty_t H^{s,*}_v (\la v \ra^\ell)}  
\Big\},
\end{aligned}
$$
$$
\begin{aligned}
\widetilde \Gamma_2 = \min\Big\{ 
& \| \widehat f \|_{L^p_\xi L^\infty_t H^{s,*}_v} \|\la v \ra^{\gamma/2+s}  \widehat g \|_{ L^1_\xi L^2_t L^2_v(\la v \ra^\ell)} , \| \widehat f \|_{L^p_\xi L^2_t H^{s,*}_v } \| \la v \ra^{\gamma/2+s} \widehat g \|_{ L^1_\xi L^\infty_t L^2_v(\la v \ra^\ell)},  \\
& \| \widehat f \|_{L^1_\xi L^\infty_t H^{s,*}_v} \| \la v \ra^{\gamma/2+s}  \widehat g \|_{ L^p_\xi L^2_t L^2_v(\la v \ra^\ell)} ,
\| \widehat f \|_{L^1_\xi L^2_t H^{s,*}_v } \| \la v \ra^{\gamma/2+s} \widehat g \|_{ L^p_\xi L^\infty_t L^2_v(\la v \ra^\ell)} 
\Big\},
\end{aligned}
$$
$$
\begin{aligned}
\widetilde \Gamma_3 = \min\Big\{ 
&\| \la v \ra^{\gamma/2+s}  \widehat f \|_{L^p_\xi L^\infty_t L^2_v} \| \widehat g \|_{ L^1_\xi L^2_t L^2_v(\la v \ra^\ell)} , \| \la v \ra^{\gamma/2+s} \widehat f \|_{L^p_\xi L^2_t L^2_v} \| \widehat g \|_{ L^1_\xi L^\infty_t L^2_v(\la v \ra^\ell)} ,\\
& \| \la v \ra^{\gamma/2+s} \widehat f \|_{L^1_\xi L^\infty_t L^2_v} \| \widehat g \|_{ L^p_\xi L^2_t L^2_v(\la v \ra^\ell)} , \| \la v \ra^{\gamma/2+s} \widehat f \|_{L^1_\xi L^2_t L^2_v} \| \widehat g \|_{ L^p_\xi L^\infty_t L^2_v(\la v \ra^\ell)}  
\Big\},
\end{aligned}
$$
and
$$
\begin{aligned}
\widetilde \Gamma_4 = \min\Big\{ 
&\|  \widehat f \|_{L^p_\xi L^\infty_t L^2_v} \| \la v \ra^{\gamma/2+s} \widehat g \|_{ L^1_\xi L^2_t L^2_v(\la v \ra^\ell)} , \| \widehat f \|_{L^p_\xi L^2_t L^2_v} \| \la v \ra^{\gamma/2+s} \widehat g \|_{ L^1_\xi L^\infty_t L^2_v(\la v \ra^\ell)}, \\
& \|   \widehat f \|_{L^1_\xi L^\infty_t L^2_v} \| \la v \ra^{\gamma/2+s} \widehat g \|_{ L^p_\xi L^2_t L^2_v(\la v \ra^\ell)} , \| \widehat f \|_{L^1_\xi L^2_t L^2_v} \| \la v \ra^{\gamma/2+s} \widehat g \|_{ L^p_\xi L^\infty_t L^2_v(\la v \ra^\ell)}  
\Big\} .
\end{aligned}
$$
\end{lem}

\subsection{Proof of Theorem~\ref{theo:boltzmann}--(1)}\label{sec:proof_theo_main1}

We consider the torus case $\Omega_x = \T^3$.

\subsubsection{Global existence}\label{sec:proof_theo_main1_existence}

Let $\ell = 0$ in the hard potentials case $\gamma+2s \ge 0$, and $\ell \ge 0$ in the soft potentials case $\gamma+2s<0$. We define the space
$$
\mathscr{X} = \left\{ f \in \FF^{-1}_x(L^1_\xi L^\infty_t L^2_v (\la v \ra^\ell) \cap L^1_\xi L^2_t H^{s,*}_v (\la v \ra^\ell) ) \mid f \text{ satisfies } \eqref{eq:normalization2}, \; \| f \|_{\mathscr{X}} <\infty \right\}
$$
with
$$
\| f \|_{\mathscr{X}} := \| \widehat f \|_{L^1_\xi L^\infty_t L^2_v(\la v \ra^\ell)} + \frac{1}{\eps} \|  \P^\perp \widehat f \|_{L^1_\xi L^2_t H^{s,*}_v(\la v \ra^\ell)} +  \|  \P \widehat f \|_{L^1_\xi L^2_t L^2_v}.
$$
Let $f^\eps_0 \in \FF^{-1}_x(L^1_\xi  L^2_v (\la v \ra^\ell) )$ verify
$$
\| \widehat f^\eps_0 \|_{L^1_\xi L^2_v} \le \eta_0,
$$
and consider the map $\Phi : \mathscr{X} \to \mathscr{X}$, $f^\eps \mapsto \Phi[f^\eps]$ defined by, for all $t \ge 0$ ,
\begin{equation}\label{eq:def:Phi}
\Phi[f^\eps](t) = U^\eps(t)  f^\eps_0 + \frac{1}{\eps}\int_0^t  U^\eps(t-\tau)  \Gamma ( f^\eps(\tau) , f^\eps(\tau)) \, \d \tau,
\end{equation}
thus, for all $\xi \in \Z^3$,
\begin{equation}\label{eq:def:Phi_fourier}
\widehat\Phi[f^\eps](t,\xi) = \widehat U^\eps(t,\xi) \widehat f^\eps_0(\xi) + \frac{1}{\eps}\int_0^t \widehat U^\eps(t-\tau,\xi) \widehat \Gamma ( f^\eps(\tau) , f^\eps(\tau))(\xi) \, \d \tau.
\end{equation}


Thanks to Proposition~\ref{prop:estimate_Ueps} we deduce, for some constant $C_0>0$ independent of $\eps$, that
$$
\| U^\eps(\cdot) f^\eps_0 \|_{\mathscr{X}} \le C_0 \| \widehat f^\eps_0 \|_{L^1_\xi L^2_v }.
$$
Moreover thanks to Proposition~\ref{prop:estimate_Ueps_regularization} and the fact that $\mathbf{P} \Gamma(f^\eps, f^\eps) = 0$ from \eqref{eq:Gamma_invariants}, we get, for some constant $C_1>0$ independent of $\eps$, 
$$
\begin{aligned}
\frac{1}{\eps} \left\| \int_0^t  U^\eps(t-\tau)  \Gamma ( f^\eps(\tau) , f^\eps(\tau)) \, \d \tau \right\|_{\mathscr{X}} 
&\le C_1 \| \la v \ra^\ell \widehat \Gamma (f^\eps,f^\eps) \|_{L^1_\xi L^2_t (H^{s,*}_v)'}\\
&\le C_1 \|  \widehat  f^\eps \|_{L^1_\xi L^\infty_t L^2_v} \| \widehat  f^\eps \|_{L^1_\xi L^2_t H^{s,*}_v (\la v \ra^\ell)}\\
&\le C_1 \| f^\eps \|_{\mathscr{X}}^2,
\end{aligned}
$$
where we have used Lemma~\ref{lem:nonlinear} or Lemma~\ref{lem:nonlinear_weighted} in the second line together with $\| \la v \ra^{-(\gamma/2+s)_{-}}  \phi \|_{L^2_v} \lesssim \min\{ \| \phi \|_{L^2_v} , \| \phi \|_{H^{s,*}_v} \}$. Gathering previous estimates yields
\begin{equation}\label{eq:estimate_Phi_1}
\| \Phi[f^\eps] \|_{\mathscr{X}} \le C_0 \| \widehat f^\eps_0 \|_{L^1_\xi L^2_v} + C_1 \| f^\eps \|_{\mathscr{X}}^2.
\end{equation}

Moreover for $f^\eps, g^\eps \in \mathscr{X}$ we first observe that
$$
\begin{aligned}
\Phi[f^\eps](t) - \Phi[g^\eps](t)
&= \frac{1}{\eps}\int_0^t  U^\eps(t-\tau)  \left\{ \Gamma ( f^\eps(\tau) , f^\eps(\tau)) -  \Gamma ( g^\eps(\tau) , g^\eps(\tau)) \right\} \, \d \tau .
\end{aligned}.
$$
Introducing the symmetrized version $\Gamma_{\mathrm{sym}}$ of $\Gamma$, namely
\begin{equation}\label{eq:def:Gamma_sym}
\Gamma_{\mathrm{sym}} (f,g) = \frac12 \Gamma(f,g) + \frac12 \Gamma(g,f),
\end{equation} 
we remark that, arguing as from obtaining the collision invariants in \eqref{eq:collision_invariants}, we have that $\Gamma_{\mathrm{sym}}(f,g)$ also verifies \eqref{eq:Gamma_invariants}, which means $\mathbf{P}  \Gamma_{\mathrm{sym}}(f,g) = 0$. 
We therefore obtain
\begin{equation}\label{eq:Phifeps-Phigeps}
\begin{aligned}
\Phi[f^\eps](t) - \Phi[g^\eps](t)
&= \frac{1}{\eps}\int_0^t  U^\eps(t-\tau)  \Gamma_{\mathrm{sym}} ( f^\eps(\tau) , f^\eps(\tau)-g^\eps(\tau)) \, \d \tau \\
&\quad
+\frac{1}{\eps}\int_0^t  U^\eps(t-\tau)  \Gamma_{\mathrm{sym}} ( g^\eps(\tau) , f^\eps(\tau)-g^\eps(\tau)) \, \d \tau 
\end{aligned}
\end{equation}
with $\mathbf{P} \Gamma_{\mathrm{sym}} ( f^\eps , f^\eps-g^\eps) = 0$ and $\mathbf{P} \Gamma_{\mathrm{sym}} ( g^\eps , f^\eps-g^\eps) = 0$.
Hence Proposition~\ref{prop:estimate_Ueps_regularization} and  Lemma~\ref{lem:nonlinear} or Lemma~\ref{lem:nonlinear_weighted} yields
$$
\begin{aligned}
&\| \Phi[f^\eps] - \Phi[g^\eps] \|_{\mathscr{X}} \\
&\quad 
\lesssim  \| \la v \ra^\ell \widehat \Gamma (f^\eps,f^\eps - g^\eps) \|_{L^1_\xi L^2_t (H^{s,*}_v)'} 
+  \| \la v \ra^\ell \widehat \Gamma (f^\eps - g^\eps,f^\eps) \|_{L^1_\xi L^2_t (H^{s,*}_v)'} \\
&\quad\quad
+  \| \la v \ra^\ell \widehat \Gamma (g^\eps,f^\eps - g^\eps) \|_{L^1_\xi L^2_t (H^{s,*}_v)'} 
+  \| \la v \ra^\ell \widehat \Gamma (f^\eps - g^\eps,g^\eps) \|_{L^1_\xi L^2_t (H^{s,*}_v)'} \\
&\quad 
\lesssim \|  \widehat  f^\eps \|_{L^1_\xi L^\infty_t L^2_v (\la v \ra^\ell)} \| \widehat  f^\eps - \widehat g^\eps \|_{L^1_\xi L^2_t H^{s,*}_v(\la v \ra^\ell)} +  \|  \widehat  f^\eps - g^\eps \|_{L^1_\xi L^\infty_t L^2_v (\la v \ra^\ell)} \| \widehat  f^\eps \|_{L^1_\xi L^2_t H^{s,*}_v(\la v \ra^\ell)} \\
&\quad\quad
+ \|  \widehat  g^\eps \|_{L^1_\xi L^\infty_t L^2_v(\la v \ra^\ell)} \| \widehat  f^\eps - \widehat g^\eps \|_{L^1_\xi L^2_t H^{s,*}_v(\la v \ra^\ell)} +  \|  \widehat  f^\eps - g^\eps \|_{L^1_\xi L^\infty_t L^2_v(\la v \ra^\ell)} \| \widehat  g^\eps \|_{L^1_\xi L^2_t H^{s,*}_v(\la v \ra^\ell)},
\end{aligned}
$$
thus we get, for some constant $C_1>0$ independent of $\eps$,
\begin{equation}\label{eq:estimate_Phi_2}
\begin{aligned}
\| \Phi[f^\eps] - \Phi[g^\eps] \|_{\mathscr{X}} 
\le C_1 (\| f^\eps \|_{\mathscr{X}} + \| g^\eps \|_{\mathscr{X}}) \| f^\eps - g^\eps \|_{\mathscr{X}}.
\end{aligned}
\end{equation}

As a consequence of estimates \eqref{eq:estimate_Phi_1}--\eqref{eq:estimate_Phi_2} we can construct a global solution $f^\eps \in \mathscr{X}$ to the equation~\eqref{eq:feps} if $\eta_0>0$ is small enough. Indeed let $B_{\mathscr{X}}(\eta) = \{ f \in \mathscr{X} \mid \| f \|_{\mathscr{X}} \le \eta \} $ for $\eta>0$ be the closed ball in $\mathscr{X}$ of radius $\eta$. Choose 
$$
\eta = 2 C_0 \eta_0
\quad\text{and}\quad
\eta_0 \le \frac{1}{8 C_0 C_1},
$$
and observe that $\eta_0$ does not depend on $\eps$. Then for any $f^\eps \in B_{\mathscr{X}}(\eta)$ we have from \eqref{eq:estimate_Phi_1} that
$$
\| \Phi [f^\eps] \|_{\mathscr{X}} \le 2 C_0 \eta_0 = \eta,
$$
and for any $f^\eps,g^\eps \in B_{\mathscr{X}}(\eta)$ we have from \eqref{eq:estimate_Phi_2} that
$$
\begin{aligned}
\| \Phi [f^\eps] - \Phi [g^\eps] \|_{\mathscr{X}}
&\le 4 C_0 C_1 \eta_0 \| f^\eps - g^\eps \|_{\mathscr{X}} 
\le \frac{1}{2} \| f^\eps - g^\eps \|_{\mathscr{X}}.
\end{aligned}
$$
Thus $\Phi : B_{\mathscr{X}}(\eta) \to B_{\mathscr{X}}(\eta)$ is a contraction and therefore there is a unique $f^\eps \in B_{\mathscr{X}}(\eta)$ such that $\Phi[f^\eps] = f^\eps$, which is then a solution to \eqref{eq:feps}.
This completes the proof of global existence in Theorem~\ref{theo:boltzmann}--(1) together with estimate~\eqref{eq:theo1:existence}.

\subsubsection{Uniqueness}\label{sec:proof_theo_main1_uniqueness}

Consider two solutions $f^\eps, g^\eps \in \FF^{-1}_x(L^1_\xi L^\infty_t L^2_v (\la v \ra^{\ell}) \cap L^1_\xi L^2_t H^{s,*}_v (\la v \ra^{\ell}) )$ to \eqref{eq:feps} associated to the same initial data $f^\eps_0 \in \FF^{-1}_x(L^1_\xi L^2_v (\la v \ra^{\ell}))$ satisfying $\| \widehat f^\eps_0 \|_{L^1_\xi L^2_v (\la v \ra^{\ell})} \le \eta_0$ with $\eta_0>0$ small enough and
$$
\begin{aligned}
\| \widehat f^\eps \|_{L^1_\xi L^\infty_t L^2_v(\la v \ra^{\ell})} + \| \widehat f^\eps  \|_{L^1_\xi L^2_t H^{s,*}_v(\la v \ra^{\ell})} &\lesssim \| \widehat f^\eps_0 \|_{L^1_\xi L^2_v (\la v \ra^{\ell})}, \\
\| \widehat g^\eps  \|_{L^1_\xi L^\infty_t L^2_v(\la v \ra^{\ell})} + \| \widehat g^\eps  \|_{L^1_\xi L^2_t H^{s,*}_v(\la v \ra^{\ell})} &\lesssim \| \widehat f^\eps_0 \|_{L^1_\xi L^2_v (\la v \ra^{\ell})}.
\end{aligned}
$$
Arguing as in the existence proof above, we obtain
$$
\begin{aligned}
&\| f^\eps - g^\eps \|_{L^1_\xi L^\infty_t L^2_v(\la v \ra^{\ell})}
+ \| f^\eps - g^\eps \|_{L^1_\xi L^2_t H^{s,*}_v(\la v \ra^{\ell})} \\
&\qquad
\lesssim \left( \| f^\eps \|_{L^1_\xi L^\infty_t L^2_v(\la v \ra^{\ell})} + \| g^\eps \|_{L^1_\xi L^2_t H^{s,*}_v(\la v \ra^{\ell})}   \right) \left( \| f^\eps - g^\eps \|_{L^1_\xi L^\infty_t L^2_v(\la v \ra^{\ell})} + \| f^\eps - g^\eps \|_{L^1_\xi L^2_t H^{s,*}_v(\la v \ra^{\ell})}
 \right).
\end{aligned}
$$
Using that $ \| f^\eps \|_{L^1_\xi L^\infty_t L^2_v(\la v \ra^{\ell})} + \| g^\eps \|_{L^1_\xi L^2_t H^{s,*}_v(\la v \ra^{\ell})} \lesssim \eta_0$ is small enough we conclude the proof of uniqueness in Theorem~\ref{theo:boltzmann}--(1).

\subsubsection{Decay for hard potentials}\label{sec:proof_theo_main1_decay_hard}
Let $f^\eps$ be the solution to \eqref{eq:feps} constructed in Theorem~\ref{theo:boltzmann}--(1) associated to the initial data $f^\eps_0$, and let $\lambda >0$ be given by Proposition~\ref{prop:estimate_Ueps}. Using Proposition~\ref{prop:estimate_Ueps_decay_hard_torus} and Proposition~\ref{prop:estimate_Ueps_regularization_hard_torus} we obtain
\begin{equation*}
\begin{aligned}
\| \mathrm{e}_{\lambda} \widehat f^\eps \|_{L^1_\xi L^\infty_t L^2_v} 
&+ \frac{1}{\eps} \|  \mathrm{e}_{\lambda}\P^\perp \widehat f^\eps \|_{L^1_\xi L^2_t H^{s,*}_v} 
+ \| \mathrm{e}_{\lambda} \P \widehat f^\eps \|_{L^1_\xi L^2_t L^2_v} \\
&\lesssim  \| \widehat f^\eps_0 \|_{L^1_\xi L^2_v }
+ \| \mathrm{e}_{\lambda}  \widehat \Gamma (f^\eps , f^\eps) \|_{L^1_\xi L^2_t (H^{s,*}_v)'}. 
\end{aligned}
\end{equation*}
Thanks to Lemma~\ref{lem:nonlinear} we have
$$
\| \mathrm{e}_{\lambda}  \widehat \Gamma (f^\eps , f^\eps) \|_{L^1_\xi L^2_t (H^{s,*}_v)'}
\lesssim \| \mathrm{e}_{\lambda} \widehat f^\eps \|_{L^1_\xi L^\infty_t L^2_v} \| \widehat f^\eps \|_{L^1_\xi L^2_t H^{s,*}_v},
$$
therefore using that $\| \widehat f^\eps \|_{L^1_\xi L^2_t H^{s,*}_v}\lesssim  \| \widehat f^\eps_0 \|_{L^1_\xi L^2_v }$ from the existence result in Theorem~\ref{theo:boltzmann}--(1), we obtain
\begin{equation*}
\begin{aligned}\| \mathrm{e}_{\lambda} \widehat f^\eps \|_{L^1_\xi L^\infty_t L^2_v} 
&+ \frac{1}{\eps} \|  \mathrm{e}_{\lambda}\P^\perp \widehat f^\eps \|_{L^1_\xi L^2_t H^{s,*}_v} 
+  \| \mathrm{e}_{\lambda} \P \widehat f^\eps \|_{L^1_\xi L^2_t L^2_v} \\
&\lesssim  \| \widehat f^\eps_0 \|_{L^1_\xi L^2_v }
+  \| \mathrm{e}_{\lambda} \widehat f^\eps \|_{L^1_\xi L^\infty_t L^2_v} \| \widehat f^\eps \|_{L^1_\xi L^2_t H^{s,*}_v}  \\
&\lesssim \| \widehat f^\eps_0 \|_{L^1_\xi L^2_v }
+  \| \widehat f^\eps_0 \|_{L^1_\xi L^2_v }
 \| \mathrm{e}_{\lambda} \widehat f^\eps \|_{L^1_\xi L^\infty_t L^2_v} . 
\end{aligned}
\end{equation*}
Since $\| \widehat f^\eps_0 \|_{L^1_\xi L^2_v} \le \eta_0$ is small enough, the last term in the right-hand side can be absorbed into the left-hand side, which thus concludes the proof of the decay estimate~\eqref{eq:theo1:decay_hard} in Theorem~\ref{theo:boltzmann}--(1).

\subsubsection{Decay for soft potentials}\label{sec:proof_theo_main1_decay_soft}
Let $f^\eps$ be the solution to \eqref{eq:feps} constructed in Theorem~\ref{theo:boltzmann}--(1) associated to the initial data $f^\eps_0$ with $\ell >0$, and let $0 < \omega < \frac{\ell}{|\gamma+2s|}$.

Using Proposition~\ref{prop:estimate_Ueps_decay_soft_torus} and Proposition~\ref{prop:estimate_Ueps_regularization_soft_torus} we obtain
\begin{equation*}
\begin{aligned}
\| \mathrm{p}_{\omega} \widehat f^\eps \|_{L^1_\xi L^\infty_t L^2_v} 
&+ \frac{1}{\eps} \| \mathrm{p}_{\omega} \P^\perp \widehat f^\eps \|_{L^1_\xi L^2_t H^{s,*}_v} 
+ \| \mathrm{p}_{\omega} \P \widehat f^\eps \|_{L^1_\xi L^2_t L^2_v} \\
&\lesssim  \| \widehat f^\eps_0 \|_{L^1_\xi L^2_v}
+\| \widehat f^\eps \|_{L^1_\xi L^\infty_t L^2_v (\la v \ra^\ell)}
+ \| \mathrm{p}_{\omega}  \widehat \Gamma (f^\eps , f^\eps) \|_{L^1_\xi L^2_t (H^{s,*}_v)'} ,
\end{aligned}
\end{equation*}
and from  Lemma~\ref{lem:nonlinear} we have
$$
\| \mathrm{p}_{\omega}  \widehat \Gamma (f^\eps , f^\eps) \|_{L^1_\xi L^2_t (H^{s,*}_v)'}
\lesssim \| \mathrm{p}_{\omega}  \widehat f^\eps \|_{L^1_\xi L^\infty_t L^2_v} 
\|  \widehat f^\eps \|_{L^1_\xi L^2_t H^{s,*}_v} .
$$
Using that $\|  \widehat f^\eps \|_{L^1_\xi L^\infty_t L^2_v(\la v \ra^\ell)}  + \|  \widehat f^\eps \|_{L^1_\xi L^2_t H^{s,*}_v} \lesssim \| \widehat f^\eps_0 \|_{L^1_\xi L^2_v (\la v \ra^\ell)} $ from the existence result in in Theorem~\ref{theo:boltzmann}--(1), we deduce
\begin{equation*}
\begin{aligned}
\| \mathrm{p}_{\omega} \widehat f^\eps \|_{L^1_\xi L^\infty_t L^2_v} 
&+ \frac{1}{\eps} \| \mathrm{p}_{\omega} \P^\perp \widehat f^\eps \|_{L^1_\xi L^2_t H^{s,*}_v} 
+ \| \mathrm{p}_{\omega} \P \widehat f^\eps \|_{L^1_\xi L^2_t L^2_v} \\
&\lesssim  \| \widehat f^\eps_0 \|_{L^1_\xi L^2_v (\la v \ra^\ell)}
+ \| \widehat f^\eps_0 \|_{L^1_\xi L^2_v (\la v \ra^\ell)} \| \mathrm{p}_{\omega} \widehat f^\eps \|_{L^1_\xi L^\infty_t L^2_v} .
\end{aligned}
\end{equation*}
Since $\| \widehat f^\eps_0 \|_{L^1_\xi L^2_v (\la v \ra^\ell)} \le \eta_0$ is small enough, the last term in the right-hand side can be absorbed into the left-hand side, which thus concludes the proof of the decay estimate~\eqref{eq:theo1:decay_soft} in Theorem~\ref{theo:boltzmann}--(1).

\subsection{Proof of Theorem~\ref{theo:boltzmann}--(2)}\label{sec:proof_theo_main1bis}

We consider the whole space case $\Omega_x= \R^3$.

\subsubsection{Global existence}\label{sec:proof_theo_main1bis_existence}

Let $\ell = 0$ in the hard potentials case $\gamma+1s \ge 0$, $\ell \ge 0$ in the soft potentials case $\gamma+2s<0$.
Recall that $p \in (3/2,\infty]$ and define the space
$$
\begin{aligned}
\mathscr{Y} = \Big\{ f \in &\FF^{-1}_x(L^1_\xi L^\infty_t L^2_v (\la v \ra^\ell) \cap L^1_\xi L^2_t H^{s,**}_v (\la v \ra^\ell)) \cap L^p_\xi L^\infty_t L^2_v (\la v \ra^\ell) \cap L^p_\xi L^2_t H^{s,**}_v (\la v \ra^\ell) ) \;\Big|\; \| f \|_{\mathscr{Y}} <\infty \Big\},
\end{aligned}
$$
with, recalling that $\| \cdot \|_{H^{s,**}_v}$ is defined in \eqref{eq:def:Hs**v},
$$
\begin{aligned}
\| f \|_{\mathscr{Y}} 
&:= \| \widehat f \|_{L^1_\xi L^\infty_t L^2_v(\la v \ra^\ell)} + \frac{1}{\eps} \|  \P^\perp \widehat f \|_{L^1_\xi L^2_t H^{s,*}_v(\la v \ra^\ell)} +  \left\| \frac{|\xi|}{\la \xi \ra}  \P \widehat f \right\|_{L^1_\xi L^2_t L^2_v} \\
&\quad
+ \| \widehat f \|_{L^p_\xi L^\infty_t L^2_v(\la v \ra^\ell)} + \frac{1}{\eps} \|  \P^\perp \widehat f \|_{L^p_\xi L^2_t H^{s,*}_v(\la v \ra^\ell)} 
+  \left\| \frac{|\xi|}{\la \xi \ra}  \P \widehat f \right\|_{L^p_\xi L^2_t L^2_v}.
\end{aligned}
$$
Let $f^\eps_0 \in \FF^{-1}_x(L^1_\xi L^2_v (\la v \ra^\ell) \cap L^p_\xi L^2_v (\la v \ra^\ell) )$ verify
$$
\| \widehat f^\eps_0 \|_{L^1_\xi L^2_v(\la v \ra^\ell} + \| \widehat f^\eps_0 \|_{L^p_\xi L^2_v(\la v \ra^\ell)} \le \eta_0,
$$
and consider the map $\Phi : \mathscr{Y} \to \mathscr{Y}$, $f^\eps \mapsto \Phi[f^\eps]$ given by \eqref{eq:def:Phi}, which in particular satisfies \eqref{eq:def:Phi_fourier} for all $\xi \in \R^3$.

Thanks to Proposition~\ref{prop:estimate_Ueps} we deduce, for some constant $C_0>0$ independent of $\eps$, that
$$
\| U^\eps(\cdot) f^\eps_0 \|_{\mathscr{Y}} \le C_0 \left( \| \widehat f^\eps_0 \|_{L^1_\xi L^2_v (\la v \ra^\ell)}
+ \| \widehat f^\eps_0 \|_{L^p_\xi L^2_v (\la v \ra^\ell)} \right).
$$
Moreover thanks to Proposition~\ref{prop:estimate_Ueps_regularization} and the fact that $\mathbf{P} \Gamma(f^\eps, f^\eps) = 0$ from \eqref{eq:Gamma_invariants}, we get 
$$
\begin{aligned}
&\frac{1}{\eps} \left\| \int_0^t  U^\eps(t-\tau)  \Gamma ( f^\eps(\tau) , f^\eps(\tau)) \, \d \tau \right\|_{\mathscr{Y}} \\
&\qquad 
\lesssim   \| \la v \ra^\ell \widehat \Gamma (f^\eps,f^\eps) \|_{L^1_\xi L^2_t (H^{s,*}_v)'} + \| \la v \ra^\ell \widehat \Gamma (f^\eps,f^\eps) \|_{L^p_\xi L^2_t (H^{s,*}_v)'}  \\
&\qquad
\lesssim  \left( \|  \widehat  f^\eps \|_{L^1_\xi L^\infty_t L^2_v(\la v \ra^\ell)}  + \|  \widehat  f^\eps \|_{L^p_\xi L^\infty_t L^2_v(\la v \ra^\ell)} \right) \| \widehat  f^\eps \|_{L^1_\xi L^2_t H^{s,*}_v(\la v \ra^\ell)},
\end{aligned}
$$
where we have used Lemma~\ref{lem:nonlinear} or Lemma~\ref{lem:nonlinear_weighted} in the second line together with $\| \la v \ra^{-(\gamma/2+s)_{-}}  \phi \|_{L^2_v} \lesssim \min\{ \| \phi \|_{L^2_v}  , \| \phi \|_{H^{s,*}_v} \}$. We now observe that, splitting $\widehat  f^\eps = \P^\perp \widehat  f^\eps + \P \widehat  f^\eps$, on the one hand we have
$$
\begin{aligned}
\| \widehat  f^\eps \|_{L^1_\xi L^2_t H^{s,*}_v(\la v \ra^\ell)} 
&\lesssim \| \P^\perp \widehat  f^\eps \|_{L^1_\xi L^2_t H^{s,*}_v(\la v \ra^\ell)}  + \|  \P \widehat f^\eps \|_{L^1_\xi L^2_t L^2_v}. 
\end{aligned}
$$
On the other hand
$$
\begin{aligned}
\| \P\widehat  f^\eps \|_{L^1_\xi L^2_t L^2_v}
& \lesssim \| \mathbf{1}_{|\xi| \ge 1}  \P \widehat f^\eps \|_{L^1_\xi L^2_t L^2_v} + \| \mathbf{1}_{|\xi| < 1}  \P  \widehat f^\eps \|_{L^1_\xi L^2_t L^2_v} \\
& \lesssim \left\| \mathbf{1}_{|\xi| \ge 1} \frac{|\xi|}{\la \xi \ra} \P \widehat f^\eps \right\|_{L^1_\xi L^2_t L^2_v} 
+ \left\| \mathbf{1}_{|\xi| < 1} |\xi|^{-1} \, \frac{|\xi|}{\la \xi \ra}  \P \widehat f^\eps \right\|_{L^1_\xi L^2_t L^2_v} \\
& \lesssim \left\|  \frac{|\xi|}{\la \xi \ra} \P \widehat f^\eps \right\|_{L^1_\xi L^2_t L^2_v} 
+ \left\|  \frac{|\xi|}{\la \xi \ra} \P \widehat f^\eps \right\|_{L^p_\xi L^2_t L^2_v},
\end{aligned}
$$
where we have used H\"older's inequality in last line, using that $p>3/2$ so that $\mathbf{1}_{|\xi| < 1} |\xi|^{-1} \in L^{p'}_\xi$. Putting together the two last estimates, we have
\begin{equation}\label{eq:fPf_Hs*v}
\begin{aligned}
\| \widehat  f^\eps \|_{L^1_\xi L^2_t H^{s,*}_v(\la v \ra^\ell)} 
\lesssim \| \P^\perp \widehat  f^\eps \|_{L^1_\xi L^2_t H^{s,*}_v(\la v \ra^\ell)}
+ \left\|  \frac{|\xi|}{\la \xi \ra} \P \widehat f^\eps \right\|_{L^1_\xi L^2_t L^2_v} 
+ \left\|  \frac{|\xi|}{\la \xi \ra}  \P \widehat f^\eps \right\|_{L^p_\xi L^2_t L^2_v}.
\end{aligned}
\end{equation}
We hence deduce that there is some constant $C_1>0$, independent of $\eps$, such that
$$
\begin{aligned}
&\frac{1}{\eps} \left\| \int_0^t  U^\eps(t-\tau)  \Gamma ( f^\eps(\tau) , f^\eps(\tau)) \, \d \tau \right\|_{\mathscr{X}} \\
&\qquad 
\le C_1  \left( \|  \widehat  f^\eps \|_{L^1_\xi L^\infty_t L^2_v}  + \|  \widehat  f^\eps \|_{L^p_\xi L^\infty_t L^2_v} \right)\\
&\qquad\qquad \times  \left( \| \P^\perp \widehat  f^\eps \|_{L^1_\xi L^2_t H^{s,*}_v(\la v \ra^\ell)}  + \left\|  \frac{|\xi|}{\la \xi \ra} \P \widehat f^\eps \right\|_{L^1_\xi L^2_t L^2_v}  + \left\|   \frac{|\xi|}{\la \xi \ra}  \P \widehat f^\eps \right\|_{L^p_\xi L^2_t L^2_v}\right)  .
\end{aligned}
$$
Therefore, gathering previous estimates, we obtain
\begin{equation}\label{eq:Phif_wholespace}
\| \Phi[f^\eps] \|_{\mathscr{Y}} \le C_0 \left( \| \widehat f^\eps_0 \|_{L^1_\xi L^2_v (\la v \ra^\ell} +  \| \widehat f^\eps_0 \|_{L^p_\xi L^2_v (\la v \ra^\ell} \right) + C_1 \| f^\eps \|_{\mathscr{Y}}^2.
\end{equation}

Moreover, for $f^\eps , g^\eps \in \mathscr{Y}$ we first observe that $\Phi[f^\eps] - \Phi[g^\eps]$ satisfies \eqref{eq:Phifeps-Phigeps}. Therefore
we obtain arguing as above, thanks to Proposition~\ref{prop:estimate_Ueps_regularization} together with $\mathbf{P} \Gamma_{\mathrm{sym}} (f^\eps, f^\eps - g^\eps)$, Lemma~\ref{lem:nonlinear} and Lemma~\ref{lem:nonlinear_weighted}, that
$$
\begin{aligned}
&\frac{1}{\eps} \left\| \int_0^t  U^\eps(t-\tau)  \Gamma_{\mathrm{sym}} ( f^\eps(\tau) , f^\eps(\tau) - g^\eps(\tau) ) \, \d \tau \right\|_{\mathscr{X}} \\
&\qquad 
\lesssim   \| \la v \ra^\ell \widehat \Gamma (f^\eps,f^\eps-g^\eps) \|_{L^1_\xi L^2_t (H^{s,*}_v)'} + \| \la v \ra^\ell \widehat \Gamma (f^\eps,f^\eps-g^\eps) \|_{L^p_\xi L^2_t (H^{s,*}_v)'}  \\
&\qquad\quad
+ \| \la v \ra^\ell \widehat \Gamma (f^\eps-g^\eps, f^\eps) \|_{L^1_\xi L^2_t (H^{s,*}_v)'} + \| \la v \ra^\ell \widehat \Gamma (f^\eps-g^\eps , f^\eps) \|_{L^p_\xi L^2_t (H^{s,*}_v)'}\\
&\qquad
\lesssim   \|  \widehat  f^\eps \|_{L^1_\xi L^\infty_t L^2_v(\la v \ra^\ell)} \| \widehat  f^\eps - \widehat g^\eps \|_{L^1_\xi L^2_t H^{s,*}_v(\la v \ra^\ell)}    
+  \|  \widehat  f^\eps \|_{L^p_\xi L^\infty_t L^2_v(\la v \ra^\ell)} \| \widehat  f^\eps - \widehat g^\eps \|_{L^1_\xi L^2_t H^{s,*}_v(\la v \ra^\ell)}   \\
&\qquad \quad
+ \|  \widehat  f^\eps - \widehat g^\eps \|_{L^1_\xi L^\infty_t L^2_v(\la v \ra^\ell)} \| \widehat  f^\eps  \|_{L^1_\xi L^2_t H^{s,*}_v(\la v \ra^\ell)}
+ \|  \widehat  f^\eps - \widehat g^\eps \|_{L^p_\xi L^\infty_t L^2_v(\la v \ra^\ell)} \| \widehat  f^\eps  \|_{L^1_\xi L^2_t H^{s,*}_v(\la v \ra^\ell)} , 
\end{aligned}
$$
and similarly
$$
\begin{aligned}
&\frac{1}{\eps} \left\| \int_0^t  U^\eps(t-\tau)  \Gamma_{\mathrm{sym}} ( g^\eps(\tau) , f^\eps(\tau) - g^\eps(\tau) ) \, \d \tau \right\|_{\mathscr{X}} \\
&\qquad 
\lesssim   \| \la v \ra^\ell \widehat \Gamma (g^\eps,f^\eps-g^\eps) \|_{L^1_\xi L^2_t (H^{s,*}_v)'} + \| \la v \ra^\ell \widehat \Gamma (g^\eps,f^\eps-g^\eps) \|_{L^p_\xi L^2_t (H^{s,*}_v)'}  \\
&\qquad\quad
+ \| \la v \ra^\ell \widehat \Gamma (f^\eps-g^\eps, g^\eps) \|_{L^1_\xi L^2_t (H^{s,*}_v)'} + \| \la v \ra^\ell \widehat \Gamma (f^\eps-g^\eps , g^\eps) \|_{L^p_\xi L^2_t (H^{s,*}_v)'}\\
&\qquad
\lesssim   \|  \widehat  g^\eps \|_{L^1_\xi L^\infty_t L^2_v(\la v \ra^\ell)} \| \widehat  f^\eps - \widehat g^\eps \|_{L^1_\xi L^2_t H^{s,*}_v(\la v \ra^\ell)}    
+  \|  \widehat  g^\eps \|_{L^p_\xi L^\infty_t L^2_v(\la v \ra^\ell)} \| \widehat  f^\eps - \widehat g^\eps \|_{L^1_\xi L^2_t H^{s,*}_v(\la v \ra^\ell)}   \\
&\qquad \quad
+ \|  \widehat  f^\eps - \widehat g^\eps \|_{L^1_\xi L^\infty_t L^2_v(\la v \ra^\ell)} \| \widehat  g^\eps  \|_{L^1_\xi L^2_t H^{s,*}_v(\la v \ra^\ell)}
+ \|  \widehat  f^\eps - \widehat g^\eps \|_{L^p_\xi L^\infty_t L^2_v(\la v \ra^\ell)} \| \widehat  g^\eps  \|_{L^1_\xi L^2_t H^{s,*}_v(\la v \ra^\ell)}.  
\end{aligned}
$$
Together with \eqref{eq:fPf_Hs*v} for the terms in $\|\cdot \|_{L^1_\xi L^2_t H^{s,*}_v(\la v \ra^\ell)}$, this implies that, for some constant $C_1>0$ independent of $\eps$,
\begin{equation}\label{eq:Phif-Phig_wholespace}
\| \Phi[f^\eps] - \Phi[g^\eps] \|_{\mathscr{Y}}
\le C_1 ( \| f \|_{\mathscr{Y}} + \| g \|_{\mathscr{Y}}) \| f-g \|_{\mathscr{Y}}.
\end{equation}

As a consequence of estimates \eqref{eq:Phif_wholespace}--\eqref{eq:Phif-Phig_wholespace} we can construct a global solution $f^\eps \in \mathscr{Y}$ to the equation~\eqref{eq:feps} if $\eta_0>0$ is small enough by arguing as in Section~\ref{sec:proof_theo_main1_existence}. This completes the proof of global existence in Theorem~\ref{theo:boltzmann}--(2) together with estimate~\eqref{eq:theo1bis:existence}.

\subsubsection{Uniqueness}\label{sec:proof_theo_main1bis_uniqueness}

Using the above estimates, we can argue as in Section~\ref{sec:proof_theo_main1_uniqueness}.

\subsubsection{Decay for hard potentials}\label{sec:proof_theo_main1bis_decay_hard}
Let $f^\eps$ be the solution to \eqref{eq:feps} constructed in Theorem~\ref{theo:boltzmann}--(2) associated to the initial data $f^\eps_0$, and let $0 < \vartheta < \frac{3}{2}(1-\frac{1}{p})$. Arguing as above, using Proposition~\ref{prop:estimate_Ueps_decay_hard_wholespace} and Proposition~\ref{prop:estimate_Ueps_regularization_hard_wholespace} we obtain
\begin{equation*}
\begin{aligned}
\| \mathrm{p}_{\vartheta} \widehat f^\eps \|_{L^1_\xi L^\infty_t L^2_v} 
&+ \frac{1}{\eps} \|  \mathrm{p}_{\vartheta} \P^\perp \widehat f^\eps \|_{L^1_\xi L^2_t H^{s,*}_v} 
+ \left\| \mathrm{p}_{\vartheta} \frac{|\xi|}{\la \xi \ra} \P \widehat f^\eps \right\|_{L^1_\xi L^2_t L^2_v} \\
&\lesssim  \| \widehat f^\eps_0 \|_{L^1_\xi L^2_v }
+ \|  \widehat f^\eps \|_{L^p_\xi L^\infty_t L^2_v} 
+ \| \mathrm{p}_{\vartheta}  \widehat \Gamma (f^\eps , f^\eps) \|_{L^1_\xi L^2_t (H^{s,*}_v)'}. 
\end{aligned}
\end{equation*}
Thanks to Lemma~\ref{lem:nonlinear} we have
$$
\| \mathrm{p}_{\vartheta}  \widehat \Gamma (f^\eps , f^\eps) \|_{L^1_\xi L^2_t (H^{s,*}_v)'}
\lesssim \| \mathrm{p}_{\vartheta} \widehat f^\eps \|_{L^1_\xi L^\infty_t L^2_v} \| \widehat f^\eps \|_{L^1_\xi L^2_t H^{s,*}_v},
$$
and by \eqref{eq:fPf_Hs*v} we have
$$
\begin{aligned}
\| \widehat f^\eps \|_{L^1_\xi L^2_t H^{s,*}_v}
&\lesssim 
\| \P^\perp \widehat f^\eps \|_{L^1_\xi L^2_t H^{s,*}_v}
+ \left\|  \frac{|\xi|}{\la \xi \ra}\widehat \P f^\eps \right\|_{L^1_\xi L^2_t L^2_v} 
+ \left\|  \frac{|\xi|}{\la \xi \ra} \widehat \P f^\eps \right\|_{L^p_\xi L^2_t L^2_v} \\
&\lesssim \| \widehat f^\eps_0 \|_{L^1_\xi L^2_v } + \| \widehat f^\eps_0 \|_{L^p_\xi L^2_v },
\end{aligned}
$$
where we have used the estimate of Theorem~\ref{theo:boltzmann}--(2) in last line. Observing that we also have $\|  \widehat f^\eps \|_{L^p_\xi L^\infty_t L^2_v} \lesssim \| \widehat f^\eps_0 \|_{L^1_\xi L^2_v } + \| \widehat f^\eps_0 \|_{L^p_\xi L^2_v }$, it follows
\begin{equation*}
\begin{aligned}
\| \mathrm{p}_{\vartheta} \widehat f^\eps \|_{L^1_\xi L^\infty_t L^2_v} 
&+ \frac{1}{\eps} \|  \mathrm{p}_{\vartheta} \P^\perp \widehat f^\eps \|_{L^1_\xi L^2_t H^{s,*}_v} 
+ \left\| \mathrm{p}_{\vartheta} \frac{|\xi|}{\la \xi \ra} \P \widehat f^\eps \right\|_{L^1_\xi L^2_t L^2_v} \\
&\lesssim  \| \widehat f^\eps_0 \|_{L^1_\xi L^2_v }
+ \| \widehat f^\eps_0 \|_{L^p_\xi L^2_v } 
+ \| \mathrm{p}_{\vartheta} \widehat f^\eps \|_{L^1_\xi L^\infty_t L^2_v} \left( \| \widehat f^\eps_0 \|_{L^1_\xi L^2_v }
+ \| \widehat f^\eps_0 \|_{L^p_\xi L^2_v } \right). 
\end{aligned}
\end{equation*}
Since $\| \widehat f^\eps_0 \|_{L^1_\xi L^2_v } + \| \widehat f^\eps_0 \|_{L^p_\xi L^2_v } \le \eta_0$ is small enough, the last term in the right-hand side can be absorbed into the left-hand side, which thus concludes the proof of the decay estimate~\eqref{eq:theo1bis:decay_hard} in Theorem~\ref{theo:boltzmann}--(2).

\subsubsection{Decay for soft potentials}\label{sec:proof_theo_main1bis_decay_soft}
Let $0 < \vartheta < \frac{3}{2}(1-\frac{1}{p})$.
Let $f^\eps$ be the solution to \eqref{eq:feps} constructed in Theorem~\ref{theo:boltzmann}--(2) associated to the initial data $f^\eps_0$ with $\ell > \vartheta |\gamma+2s|$. Arguing as above, using Proposition~\ref{prop:estimate_Ueps_decay_soft_wholespace} and Proposition~\ref{prop:estimate_Ueps_regularization_soft_wholespace} we obtain
\begin{equation*}
\begin{aligned}
\| \mathrm{p}_{\vartheta} \widehat f^\eps \|_{L^1_\xi L^\infty_t L^2_v} 
&+ \frac{1}{\eps} \|  \mathrm{p}_{\vartheta} \P^\perp \widehat f^\eps \|_{L^1_\xi L^2_t H^{s,*}_v} 
+ \left\| \mathrm{p}_{\vartheta} \frac{|\xi|}{\la \xi \ra} \P \widehat f^\eps \right\|_{L^1_\xi L^2_t L^2_v} \\
&\lesssim  \| \widehat f^\eps_0 \|_{L^1_\xi L^2_v}
+ \|  \widehat f^\eps \|_{L^1_\xi L^\infty_t L^2_v (\la v \ra^\ell)} 
+ \|  \widehat f^\eps \|_{L^p_\xi L^\infty_t L^2_v} 
+ \| \mathrm{p}_{\vartheta} \widehat \Gamma (f^\eps , f^\eps) \|_{L^1_\xi L^2_t (H^{s,*}_v)'}. 
\end{aligned}
\end{equation*}
For the nonlinear term above, we argue as in Section~\ref{sec:proof_theo_main1bis_decay_hard} so that
$$
\begin{aligned}
&\| \mathrm{p}_{\vartheta}  \widehat \Gamma (f^\eps , f^\eps) \|_{L^1_\xi L^2_t (H^{s,*}_v)'} \\
&\quad
\lesssim \| \mathrm{p}_{\vartheta} \widehat f^\eps \|_{L^1_\xi L^\infty_t L^2_v} \left( 
\| \P^\perp \widehat f^\eps \|_{L^1_\xi L^2_t H^{s,*}_v}
+ \left\|  \frac{|\xi|}{\la \xi \ra}\widehat \P f^\eps \right\|_{L^1_\xi L^2_t L^2_v} 
+ \left\|  \frac{|\xi|}{\la \xi \ra} \widehat \P f^\eps \right\|_{L^p_\xi L^2_t L^2_v}  \right).
\end{aligned}
$$
Therefore, using the estimate of Theorem~\ref{theo:boltzmann}--(2), we obtain
\begin{equation*}
\begin{aligned}
&\| \mathrm{p}_{\vartheta} \widehat f^\eps \|_{L^1_\xi L^\infty_t L^2_v} 
+ \frac{1}{\eps} \|  \mathrm{p}_{\vartheta} \P^\perp \widehat f^\eps \|_{L^1_\xi L^2_t H^{s,*}_v} 
+ \left\| \mathrm{p}_{\vartheta} \frac{|\xi|}{\la \xi \ra} \P \widehat f^\eps \right\|_{L^1_\xi L^2_t L^2_v} \\
&\qquad
\lesssim  \| \widehat f^\eps_0 \|_{L^1_\xi L^2_v (\la v \ra^\ell)}
+ \| \widehat f^\eps_0 \|_{L^p_\xi L^2_v (\la v \ra^\ell)} 
+ \| \mathrm{p}_{\vartheta} \widehat f^\eps \|_{L^1_\xi L^\infty_t L^2_v} \left( \| \widehat f^\eps_0 \|_{L^1_\xi L^2_v (\la v \ra^\ell)}
+ \| \widehat f^\eps_0 \|_{L^p_\xi L^2_v (\la v \ra^\ell)} \right). 
\end{aligned}
\end{equation*}
Since $\| \widehat f^\eps_0 \|_{L^1_\xi L^2_v (\la v \ra^\ell)} + \| \widehat f^\eps_0 \|_{L^p_\xi L^2_v (\la v \ra^\ell)} \le \eta_0$ is small enough, the last term in the right-hand side can be absorbed into the left-hand side, which thus concludes the proof of the decay estimate~\eqref{eq:theo1bis:decay_soft} in Theorem~\ref{theo:boltzmann}--(2).

\section{Well-posedness for the Navier-Stokes-Fourier system}\label{section Navier-Stokes-Fourier}

We start by considering the incompressible Navier-Stokes equation, that is, the first equation in \eqref{Navier-Stokes-Fourier system}.
We denote by $V$ the semigroup associated to the operator $\nu_1 \Delta_x$, and we also denote, for all $t \ge 0$ and $\xi \in \Omega'_\xi$,
$$
\widehat V (t,\xi) = \FF_{x} (V(t) \FF^{-1}_{x}) (\xi) = e^{- \nu_1 |\xi|^2 t}. 
$$
We shall obtain below boundedness and integrated-in-time regularization estimates for $V$ as well as for its integral in time against a source $\int_0^t V(t-\tau) S(\tau) \, \d \tau$.

\begin{prop}\label{prop:estimate_V}
Let $p \in [1, \infty]$. Let $u_0 \in \FF^{-1}_x (L^p_\xi)$ and suppose moreover that $u_0$ is mean-free in the torus case $\Omega_x = \T^3$.
Then \begin{align*}
\Vert    \widehat{V}(\cdot) \widehat u_0   \Vert_{L^p_\xi L^\infty_t} + \Vert   |\xi| \widehat{V}(\cdot) \widehat u_0  \Vert_{L^p_\xi L^2_t}  \lesssim \Vert    \widehat u_0  \Vert_{L^p_\xi},
\end{align*}
and moreover $V(t)u_0$ also is mean-free for all $t \ge 0$ in the torus case (that is it satisfies \eqref{eq:meanfree}).
\end{prop}

\begin{rem}\label{rem:prop:estimate_V}
Observe that, in the torus case $\Omega_x = \T^3$, one can replace $ |\xi| \widehat{V}(\cdot) \widehat u_0  $ in above estimate by $ \la \xi \ra \widehat{V}(\cdot) \widehat u_0 $ since $V(t)u_0$ is mean-free.
\end{rem}

\begin{proof}
Let $u(t) = V(t) u_0$, which satisfies
$$
\partial_t u = - \nu_1 \Delta_x u, \quad u_{|t=0} = u_0.
$$
We already observe that, in the torus case, the solution $u(t)$ is also mean-free, that is satisfies \eqref{eq:meanfree}. For all $\xi \in \Omega'_\xi$ we thus have
$$
\partial_t \widehat u(t,\xi) = - \nu_1 |\xi|^2 \widehat u(t,\xi), \quad \widehat u(\xi)_{|t=0} = \widehat u_0(\xi),
$$
thus for any $t \ge 0$ we have
\[
|\widehat{u}(t, \xi)|^2   +\int_0^t  |\xi|^2| \widehat{u}(\tau, \xi) |^2 \, \d \tau \lesssim |\widehat{u}_0(\xi)|^2.
\]
Taking the supremum in time and then taking the square-root of previous estimate yields
\[
\Vert  \widehat{ u } (\xi) \Vert_{L^\infty_t} +  \Vert | \xi | \widehat{ u } (\xi) \Vert_{L^2_t}   \lesssim  |\widehat{u}_0(\xi)|,
\]
and we conclude the proof by taking the $L^p_\xi$ norm.
\end{proof}

\begin{prop}\label{prop:estimate_V_regularization}
Suppose $p \in [1, \infty]$. Let $S=S(t, \xi)$ satisfies $   \langle \xi\rangle^{-1} \widehat{S}  \in L^p_\xi L^2_t $ and \eqref{eq:meanfree} in the torus case $\Omega_x = \T^3$, and 
$ |\xi|^{-1} \widehat{S} \in L^p_\xi L^2_t$  in the whole space case $\Omega_x = \R^3$. Denote
\[
u_S(t) =\int_0^t V(t-\tau) S(\tau) \, \d \tau.
\]
Then in the torus case we have
\begin{align*}
\Vert    \widehat{u}_S  \Vert_{L^p_\xi L^\infty_t} + \Vert   \la \xi \ra \widehat{u}_S  \Vert_{L^p_\xi L^2_t}  \lesssim \Vert   \la \xi \ra^{-1} \widehat{S}  \Vert_{L^p_\xi L^2_t}.
\end{align*}
and in the whole space case
\begin{align*}
\Vert    \widehat{u}_S  \Vert_{L^p_\xi L^\infty_t} + \Vert   |\xi| \widehat{u}_S  \Vert_{L^p_\xi L^2_t}  \lesssim \Vert   |\xi|^{-1} \widehat{S}  \Vert_{L^p_\xi L^2_t}.
\end{align*}

\end{prop}

\begin{proof}
We first observe that $u_S$ satisfies
\[
\partial_t u_S + \nu_1 \Delta_x u_S =S, \quad {u_S}_{|t=0} = 0.
\]
We only prove the whole space case, the case of the torus being similar by observing that $u_S$ is mean-free, that is verifies \eqref{eq:meanfree}.

For all $\xi \in \R^3$ and all $t \ge 0$ we have
$$
\partial_t \widehat u_S(t,\xi) + \nu_1 |\xi|^2 \widehat u_S(t,\xi) = \widehat S(t,\xi), \quad \widehat {u_S(\xi)}_{|t=0} = 0.
$$
We can compute
\[
\partial_t \frac 1 2 |\widehat{u}_S (t, \xi)|^2   + \nu_1 |\xi|^2| \widehat{u}_S( \xi) |^2  \le (\widehat{S} (\xi), \widehat{u}_S (\xi)),
\]
which implies, for all $t \ge 0$,
\[
|\widehat{u}_S(t, \xi)|^2   +  \int_0^t |\xi|^2| \widehat{u}_S( \tau, \xi) |^2 \, \d \tau \lesssim  \int_0^t ||\xi|^{-1}  S(\tau,\xi)|^2 \, \d \tau.
\]
Taking the supremum in time, then taking the square-root of the estimate, and taking the $L^P_\xi$ norm, the proof is thus finished. 
\end{proof}

We now obtain bilinear estimates for the operator $Q_{\mathrm{NS}}$ defined in \eqref{eq:def:QNS}.

\begin{lem}\label{lem:estimate_QNS}
Let $p \in [1,\infty]$. Let $u, v \in \FF^{-1}_x (L^1_\xi L^\infty_t  \cap L^p_\xi L^\infty_t)$, then
\begin{equation}
\label{estimate Q NS}
 \|   |\xi|^{-1} \widehat{Q}_{\mathrm{NS}} (v, u) \|_{L^p_\xi L^2_t}   \lesssim  \| v \|_{L^p_\xi L^2_t}  \|  u  \|_{L^1_{\xi} L^\infty_t} ,
\end{equation}
and also
\begin{equation}
\label{estimate Q NS-2}
 \|   |\xi|^{-1} \widehat{Q}_{\mathrm{NS}} (v, u) \|_{L^p_\xi L^2_t}   \lesssim  \| v \|_{L^p_\xi L^\infty_t}  \|  u  \|_{L^1_{\xi} L^2_t} .
\end{equation}
\end{lem}

\begin{proof}
From the definition of $Q_{\mathrm{NS}}$, we first observe that for all $\xi \in \Omega'_\xi$ and $j \in \{ 1 , 2 , 3 \}$ we have
$$
\begin{aligned}
\widehat Q_{\mathrm{NS}} (v,u)^j (\xi) 
&= -\frac12 \sum_{k=1}^3 \mathrm{i} \xi_k \left\{   \FF_x (v^j u^k) (\xi) - \frac{1}{|\xi|^2} \sum_{l=1}^3  \xi_j \xi_l \FF_x (v^l u^k) (\xi)  \right\} \\ 
&\quad 
-\frac12 \sum_{k=1}^3 \mathrm{i} \xi_k \left\{   \FF_x (u^j v^k) (\xi) - \frac{1}{|\xi|^2} \sum_{l=1}^3  \xi_j \xi_l \FF_x (u^l v^k) (\xi)  \right\} . 
\end{aligned}
$$
We obtain
$$
|\widehat Q_{\mathrm{NS}} (v,u) (\xi)| \lesssim |\xi| \int_{\Omega'_\eta} |\widehat v (\eta)|  |\widehat u (\xi - \eta)| \, \d \eta ,
$$
thus by Minkowski's inequality and then H\"older's inequality
$$
\begin{aligned}
\| |\xi|^{-1} \widehat Q_{\mathrm{NS}} (v,u) (\xi) \|_{L^2_t}
&\lesssim  \int_{\Omega'_\eta} \left( \int_0^\infty |\widehat v (t,\eta)|^2  |\widehat u (t,\xi - \eta)|^2 \, \d t \right)^{1/2} \d \eta \\
&\lesssim  \int_{\Omega'_\eta}  \| \widehat v (\eta) \|_{L^2_t}  \|\widehat u (\xi - \eta) \|_{L^\infty_t} \, \d \eta .
\end{aligned}
$$
We then conclude the proof of~\eqref{estimate Q NS} by taking the $L^p_\xi$ norm above and applying Young's convolution inequality. The proof of~\eqref{estimate Q NS-2} can be obtained in a similar way, by exchanging the role of $u$ and $v$ when applying H\"older's inequality with respect to the time variable.
\end{proof}

\subsection{Global existence in the torus $\Omega_x=\T^3$}\label{sec:nsf_torus}

We shall construct mild solutions to the first equation in \eqref{Navier-Stokes-Fourier system}, namely
\begin{equation}\label{eq:NS_u_mild}
u(t) = V(t) u_0 + \int_0^t V(t-\tau) Q_{\mathrm{NS}} (u(\tau) , u(\tau)) \, \d \tau.
\end{equation}
We define the space
$$
\mathscr{X} = \left\{ u \in \FF^{-1}_x(L^1_\xi L^\infty_t \cap L^1_\xi (\la \xi \ra) L^2_t) \mid u \text{ satisfies } \eqref{eq:meanfree}, \; \| u \|_{\mathscr{X}} <\infty \right\},
$$
with
$$
\| u \|_{\mathscr{X}}
:= \| \widehat u \|_{L^1_\xi L^\infty_t} +  \|   \la \xi \ra \widehat u \|_{L^1_\xi L^2_t}.
$$
Let $u_0 \in \FF^{-1}_x (L^1_\xi)$ be mean-free and
\[
\| \widehat {u}_0 \|_{L^1_\xi} \le \eta_1.
\]
Consider the map $\Phi : \mathscr{X} \to \mathscr{X}$, $u \mapsto \Phi[u]$ defined by, for all $t \ge 0$ ,
\begin{equation}\label{eq:Phiu}
\Phi[u](t) = V (t) u_0 + \int_0^t  V (t-\tau)  Q_{\mathrm{NS}}  ( u(\tau) , u(\tau)) \, \d \tau.
\end{equation}
thus, for all $\xi \in \Z^3$,
\begin{equation}\label{eq:Phiu_fourier}
\widehat\Phi[u] (t,\xi) = \widehat V(t,\xi) \widehat u_0(\xi) + \int_0^t \widehat V (t-\tau,\xi) \widehat Q_{\mathrm{NS}}  ( u(\tau) , u(\tau)) (\xi) \, \d \tau.
\end{equation}
For the first term we have from Proposition~\ref{prop:estimate_V} that
\[
\Vert \widehat V(t,\xi) \widehat u_0(\xi)  \Vert_{\mathscr{X}} \le C_0 \Vert \widehat{u}_0 \Vert_{L^1_\xi} ,
\]
and by Proposition~\ref{prop:estimate_V_regularization} we have
\begin{align*}
\left \Vert \int_0^t \widehat V (t-\tau,\xi) \widehat Q_{\mathrm{NS}}  ( u(\tau) , u(\tau)) (\xi) \, \d \tau \right\Vert_{\mathscr{X}} 
&\lesssim   \Vert   |\xi|^{-1} \widehat Q_{\mathrm{NS}}  ( u , u) \Vert_{L^1_\xi L^2_t}  \\
&\lesssim   \| \widehat u \|_{L^1_\xi L^2_t} \| \widehat u \|_{L^1_\xi L^\infty_t} \\
&\lesssim   \| \la \xi \ra \widehat u \|_{L^1_\xi L^2_t} \| \widehat u \|_{L^1_\xi L^\infty_t} \\
&\lesssim \Vert u \Vert_{\mathscr{X}}^2,
\end{align*}
where we have used Lemma~\ref{lem:estimate_QNS}. Thus we obtain 
\[
\Vert \Phi [u] \Vert_{\mathscr{X}} \lesssim C_0 \Vert \widehat{u}_0 \Vert_{L^1_\xi} + C_1 \Vert u \Vert_{\mathscr{X}}^2.
\]
Moreover for $u, v \in \mathscr{X}$ we can also compute, using again Proposition~\ref{prop:estimate_V_regularization} and Lemma~\ref{lem:estimate_QNS}, that
\begin{align*}
&\left \Vert \int_0^t \widehat V (t-\tau,\xi) \widehat Q_{\mathrm{NS}}  ( (u-v)(\tau) , v(\tau)) (\xi) \, \d \tau \right\Vert_{\mathscr{X}}  + \left \Vert \int_0^t \widehat V (t-\tau,\xi) \widehat Q_{\mathrm{NS}}  ( u(\tau) , (u-v)(\tau)) (\xi) \, \d \tau \right\Vert_{\mathscr{X}}
\\
&\qquad 
\lesssim  \Vert   |\xi|^{-1} \widehat Q_{\mathrm{NS}}  ( u-v,v)  \Vert_{L^1_\xi L^2_t}   + \Vert   |\xi|^{-1} \widehat Q_{\mathrm{NS}}  ( u , u-v)  \Vert_{L^1_\xi L^2_t}  
\\
&\qquad 
\lesssim 
\| \widehat u - \widehat v \|_{L^1_\xi L^\infty_t} \| \widehat v \|_{L^1_\xi L^2_t}
+  \| \widehat u \|_{L^1_\xi L^2_t} \| \widehat u - \widehat v \|_{L^1_\xi L^\infty_t}.
\end{align*}
Therefore there is $C_1>0$ such that
$$
\| \Phi[u] - \Phi[v] \|_{\mathscr X} \le C_1 ( \Vert u \Vert_{\mathscr{X}} + \Vert v \Vert_{\mathscr{X}}) \Vert u-v \Vert_{\mathscr{X}}.
$$

Gathering the two inequalities and arguing as in Sections~\ref{sec:proof_theo_main1_existence} and \ref{sec:proof_theo_main1_uniqueness}, we can construct a global unique solution $u \in \mathscr{X}$ to the equation \eqref{eq:NS_u_mild} if $\eta_1 >0$ is small enough, which moreover satisfies
\[
\Vert u \Vert_{\mathscr{X}} \lesssim \Vert \widehat{u}_0 \Vert_{L^1_\xi} .
\]
Once $u$ have been constructed, we can then argue in a similar and even simpler way in order to construct a global unique mild solution $\theta$ for the second equations in \eqref{Navier-Stokes-Fourier system} if $\eta_1 >0$ is small enough, namely
$$
\theta(t) = \overline V(t) \theta_0 + \int_0^t \overline V (t-\tau) \left[- \Div_x (u(\tau) \theta (\tau) ) \right]  \d \tau ,
$$
where $\overline V$ denotes the semigroup associated to the operator $\nu_2 \Delta_x$, and which satisfies moreover
\[
\Vert \theta \Vert_{\mathscr{X}} \lesssim \Vert \widehat{u}_0 \Vert_{L^1_\xi} + \Vert \widehat{\theta}_0 \Vert_{L^1_\xi} .
\]
We finally obtain the solution $\rho$ by using the last equation in \eqref{Navier-Stokes-Fourier system} and observing that we consider mean-free solutions, so that $\widehat \rho(t,0) = \widehat \theta(t,0) = 0$. This completes the proof of Theorem~\ref{theo:NSF}--(1).

\subsection{Global existence in the whole space $\Omega_x=\R^3$}
Similarly as before we define the space, recalling that $p \in (3/2,+\infty]$,
\[
\mathscr{Y} = \left\{ u \in \FF^{-1}_x(L^1_\xi L^\infty_t  \cap L^1_\xi(|\xi|) L^2_t ) \cap  \FF^{-1}_x(L^p_\xi L^\infty_t  \cap L^p_\xi(|\xi|) L^2_t )  \mid \| u \|_{\mathscr{Y}} <\infty \right\},
\]
with
\[
\| u \|_{\mathscr{Y}} := \| \widehat u \|_{L^1_\xi L^\infty_t} +  \|   |\xi |\widehat u \|_{L^1_\xi L^2_t} +  \| \widehat u \|_{L^p_\xi L^\infty_t } + \|   |\xi |\widehat u \|_{L^p_\xi L^2_t} .
\]
Let $u_0 \in \FF^{-1}_x (L^1_\xi \cap L^p_\xi)$ satisfy
\[
\| \widehat {u}_0 \|_{L^1_\xi} + \| \widehat {u}_0 \|_{L^p_\xi} \le \eta_1,
\]
and consider the map $\Phi : \mathscr{Y} \to \mathscr{Y}$, $u \mapsto \Phi[u]$ defined by \eqref{eq:Phiu}, in particular \eqref{eq:Phiu_fourier} is verified for all $\xi \in \R^3$.

For the first term in \eqref{eq:Phiu_fourier} we have from Proposition~\ref{prop:estimate_V} that
\[
\Vert \widehat V(t,\xi) \widehat u_0(\xi)  \Vert_{\mathscr{Y}} \le C_0 (\Vert \widehat{u}_0 \Vert_{L^1_\xi} +\Vert \widehat{u}_0 \Vert_{L^p_\xi}).
\]
Furthermore, by Proposition~\ref{prop:estimate_V_regularization} we have
\begin{align*}
\left \Vert \int_0^t \widehat V (t-\tau,\xi) \widehat Q_{\mathrm{NS}}  ( u(\tau) , u(\tau)) (\xi) \, \d \tau \right\Vert_{\mathscr{Y}} 
&\lesssim   \Vert   |\xi|^{-1} \widehat Q_{\mathrm{NS}}  ( u , u) \Vert_{L^1_\xi L^2_t} +  \Vert   |\xi|^{-1} \widehat Q_{\mathrm{NS}}  ( u , u) \Vert_{L^p_\xi L^2_t}  \\
&\lesssim   \| \widehat u \|_{L^1_\xi L^2_t} \left( \| \widehat u \|_{L^1_\xi L^\infty_t} + \| \widehat u \|_{L^p_\xi L^\infty_t} \right).
\end{align*}
where we have used Lemma~\ref{lem:estimate_QNS}. We now observe that
$$
\| \widehat u \|_{L^1_\xi L^2_t} \lesssim \| \mathbf{1}_{|\xi|\ge 1} \widehat u \|_{L^1_\xi L^2_t} + \| \mathbf{1}_{|\xi| < 1}\widehat u \|_{L^1_\xi L^2_t},
$$
and for the first term we easily have
$$
\| \mathbf{1}_{|\xi|\ge 1} \widehat u \|_{L^1_\xi L^2_t} \lesssim \| |\xi| \widehat u \|_{L^1_\xi L^2_t}.
$$
For the second term we use H\"older's inequality to obtain
$$
\| \mathbf{1}_{|\xi| < 1}\widehat u \|_{L^1_\xi L^2_t}
\lesssim \| \mathbf{1}_{|\xi| < 1} |\xi|^{-1} \|_{L^{p'}_\xi} \| \mathbf{1}_{|\xi| < 1} |\xi| \widehat u \|_{L^p_\xi L^2_t}
\lesssim \| |\xi| \widehat u \|_{L^p_\xi L^2_t},
$$
where we have used that $\| \mathbf{1}_{|\xi| < 1} |\xi|^{-1} \|_{L^{p'}_\xi}<\infty$ since $p>3/2$. Therefore we get
\begin{equation}\label{eq:u_L1_Lp}
\| \widehat u \|_{L^1_\xi L^2_t} \lesssim \| |\xi| \widehat u \|_{L^1_\xi L^2_t} + \| |\xi| \widehat u \|_{L^p_\xi L^2_t}.
\end{equation}
Gathering previous estimates, we have hence obtained
\[
\Vert \Phi [u] \Vert_{\mathscr{Y}} \le C_0 \left( \Vert \widehat{u}_0 \Vert_{L^1_\xi} +  \Vert \widehat{u}_0 \Vert_{L^p_\xi} \right) + C_1 \Vert u \Vert_{\mathscr{Y}}^2.
\]
Moreover for $u, v \in \mathscr{Y}$ we can also compute, using again Proposition~\ref{prop:estimate_V_regularization} and Lemma~\ref{lem:estimate_QNS}, that
\begin{align*}
&\left \Vert \int_0^t \widehat V (t-\tau,\xi) \widehat Q_{\mathrm{NS}}  ( (u-v)(\tau) , v(\tau)) (\xi) \, \d \tau \right\Vert_{\mathscr{X}}  + \left \Vert \int_0^t \widehat V (t-\tau,\xi) \widehat Q_{\mathrm{NS}}  ( u(\tau) , (u-v)(\tau)) (\xi) \, \d \tau \right\Vert_{\mathscr{X}}
\\
&\qquad 
\lesssim  \Vert   |\xi|^{-1} \widehat Q_{\mathrm{NS}}  ( u-v,v)  \Vert_{L^1_\xi L^2_t}   + \Vert   |\xi|^{-1} \widehat Q_{\mathrm{NS}}  ( u , u-v)  \Vert_{L^1_\xi L^2_t} \\
&\qquad \quad 
+\Vert   |\xi|^{-1} \widehat Q_{\mathrm{NS}}  ( u-v,v)  \Vert_{L^p_\xi L^2_t}   + \Vert   |\xi|^{-1} \widehat Q_{\mathrm{NS}}  ( u , u-v)  \Vert_{L^p_\xi L^2_t}  
\\
&\qquad 
\lesssim 
\| \widehat u - \widehat v \|_{L^1_\xi L^\infty_t} \| \widehat v \|_{L^1_\xi L^2_t}
+  \| \widehat u \|_{L^1_\xi L^2_t} \| \widehat u - \widehat v \|_{L^1_\xi L^\infty_t}
+ \| \widehat u - \widehat v \|_{L^p_\xi L^\infty_t} \| \widehat v \|_{L^1_\xi L^2_t}
+  \| \widehat u \|_{L^1_\xi L^2_t} \| \widehat u - \widehat v \|_{L^p_\xi L^\infty_t}
\\
&\qquad 
\lesssim \left( \| \widehat u \|_{L^1_\xi L^2_t} + \| \widehat v \|_{L^1_\xi L^2_t} \right) \left( \| \widehat u - \widehat v \|_{L^1_\xi L^\infty_t}
 + \| \widehat u - \widehat v \|_{L^p_\xi L^\infty_t}
\right).
\end{align*}
Using inequality \eqref{eq:u_L1_Lp} we therefore get, for some constant $C_1>0$,
$$
\| \Phi[u] - \Phi[v] \|_{\mathscr Y} \le C_1 \left( \| u \|_{\mathscr Y} + \| v \|_{\mathscr Y}\right) \| u-v \|_{\mathscr Y}.
$$
Gathering these two inequalities together, the proof of Theorem~\ref{theo:NSF}--(2) is completed by arguing as in Section~\ref{sec:nsf_torus} above.

\section{Hydrodynamic limit}\label{section hydrodynamical limit}

Recalling that the semigroup $U^\eps$ is defined in \eqref{eq:def:Ueps}, and also $\widehat U^\eps$ in \eqref{eq:def:hatUeps}, we also define, for all $t \ge 0$,
\begin{equation}\label{definition Psi epsilon}
\Psi^\eps [f,g] (t) = \frac{1}{\eps} \int_0^t U^\eps (t-\tau) \Gamma_{\mathrm{sym}} ( f(\tau) , g(\tau)) \, \d \tau, 
\end{equation}
as well as its Fourier transform in space, for all $\xi \in \Omega'_\xi$,
\begin{equation}\label{definition Psi epsilon fourier}
\widehat \Psi^\eps [f,g] (t, \xi) = \frac{1}{\eps} \int_0^t \widehat U^\eps (t-\tau,\xi) \widehat \Gamma_{\mathrm{sym}} ( f(\tau) , g(\tau)) (\xi) \, \d \tau,
\end{equation}
where we recall that $\Gamma_{\mathrm{sym}}(f,g)$ is the symmetrized version of $\Gamma(f,g)$ defined in \eqref{eq:def:Gamma_sym}.

\subsection{Estimates on $\widehat U^\eps$}

We denote that $0 \le \chi \le 1$ is a fixed compactly supported function of $B_{1}$ equal to one on $B_{\frac 1 2} $, where $B_R$ is the ball with radius $R$ centered at zero. 

Arguing as in \cite{BU,GT} but using the spectral estimates of \cite{YY,YY2} for the non-cutoff Boltzmann equation, we then have:
\begin{lem}\label{basic property U epsilon}
There exist $\kappa>0$ such that one can write 
\[
U^\eps(t) = \sum_{j=1}^4 U_{j}^\eps  (t) +U^{\eps \#} (t),
\] 
with 
\[
\widehat{U}_j^\eps (t, \xi) : = \widehat{U}_j(\frac t {\eps^2} , \eps \xi), \quad \widehat{U}^{\eps \#} (t, \xi) = \widehat{U}
^{\#}  (\frac t {\eps^2} , \eps \xi),
\]
where we have the following properties:
\begin{enumerate}

\item For $1 \le j \le 4$, 
\[
\widehat{U}_j(t, \xi) =\chi(\frac {|\xi|} {\kappa})e^{t \lambda_j(\xi)}P_j(\xi),
\] 
with $\lambda_j$ satisfying 
\[
\lambda_j(\xi) = i \alpha_j(\xi) -\beta_j |\xi|^2 +\gamma_j(|\xi|),
\]
with
\[
\alpha_1>0,\quad \alpha_2<0, \quad \alpha_3=\alpha_4 =0,\quad \beta_j >0,
\]
and
\[
\gamma_j(|\xi|) = O(|\xi|^3) ,\quad \hbox{as} \quad \xi \to 0, \quad \gamma_j(\xi) \le \beta_j |\xi|^2 /2,\quad \forall |\xi| \le \kappa,
\]
as well as
\[
P_j(\xi) = P_j^0 (\tfrac{\xi} {|\xi|} ) + |\xi|P_j^1 (\tfrac{\xi} {|\xi|} ) +|\xi|^2 P_j^2(\xi),
\]
with $P_j^n$ bounded linear operators on $L^2_v$ with operator norms uniformly bounded for $|\xi| \le \kappa$.

\item We also have that the orthogonal projector $\P$ onto $\Ker L$ satisfies
\[
\P =\sum_{j=1}^4 P_j^0 (\tfrac {\xi} {|\xi|} ) .
\]
Moreover $ P_j^0 (\frac {\xi} {|\xi|} )$,  $P_j^1 (\frac {\xi} {|\xi|})$ and $P_j^2(\xi)$ are bounded from $L^2_v$ to $L^2_v(\langle v \rangle^l)$ uniformly in $|\xi| \le \kappa$ for all $l \ge 0$.

\item In the hard potentials case $\gamma + 2s \ge 0$, for all $t \ge 0$ and all $\xi \in \R^3$ there holds, for any $\ell \ge 0$,
\begin{equation}
\label{U hard estimate}
\Vert \widehat{U}^{\eps \#} (t,\xi) \widehat f(\xi) \Vert_{L^2_v (\langle v \rangle^\ell) } \le   Ce^{- \lambda_1 \frac {t} {\eps^2}} \Vert \widehat f(\xi) \Vert_{L^2_v (\langle v \rangle^\ell) },
\end{equation}
for any $f$ satisfying moreover \eqref{eq:normalization1} in the torus case, where $\lambda_1, C>0$ are independent of $t, \xi, \eps$.

\item In the soft potential case $\gamma + 2s < 0$, for all $t \ge 0$ and all $\xi \in \R^3$ there holds, for any $k, \ell \ge 0$,
\begin{equation}
\label{U soft estimate}
\Vert \widehat{U}^{\eps \#} (t,\xi) \mathbf{P}^\perp \widehat f(\xi)\Vert_{L^2_v (\langle v \rangle^k)} \le C \left( 1+ \frac{t}{\eps^2} \right)^{ -\frac{\ell}{|\gamma+2s|}} \Vert \widehat f(\xi) \Vert_{L^2_v (\langle v \rangle^{k+\ell})},
\end{equation}
for any $f$ satisfying moreover \eqref{eq:normalization1} in the torus case, where $C>0$ is independent of $t, \xi, \eps$.

\end{enumerate}

\end{lem}

\begin{proof}
The proof is the same as in \cite[Lemma 5.10]{CRT}. For the soft potentials case, we need to replace the use of \cite[Theorem 3.2 and Remark 5.2]{YY} in the proof by \cite[Theorem 1.1 and Section 4]{YY2}, in particular the decay estimate~\eqref{U soft estimate} comes from \cite[Equation~(2.46)]{YY2} and the fact that $B_0(\xi) \P^\perp =B(\xi) \P^\perp$, where $B_0(\xi)$ and $B(\xi)$ are defined in \cite[Equation~(1.18)]{YY2} and satisfy $B_0(\xi) =B(\xi) -\P$.
\end{proof}

Denoting 
\[
\widetilde{P}_j \left(\xi, \frac {\xi} {|\xi|} \right) := P^1_j \left(\frac {\xi} {|\xi|}\right) +\xi P^2_j(\xi),
\]
for $1 \le j \le 4$, we can further split $\widehat{U}_j^\eps $ into four parts (one main part and three remainder terms):
\begin{equation}\label{eq:decomposition_Ujeps}
U_j^\eps = U_{j0}^\eps +  U_{j0}^{\eps \#}  + U_{j1}^\eps + U_{j2}^\eps,
\end{equation}
where 
\begin{align*}
&\widehat{U}_{j0}^\eps (t, \xi) =e^{i \alpha_j |\xi| \frac t \eps -\beta_j t |\xi|^2} P_j^0\left(\frac {\xi} {|\xi|}\right),
\\
&\widehat{U}_{j0}^{\eps \#} (t, \xi) = \left( \chi   \left(\frac {\eps|\xi|} {\kappa} \right) -1 \right) e^{i \alpha_j |\xi| \frac t \eps -\beta_j t |\xi|^2} P_j^0\left(\frac {\xi} {|\xi|}\right),
\\
&\widehat{U}_{j1}^\eps (t, \xi) =\chi   \left(\frac {\eps|\xi|} {\kappa} \right)    e^{i \alpha_j |\xi| \frac t \eps -\beta_j t |\xi|^2} \left (e^{t \frac {\gamma_j(\eps |\xi|)  }     { \eps^2}  }   -1 \right)P_j^0 \left(\frac {\xi} {|\xi|}\right),
\\
&\widehat{U}_{j2}^\eps (t, \xi) =\chi   \left(\frac {\eps|\xi|} {\kappa} \right)    e^{i \alpha_j |\xi| \frac t \eps -\beta_j t |\xi|^2}   e^{t \frac {\gamma_j(\eps |\xi|)  }    { \eps^2}  }  \eps |\xi | \widetilde{P}_j \left(\eps \xi, \frac {\xi} {|\xi|} \right).
\end{align*}
In particular we observe that $\widehat U^\eps_{30}$ and $\widehat U^\eps_{40}$ are independent of $\eps$, so that we define
\begin{equation}\label{eq:def:hatUxi}
\widehat U(t,\xi) := \widehat U^\eps_{30}(t,\xi) + \widehat U^\eps_{40}(t,\xi),
\end{equation}
which is then independent of $\eps$. We finally define
\begin{equation}\label{eq:def:U(t)}
U(t) = \FF_x^{-1} \widehat U(t) \FF_x.
\end{equation}

We say that a function $f=f(x,v) \in \Ker L$ is well-prepared if
$$
f(x,v) = \left\{ \rho[f](x)  + u[f](x) \cdot v  + \theta[f](x) \frac{(|v|^2-3)}{2} \right\} \sqrt{\mu}(v)
$$
with 
$$
\nabla_x \cdot u[f] = 0 \quad \text{and} \quad \rho[f] + \theta[f] = 0,
$$
where we recall that $\rho[f], u[f], \theta[f]$ are defined in \eqref{eq:def:rho_u_theta}

\begin{lem}\label{Basic properties U t}(\cite{GT}, Proposition A.3) 
We have that $U(0)$ is the projection on the subset of $\Ker L$ consisting of functions $f$ that are well-prepared. 
We also have
\[
U(t)f = U(t) U(0) f, \quad \forall t \ge 0 ,
\]
and
\[
\nabla_x \cdot  u[f] =0 \quad \hbox{and} \quad \rho[f] + \theta[f] = 0 \quad\Rightarrow\quad P_j^0( \tfrac {\xi} {|\xi|} ) f = 0 , \quad j=1, 2. 
\]
\end{lem}

The following lemma studies the limit of $U^ \eps(t)$ as $\eps$ goes to 0. 

\begin{lem}\label{convergence U epsilon U}
Let $f=f(x,v) \in \Ker L$ be well-prepared, then we have
\[
\Vert  ( \widehat U^\eps (\cdot) - \widehat U(\cdot) ) \widehat f \Vert_{L^1_\xi L^\infty_t L^2_v} \lesssim \Vert     \widehat f  \Vert_{L^1_\xi   L^2_v   },
\]
and
\[
\Vert  (\widehat U^\eps (\cdot) - \widehat U(\cdot) ) \widehat f \Vert_{L^1_\xi L^\infty_t L^2_v   } \lesssim  \eps \Vert  |\xi|  \widehat f  \Vert_{L^1_\xi  L^2_v  }.
\]

\end{lem}

\begin{proof}
The proof follows the idea of  \cite[Lemma 3.5]{GT}, that we shall adapt since we work in different functional spaces.

First of all we observe that from the decomposition of $U^\eps$ in \eqref{eq:decomposition_Ujeps} we can write, for all $t \ge 0$ and $\xi \in \Omega'_\xi$,
$$
\begin{aligned}
\widehat U^\eps (t,\xi) \widehat f(\xi) - \widehat U(t,\xi) \widehat f(\xi) 
&= \sum_{j=1}^4 \left\{ \widehat U^{\eps \#}_{j0} (t,\xi) \widehat f(\xi) 
+ \widehat U^\eps_{j1} (t,\xi) \widehat f(\xi)  
+ \widehat U^\eps_{j2} (t,\xi) \widehat f(\xi) \right\} \\
&\quad
+ \sum_{j=1}^2 \widehat U^{\eps}_{j0} (t,\xi) \widehat f(\xi)
+ \widehat U^{\eps \#} (t,\xi) \widehat f(\xi),
\end{aligned}
$$ 
and we shall estimate each term separately below.

We first compute the term $U_{j m}^\eps(t) f $ for $j=1,2, 3, 4$ and $m=1, 2$. For the $U_{j1} $ term, using Lemma~\ref{basic property U epsilon} together with  the inequality $|e^a-1| \le |a| e^{|a|}$ for any $a \in \R^+$, we have
\begin{equation}\label{small epsilon xi estimate 1}
\chi \left(\frac {\eps|\xi|} {\kappa} \right)    e^{ -\beta_j t |\xi|^2} \left| e^{t \frac {\gamma_j(\eps |\xi|)  }     { \eps^2}  }   -1 \right| \le \chi \left(\frac {\eps|\xi|} {\kappa} \right)    e^{ -\frac {\beta_j} 2 t |\xi|^2}  t \eps |\xi|^3 \lesssim  \chi \left(\frac {\eps|\xi|} {\kappa} \right)   \eps |\xi| \lesssim \min \{ 1, \eps |\xi| \}.
\end{equation}
Then we can compute, for all $t \ge 0$ and $\xi \in \R^3$,
\begin{align*}
\Vert \widehat{U}_{j1}^\eps (t, \xi) \widehat{f}(\xi) \Vert_{L^2_v} \le &\chi   \left(\frac {\eps|\xi|} {\kappa} \right)    \left| e^{i \alpha_j |\xi| \frac t \eps -\beta_j t |\xi|^2} \right|\left (e^{t \frac {\gamma_j(\eps |\xi|)  }     { \eps^2}  }   -1 \right)  \Vert P_j^0(\tfrac {\xi} {|\xi|})  \widehat{f} (\xi) \Vert_{L^2_v}
\\
\lesssim&  \min \{ 1, \eps |\xi| \} \Vert \hat{f}(\xi) \Vert_{L^2_v}.
\end{align*}
For the $\widehat{U}_{j2}^\eps (t, \xi)$ term we have
\begin{align*}
\Vert \widehat{U}_{j2}^\eps (t, \xi) \widehat{f}(\xi) \Vert_{L^2_v} \le &\chi   \left(\frac {\eps|\xi|} {\kappa} \right)   \left |  e^{i \alpha_j |\xi| \frac t \eps -\beta_j t |\xi|^2}    e^{t \frac {\gamma_j(\eps |\xi|)  }    { \eps^2}  }  \right|  \eps |\xi | \Vert \tilde{P}_j (\eps \xi, \tfrac{\xi} {|\xi|} )  \widehat{f} (\xi) \Vert_{L^2_v}
\\
\lesssim&  \min \{ 1, \eps |\xi| \} \Vert \hat{f}(\xi) \Vert_{L^2_v}.
\end{align*}
For the term $\widehat{U}^{\eps \#}_{j0} (t, \xi)$, using the fact that
\begin{equation}\label{small epsilon xi estimate 2}
\left |\chi   \left(\frac {\eps|\xi|} {\kappa} \right)  -  1\right| \lesssim \min \{ 1, \eps |\xi| \},
\end{equation}
we have
\begin{align*}
\Vert \widehat{U}_{j0}^{\eps \#} (t, \xi) \widehat{f}(\xi) \Vert_{L^2_v} \le &\left( \chi   \left(\frac {\eps|\xi|} {\kappa} \right) -1 \right) \left | e^{i \alpha_j |\xi| \frac t \eps -\beta_j t |\xi|^2}   \right|   \Vert P_j^0(\tfrac {\xi} {|\xi|}) \widehat{f} (\xi) \Vert_{L^2_v}
\\
\lesssim&  \min \{ 1, \eps |\xi| \} \Vert \hat{f}(\xi) \Vert_{L^2_v}.
\end{align*}
Taking the $L^1_\xi L^\infty_t$ norm in both sides yields, for all $j = 1,2,3,4$,
\[
\Vert \widehat{U}_{j1}^\eps (\cdot) \widehat{f}\Vert_{L^1_\xi L^\infty_t L^2_v} + \Vert \widehat{U}_{j2}^\eps (\cdot) \widehat{f}\Vert_{L^1_\xi L^\infty_t L^2_v}+ \Vert \widehat{U}_{j0}^{\eps \#}  (\cdot) \widehat{f}\Vert_{L^1_\xi L^\infty_t L^2_v}\lesssim  \min \{\Vert \widehat{f}\Vert_{L^1_\xi L^2_v}, \eps \Vert |\xi| \widehat{f}\Vert_{L^1_\xi L^2_v} \}.
\]
By Lemma \ref{Basic properties U t} we have, if $f \in \Ker L$ is a well-prepared data, then 
\[
U^{\eps}_{10} f + U^{\eps}_{20} f  = 0.
\]
Finally we compute the term $U^{\eps \#} (t, \xi) $, noticing that 
\[
\widehat{U}^{\eps \#} (t, \xi)  \widehat{f}(\xi, v) = \widehat{U}^{\eps} (t, \xi)\widehat{U}^{\eps \#}(0, \xi)  \widehat{f}(\xi, v) = \widehat{U}^\eps (t, \xi)  \left (   1-  \chi   \left(\frac {\eps|\xi|} {\kappa} \right)  \sum_{j=1}^4  P_j(\eps \xi )  \right) \widehat{f} (\xi, v).
\]
Since $f$ belongs to $\Ker L$, we have
\[
\widehat{U}^{\eps \#} (t, \xi)  \widehat{f}(\xi, v) = \widehat{U}^\eps (t, \xi)  \left( 1-   \chi   \left(\frac {\eps|\xi|} {\kappa} \right)  - \eps |\xi| \chi   \left(\frac {\eps|\xi|} {\kappa} \right)  \sum_{j=1}^4  \tilde{P}_j (\eps \xi )  \right) \widehat{f} (\xi, v).
\]
By Proposition \ref{prop:estimate_Ueps} we deduce
\begin{align*}
\Vert \widehat{U}^{\eps \#} (\cdot) \widehat{f}\Vert_{L^1_\xi L^\infty_t L^2_v} \lesssim&  \Vert \Big( 1-   \chi   \left(\frac {\eps|\xi|} {\kappa} \right)  - \eps |\xi| \chi   \left(\frac {\eps|\xi|} {\kappa} \right)  \sum_{j=1}^4  \tilde{P}_j (\eps \xi )  \Big) \widehat{f} (\xi) \Vert_{L^1_\xi L^2_v} 
\\
\lesssim&  \min \{\Vert \widehat{f}\Vert_{L^1_\xi L^2_v}, \eps \Vert |\xi| \widehat{f}\Vert_{L^1_\xi L^2_v} \},
\end{align*}
thus the proof is finished by gathering together the two previous estimates. 
\end{proof}

\subsection{Estimates on $\widehat \Psi^\eps$}

The decomposition of the semigroup $U^\eps(t) $ in \eqref{eq:decomposition_Ujeps} also gives us a decomposition of the operator $\Psi^\eps(t)$ defined in \eqref{definition Psi epsilon}. 

\begin{lem}\label{decomposition Psi epsilon}
The following decomposition holds
\[
\Psi^\eps = \sum_{j=1}^4 \Psi_j^\eps +  \Psi^{\eps \#},
\]
with
\[
\widehat{\Psi}^{\eps \#} [f_1, f_2] (t,\xi) : =    \frac 1 \eps \int_0^t \widehat{U}^{\eps \#} (t-\tau,\xi )\widehat{\Gamma}_{\mathrm{sym}}(f_1(\tau), f_2(\tau))(\xi) \, \d \tau,
\]
and, for all $1 \le j \le 4$,
\[
\Psi_j^\eps = \Psi_{j0}^\eps + \Psi_{j0}^{\eps \#} + \Psi_{j1}^\eps+ \Psi_{j2}^\eps,
\]
where 
\begin{align*}
\widehat \Psi_{j0}^\eps [f_1, f_2](t,\xi) 
&= \int_0^t e^{i \alpha_j |\xi| \frac {t-\tau} \eps -\beta_j (t-\tau) |\xi|^2} |\xi| P_j^1(\tfrac {\xi} {|\xi|})  \widehat{\Gamma}_{\mathrm{sym}}(f_1(\tau), f_2(\tau))(\xi) \, \d \tau,
\\
\widehat \Psi_{j0}^{\eps \#} [f_1, f_2](t,\xi) 
&= \left(\chi   \left(\frac {\eps|\xi|} {\kappa} \right)  -1 \right) \int_0^t e^{i \alpha_j |\xi| \frac {t-\tau} \eps -\beta_j (t-\tau) |\xi|^2} |\xi| P_j^1(\tfrac {\xi} {|\xi|})  \widehat{\Gamma}_{\mathrm{sym}}(f_1(\tau), f_2(\tau))(\xi) \, \d \tau,
\\
\widehat \Psi_{j1}^\eps [f_1, f_2](t,\xi) 
&=\chi   \left(\frac {\eps|\xi|} {\kappa} \right)   \int_0^t e^{i \alpha_j |\xi| \frac {t-\tau} \eps -\beta_j (t-\tau) |\xi|^2}    \left(e^{(t-\tau) \frac {\gamma_j(\eps |\xi|)  }     { \eps^2}  }   -1 \right) |\xi| P_j^1(\tfrac {\xi} {|\xi|})  \widehat{\Gamma}_{\mathrm{sym}}(f_1(\tau), f_2(\tau))(\xi) \, \d \tau,
\\
\widehat \Psi_{j2}^\eps [f_1, f_2](t,\xi)  
&= \chi   \left(\frac {\eps|\xi|} {\kappa} \right)  \int_0^t e^{i \alpha_j |\xi| \frac {t-\tau} \eps -\beta_j (t-\tau) |\xi|^2}  e^{(t-\tau) \frac {\gamma_j(\eps |\xi|)  }     { \eps^2}  }     \eps|\xi|^2 P_j^2(\eps \xi)  \widehat{\Gamma}_{\mathrm{sym}}(f_1(\tau), f_2(\tau))(\xi) \, \d \tau.
\end{align*}
\end{lem}

Similarly as above, we observe again that that $\widehat \Psi^\eps_{30}$ and $\widehat \Psi^\eps_{40}$ are independent of $\eps$, so that we define
\begin{equation}\label{eq:def:hatPsixi}
\widehat \Psi[f,g](t,\xi) := \widehat \Psi^\eps_{30}[f,g](t,\xi) + \widehat \Psi^\eps_{40}[f,g](t,\xi)
\end{equation}
which is then independent of $\eps$. We finally define
\begin{equation}\label{eq:def:Psi(t)}
\Psi [f,g](t) = \FF_x^{-1} \widehat \Psi [f,g](t) \FF_x.
\end{equation}

We are now able to prove the following result on the convergence of $\Psi^\eps $ towards $\Psi$.
\begin{lem}\label{convergence Psi epsilon Psi}
Let $(\rho_0, u_0 , \theta_0)$ satisfy the hypotheses of Theorem~\ref{theo:NSF} and consider the associated global unique solution $(\rho,u,\theta)$ to \eqref{Navier-Stokes-Fourier system}. Let also $g_0 = g_0(x,v) \in \Ker L$ be defined by \eqref{eq:g_0} and $g=g(t,x,v) \in \Ker L$ by \eqref{eq:g(t)}. Then we have: 
\begin{enumerate}
\item Torus case $\Omega_x = \T^3$: There holds
$$
\Vert  \Psi^\eps [g, g] -\Psi[g, g] \Vert_{L^1_\xi L^\infty_t L^2_v}
\lesssim  \eps  \left( \| \widehat  g_0 \|_{L^1_\xi L^2_v}^2
+ \| \widehat  g_0 \|_{L^1_\xi L^2_v}^3 \right) .
$$

\item Whole space case $\Omega_x = \R^3$: For any $p \in (3/2,\infty]$ there holds
$$
\Vert  \Psi^\eps [g, g] -\Psi[g, g] \Vert_{L^1_\xi L^\infty_t L^2_v}
\lesssim  \eps  \left( \| \widehat  g_0 \|_{L^1_\xi L^2_v}^2
+ \| \widehat  g_0 \|_{L^1_\xi L^2_v}^3
+ \| \widehat  g_0 \|_{L^p_\xi L^2_v}^2
+ \| \widehat  g_0 \|_{L^p_\xi L^2_v}^3 \right).
$$

\end{enumerate}

\end{lem}

\begin{proof}
We adapt the proof of \cite[Lemma 4.1]{GT} for the cutoff Boltzmann equation with hard potentials.
Thanks to the decomposition of $\Psi^\eps$ in Lemma~\ref{decomposition Psi epsilon} we write, for all $t \ge 0$ and $\xi \in \Omega'_\xi  $,
$$
\begin{aligned}
\widehat \Psi^\eps [g, g] (t,\xi) - \widehat \Psi[g,g](t,\xi) 
&= \sum_{j=1}^4 \left\{ \widehat \Psi^{\eps \#}_{j0}[g, g] (t,\xi) + \widehat \Psi^\eps_{j1}[g, g] (t,\xi)  + \widehat \Psi^\eps_{j2}[g, g] (t,\xi) \right\} \\
&\quad
+ \sum_{j=1}^2 \widehat \Psi^{\eps}_{j0}[g, g] (t,\xi)
+ \widehat \Psi^{\eps \#}[g, g] (t,\xi).
\end{aligned}
$$ 
We remark that for the zero frequency $\xi=0$ we have
$$
\widehat \Psi^\eps [g, g] (t,0) = \widehat \Psi^{\eps \#}[g, g] (t,0).
$$
We split the proof into several steps and estimate each term separately below.

\medskip\noindent
\textit{Step 1.}
By Lemma \ref{decomposition Psi epsilon} and \eqref{small epsilon xi estimate 2}, for the term $ \widehat{\Psi}_{j0}^{\eps \#} [g, g]$ with $j=1,2,3,4$, for all $t \ge 0$ and all $\xi \in \Omega'_\xi \setminus \{ 0 \}$ we have
\begin{align*}
\left \Vert \widehat{\Psi}_{j0}^{\eps \#} [g, g](t, \xi) \right \Vert_{L^2_v}  &\lesssim \left| \chi   \left(\frac {\eps|\xi|} {\kappa} \right)  -1 \right| \int_0^t e^{-\beta_j (t-\tau) |\xi|^2} |\xi|   \left \Vert P_j^1(\tfrac {\xi} {|\xi|})  \widehat{\Gamma}(g(\tau), g(\tau)) (\xi)  \right \Vert_{L^2_v}  \d \tau
\\
&\lesssim  \eps \int_0^t e^{-\beta_j (t-\tau) |\xi|^2} |\xi|^2    \Vert  \widehat{\Gamma}(g(\tau), g(\tau)) (\xi)   \Vert_{L^2_v}  \, \d \tau
\\
&\lesssim \eps \Vert  \widehat{\Gamma}(g, g) (\xi)   \Vert_{L^\infty_t L^2_v}.
\end{align*}
Similarly for the term $ \widehat{\Psi}_{j1}^{\eps} [g, g]$, by Lemma \ref{decomposition Psi epsilon} and \eqref{small epsilon xi estimate 1} we have, for all $j=1,2,3,4$,
\begin{align*}
\left \Vert \widehat{\Psi}_{j1}^{\eps} [g, g](t, \xi) \right \Vert_{L^2_v}  &\lesssim \chi   \left(\frac {\eps|\xi|} {\kappa} \right)   \int_0^t e^{ -\beta_j (t-\tau) |\xi|^2}    \left | e^{(t-\tau) \frac {\gamma_j(\eps |\xi|)  }     { \eps^2}  }   -1 \right| |\xi|  \left \Vert P_j^1(\tfrac {\xi} {|\xi|})  \widehat{\Gamma}(g(\tau), g(\tau))  \right \Vert_{L^2_v}  \d \tau
\\
&\lesssim  \eps \int_0^t e^{-\frac {\beta_j} 4 (t-\tau) |\xi|^2} |\xi|^2   \Vert  \widehat{\Gamma}(g(\tau), g(\tau)) (\xi)   \Vert_{L^2_v}  \, \d \tau
\\
&\lesssim  \eps \Vert  \widehat{\Gamma}(g, g) (\xi)   \Vert_{L^\infty_t L^2_v}. 
\end{align*}
Similarly for the term $ \widehat{\Psi}_{j2}^{\eps} [g, g]$, by Lemma \ref{decomposition Psi epsilon} we have, for all $j=1,2,3,4$,
\begin{align*}
\left \Vert \widehat{\Psi}_{j2}^{\eps} [g, g](t, \xi) \right \Vert_{L^2_v} &\lesssim \chi   \left(\frac {\eps|\xi|} {\kappa} \right)  \int_0^t e^{-\beta_j (t-\tau) |\xi|^2}  \left| e^{(t-\tau) \frac {\gamma_j(\eps |\xi|)  }     { \eps^2}  }   \right|  \eps|\xi|^2 \left \Vert P_j^2(\eps \xi)  \widehat{\Gamma}(g(\tau), g(\tau)) \right \Vert_{L^2_v}  \d \tau
\\
&\lesssim  \eps \int_0^t e^{-\frac {\beta_j} 4 (t-\tau) |\xi|^2} |\xi|^2    \Vert  \widehat{\Gamma}(g(\tau), g(\tau)) (\xi)   \Vert_{L^2_v} \, \d \tau
\\
&\lesssim \eps \Vert  \widehat{\Gamma}(g, g) (\xi)   \Vert_{L^\infty_t L^2_v}.
\end{align*}
Taking the $L^1_{\xi} L^\infty_t$ norm on both side we finally obtain, for all $j=1,2,3,4$,
\begin{align*}
\Vert \widehat{\Psi}_{j0}^{\eps \#} [g, g]  \Vert_{L^1_\xi L^\infty_t L^2_v}  +\Vert \widehat{\Psi}_{j1}^{\eps} [g, g]  \Vert_{L^1_\xi L^\infty_t L^2_v}+ \Vert \widehat{\Psi}_{j2}^{\eps} [g, g]  \Vert_{L^1_\xi L^\infty_t L^2_v} 
\lesssim  \eps \Vert  \widehat{\Gamma}(g, g)   \Vert_{L^1_\xi L^\infty_t L^2_v}.
\end{align*}
Thanks to \cite{Strain} and the fact that $\| \la v \ra^\ell \P \phi \|_{H^{m}_v } \lesssim \| \P \phi \|_{L^2_v}$ for all $m,\ell \ge 0$, we have
\begin{equation}\label{eq:Gamma_L2v}
\| \Gamma (\P g_1 , \P g_2) \|_{L^2_v } \lesssim \| \P g_1 \|_{L^2_v} \| \P g_2 \|_{L^2_v},
\end{equation}
therefore arguing as in Lemma~\ref{lem:nonlinear} it follows, for any $p \in [1,\infty]$ and $\ell \ge 0$,
\begin{equation}\label{eq:hatGamma_L1xiL2v}
\|  \widehat{\Gamma}(g, g)  \|_{L^p_\xi L^\infty_t L^2_v (\la v \ra^\ell) }
\lesssim \| g \|_{L^1_\xi L^\infty_t L^2_v} \| g \|_{L^p_\xi L^\infty_t L^2_v}.
\end{equation}
We therefore obtain, for all $j=1,2,3,4$,
\begin{equation}\label{eq:Psij0-Psij1-Psij2}
\begin{aligned}
\Vert \widehat{\Psi}_{j0}^{\eps \#} [g, g]  \Vert_{L^1_\xi L^\infty_t L^2_v}  +\Vert \widehat{\Psi}_{j1}^{\eps} [g, g]  \Vert_{L^1_\xi L^\infty_t L^2_v}+ \Vert \widehat{\Psi}_{j2}^{\eps} [g, g]  \Vert_{L^1_\xi L^\infty_t L^2_v} 
\lesssim  \eps \| g \|_{L^1_\xi L^\infty_t L^2_v}^2.
\end{aligned}
\end{equation}

\medskip\noindent
\textit{Step 2.}
We now focus on the term $\widehat{\Psi}_{j0}^\eps [g, g] $ with $j=1, 2$, and recall that $\alpha_j > 0$ for $j=1, 2$. We denote 
\[
\widehat{H_j}( t, \tau , \xi) =   e^{ -\beta_j (t-\tau) |\xi|^2} |\xi| P_j^1(\tfrac {\xi} {|\xi|})  \widehat{\Gamma}(g(\tau), g(\tau)) (\xi),
\]
and thus, using integration by parts, for all $t \ge 0$ and all $\xi \in \Omega'_\xi \setminus \{ 0 \}$ we have
\begin{equation}\label{eq:Psi_j0}
\begin{aligned}
\widehat{\Psi}_{j0}^\eps [g, g](t,\xi)  
&= \int_0^t e^{i \alpha_j |\xi| \frac {t-\tau} \eps -\beta_j (t-\tau) |\xi|^2} |\xi| P_j^1(\frac {\xi} {|\xi|})  \widehat{\Gamma}(g(\tau), g(\tau)) (\xi) \, \d \tau
\\
&= \frac {\eps} {i \alpha_j |\xi|} \left(\int_0^t e^{i \alpha_j |\xi| \frac {t-\tau} \eps} \partial_\tau \widehat{H_j}( t, \tau , \xi)  \, \d \tau - \widehat{H}_j(t, t, \xi)  +e^{i \alpha_j |\xi| \frac {t} \eps} \widehat{H_j}( t, 0 , \xi)  \right)
\\
&= \frac {\eps} {i \alpha_j } \left(\int_0^t e^{i \alpha_j |\xi| \frac {t-\tau} \eps}  \beta_j |\xi|^2  e^{ -\beta_j (t-\tau) |\xi|^2}  P_j^1(\frac {\xi} {|\xi|})  \widehat{\Gamma}(g(\tau), g(\tau)) (\xi)    \, \d \tau \right)
\\
&\quad + \frac {\eps} {i \alpha_j} \left(\int_0^t  e^{i \alpha_j |\xi| \frac {t-\tau} \eps} e^{ -\beta_j (t-\tau) |\xi|^2} P_j^1(\frac {\xi} {|\xi|})  \partial_\tau \widehat{\Gamma}(g(\tau), g(\tau)) (\xi) \, \d \tau   \right)
\\
&\quad - \frac {\eps} {i \alpha_j } P_j^1(\frac {\xi} {|\xi|})  \widehat{\Gamma}(g(t), g(t)) (\xi)  
\\
&\quad + \frac {\eps} {i \alpha_j } e^{i \alpha_j |\xi| \frac {t} \eps} e^{ -\beta_j t |\xi|^2} P_j^1(\frac {\xi} {|\xi|})  \widehat{\Gamma}(g(0), g(0)) (\xi) \\
&=: I_1(t,\xi) + I_2(t,\xi) + I_3(t,\xi) + I_4(t,\xi).
\end{aligned}
\end{equation}
For the first term in \eqref{eq:Psi_j0} we have for all $t \ge 0$ and $\xi \in \Omega'_\xi$, using Lemma~\ref{decomposition Psi epsilon},
$$
\begin{aligned}
\| I_1(t,\xi) \|_{L^2_v} 
&\lesssim \eps \int_0^t  \beta_j |\xi|^2  e^{ -\beta_j (t-\tau) |\xi|^2} \| \widehat{\Gamma}(g(\tau), g(\tau)) (\xi) \|_{L^2_v} \, \d \tau \\
&\lesssim \eps \| \widehat{\Gamma}(g, g) (\xi) \|_{L^\infty_t L^2_v}.
\end{aligned}
$$
Similarly, for the third term in \eqref{eq:Psi_j0} there holds
$$
\begin{aligned}
\| I_3(t,\xi) \|_{L^2_v} 
&\lesssim \eps  \|  \widehat{\Gamma}(g(t), g(t)) (\xi) \|_{L^2_v}  \\
&\lesssim \eps \|  \widehat{\Gamma}(g, g) (\xi) \|_{L^\infty_t L^2_v}.
\end{aligned}
$$
and for the fourth one
$$
\begin{aligned}
\| I_4(t,\xi) \|_{L^2_v} 
&\lesssim \eps  e^{ -\beta_j t |\xi|^2}\|  \widehat{\Gamma}(g(0), g(0)) (\xi) \|_{L^2_v}  \\
&\lesssim \eps \|  \widehat{\Gamma}(g, g) (\xi) \|_{L^\infty_t L^2_v}.
\end{aligned}
$$
This yields
\begin{equation}\label{eq:I1-I3-I4}
\| I_1 \|_{L^1_\xi L^\infty_t L^2_v} 
+\| I_3 \|_{L^1_\xi L^\infty_t L^2_v} 
+\| I_4 \|_{L^1_\xi L^\infty_t L^2_v} 
\lesssim \eps \|  \widehat{\Gamma}(g, g)  \|_{L^1_\xi L^\infty_t L^2_v}
\lesssim \eps \|  \widehat g  \|_{L^1_\xi L^\infty_t L^2_v}^2,
\end{equation}
where we have used \eqref{eq:hatGamma_L1xiL2v} in last inequality.

For the second term in \eqref{eq:Psi_j0} we first write,  for all $t \ge 0$ and $\xi \in \Omega'_\xi$,
$$
\begin{aligned}
\| I_2(t,\xi) \|_{L^2_v} 
&\lesssim \eps \int_0^t  e^{ -\beta_j (t-\tau) |\xi|^2} \| \partial_\tau \widehat{\Gamma}(g(\tau), g(\tau)) (\xi) \|_{L^2_v} \, \d \tau .
\end{aligned}
$$
Since $\partial_\tau \widehat{\Gamma}(g, g) =  \widehat{\Gamma}(\partial_\tau g, g) + \widehat{\Gamma}(g, \partial_\tau g)$, from \eqref{eq:Gamma_L2v} we get
$$
\| \partial_\tau \widehat{\Gamma}(g(\tau), g(\tau)) (\xi) \|_{L^2_v} 
\lesssim \int_{\Omega'_\eta} \| \widehat g(\tau, \xi - \eta) \|_{L^2_v} \| \partial_\tau g (\tau, \eta) \|_{L^2_v} \, \d \eta.
$$
As $g$ is defined through $(u, \theta, \rho)$ which satisfies the Navier-Stokes-Fourier system \eqref{Navier-Stokes-Fourier system}, we have for all $\tau \ge 0$ and all $\eta \in \Omega'_\eta$
\[
\Vert \partial_\tau \widehat{g} (\tau,\eta) \Vert_{L^2_v} \lesssim      |\eta|^2 \| \widehat{g}(\tau,\eta) \|_{L^2_v}+ |\eta| \int_{\Omega'_\zeta}   \| \widehat{g}(\tau,\eta-\zeta) \|_{L^2_v}  \|  \widehat{g}(\tau,\zeta) \|_{L^2_v} \, \d \zeta.
\]
This implies
$$
\begin{aligned}
\| I_2(t,\xi) \|_{L^2_v} 
&\lesssim \eps \int_0^t  e^{ -\beta_j (t-\tau) |\xi|^2} \int_{\Omega'_\eta} \| \widehat g(\tau, \xi - \eta) \|_{L^2_v} |\eta|^2 \| \widehat{g}(\tau,\eta) \|_{L^2_v} \, \d \eta  \, \d \tau \\
&\quad
+ \eps \int_0^t  e^{ -\beta_j (t-\tau) |\xi|^2}\int_{\Omega'_\eta} \| \widehat g(\tau, \xi - \eta) \|_{L^2_v} |\eta| \int_{\Omega'_\zeta}   \| \widehat{g}(\tau,\eta-\zeta) \|_{L^2_v}  \|  \widehat{g}(\tau,\zeta) \|_{L^2_v} \, \d \zeta \, \d \eta \, \d \tau \\
&
=: R_1(t,\xi) + R_2(t,\xi).
\end{aligned}
$$
For the term $R_1$ we split the integral on $\eta$ into two parts: the region $2|\xi|  >|\eta|$ in which we have $|\eta|^2 \le 4|\xi|^2$; and the region  $2|\xi|  \le |\eta|$ where we have $|\eta-\xi|  \sim |\eta| $, which yields
$$
\begin{aligned}
R_1(t,\xi)
&\lesssim \eps \int_0^t  e^{ -\beta_j (t-\tau) |\xi|^2} \int_{\Omega'_\eta} \mathbf{1}_{|\eta| < 2 |\xi|} \| \widehat g(\tau, \xi - \eta) \|_{L^2_v} |\eta|^2 \| \widehat{g}(\tau,\eta) \|_{L^2_v} \, \d \eta  \, \d \tau \\
&\quad
+ \eps \int_0^t  e^{ -\beta_j (t-\tau) |\xi|^2} \int_{\Omega'_\eta} \mathbf{1}_{|\eta| \ge 2 |\xi|} \| \widehat g(\tau, \xi - \eta) \|_{L^2_v} |\eta|^2 \| \widehat{g}(\tau,\eta) \|_{L^2_v} \, \d \eta  \, \d \tau \\
&\lesssim \eps \int_0^t  |\xi|^2 e^{ -\beta_j (t-\tau) |\xi|^2} \int_{\Omega'_\eta}  \| \widehat g(\tau, \xi - \eta) \|_{L^2_v}  \| \widehat{g}(\tau,\eta) \|_{L^2_v} \, \d \eta  \, \d \tau \\
&\quad
+ \eps \int_0^t  e^{ -\beta_j (t-\tau) |\xi|^2} \int_{\Omega'_\eta} |\xi-\eta| \| \widehat g(\tau, \xi - \eta) \|_{L^2_v} |\eta| \| \widehat{g}(\tau,\eta) \|_{L^2_v} \, \d \eta  \, \d \tau .
\end{aligned}
$$
Thanks to H\"older's inequality in the time variable, it follows
$$
\begin{aligned}
\| R_1(\xi) \|_{L^\infty_t}
&\lesssim \eps  \int_{\Omega'_\eta}  \| \widehat g(\xi - \eta) \|_{L^\infty_t L^2_v}  \| \widehat{g}(\eta) \|_{L^\infty_t L^2_v} \, \d \eta   \\
&\quad
+ \eps  \int_{\Omega'_\eta} |\xi-\eta| \| \widehat g(\xi - \eta) \|_{L^2_t L^2_v} |\eta| \| \widehat{g}(\eta) \|_{L^2_t L^2_v} \, \d \eta   ,
\end{aligned}
$$
therefore taking the $L^1_\xi$ norm and using Young's convolution inequality we obtain
\begin{equation}\label{eq:R1}
\begin{aligned}
\| R_1 \|_{L^1_\xi L^\infty_t}
&\lesssim \eps  \| \widehat g \|_{L^1_\xi L^\infty_t L^2_v}^2  + \eps \| |\xi|  \widehat g \|_{L^1_\xi L^2_t L^2_v}^2.
\end{aligned}
\end{equation}

For the term $R_2$ we write
$$
\begin{aligned}
\| R_2(\xi) \|_{L^\infty_t}
&\lesssim  \eps \sup_{t \ge 0} \int_0^t  |\xi| e^{ -\beta_j (t-\tau) |\xi|^2} |\xi|^{-1}\int_{\Omega'_\eta} \| \widehat g(\tau, \xi - \eta) \|_{L^2_v} |\eta| \int_{\Omega'_\zeta}   \| \widehat{g}(\tau,\eta-\zeta) \|_{L^2_v}  \|  \widehat{g}(\tau,\zeta) \|_{L^2_v} \, \d \zeta \, \d \eta \, \d \tau \\
&\lesssim  \eps \sup_{t \ge 0}\left(\int_0^t  |\xi|^2 e^{ -\beta_j (t-\tau) |\xi|^2} \, \d \tau \right)^{1/2}  |\xi|^{-1} \left( \int_0^\infty G(\tau,\xi)^2 \, \d \tau \right)^{1/2},
\end{aligned}
$$
where we denote
$$
G(\tau,\xi) = \int_{\Omega'_\eta} \| \widehat g(\tau, \xi - \eta) \|_{L^2_v} H(\tau,\eta) \, \d \eta ,
\quad
H(\tau,\eta) = |\eta| \int_{\Omega'_\zeta}   \| \widehat{g}(\tau,\eta-\zeta) \|_{L^2_v}  \|  \widehat{g}(\tau,\zeta) \|_{L^2_v} \, \d \zeta  .
$$
By Minkowski and H\"older inequalities
$$
\begin{aligned}
\|   G(\xi) \|_{L^2_t}
&\lesssim \int_{\Omega'_\eta} \left(\int_0^\infty \| \widehat g(\tau, \xi - \eta) \|_{L^2_v}^2  |H(\tau,\eta)|^2 \, \d \tau \right)^{1/2}   \d \eta \\
&\lesssim \int_{\Omega'_\eta} \| \widehat g( \xi - \eta) \|_{L^\infty_t L^2_v} \|H(\eta) \|_{L^2_t}  \, \d \eta.
\end{aligned}
$$
Moreover
$$
H(\tau,\eta) \lesssim  \int_{\Omega'_\zeta}  |\eta- \zeta| \| \widehat{g}(\tau,\eta-\zeta) \|_{L^2_v}  \|  \widehat{g}(\tau,\zeta) \|_{L^2_v} \, \d \zeta  
+ \int_{\Omega'_\zeta} \| \widehat{g}(\tau,\eta-\zeta) \|_{L^2_v}  |\zeta|\|  \widehat{g}(\tau,\zeta) \|_{L^2_v} \, \d \zeta  .
$$
Thus again by Minkowski and H\"older inequalities,
$$
\begin{aligned}
\|  H(\eta) \|_{L^2_t}
&\lesssim \int_{\Omega'_\zeta} \left(\int_0^\infty \| |\eta-\zeta| \widehat g(\tau,\eta-\zeta) \|_{L^2_v}^2  \| \widehat g(\tau,\zeta) \|_{L^2_v}^2 \, \d \tau \right)^{1/2}   \d \zeta \\
&\quad
+ \int_{\Omega'_\zeta} \left(\int_0^\infty \| \widehat g(\tau,\eta-\zeta) \|_{L^2_v}^2  \| |\zeta| \widehat g(\tau,\zeta) \|_{L^2_v}^2 \, \d \tau \right)^{1/2}   \d \zeta \\
&\lesssim
 \int_{\Omega'_\zeta} \| \widehat g(\eta-\zeta) \|_{L^\infty_t L^2_v}  \| |\zeta| \widehat g(s,\zeta) \|_{L^2_t L^2_v} \,   \d \zeta .
\end{aligned}
$$
Hence we get
$$
\begin{aligned}
\| R_2(\xi) \|_{L^\infty_t}
&\lesssim  \eps  |\xi|^{-1} \int_{\Omega'_\eta} \int_{\Omega'_\zeta} \| \widehat g( \xi - \eta) \|_{L^\infty_t L^2_v}  \| \widehat g(\eta-\zeta) \|_{L^\infty_t L^2_v}  \| |\zeta| \widehat g(\zeta) \|_{L^2_t L^2_v} \,   \d \zeta \, \d \eta .
\end{aligned}
$$
Taking the $L^1_\xi$ norm and distinguishing between high and low frequencies yields
\begin{equation}\label{eq:R2}
\| \mathbf{1}_{|\xi| \ge 1} R_2 \|_{L^1_\xi L^\infty_t}
\lesssim \eps \| \widehat g \|_{L^1_\xi L^\infty_t L^2_v}^2 \| |\xi| \widehat g \|_{L^1_\xi L^2_t L^2_v},
\end{equation}
and, in the whole space case $\Omega_x = \R^3$ and $\Omega'_\xi=\R^3$,
\begin{equation}\label{eq:R2bis}
\begin{aligned}
\| \mathbf{1}_{|\xi| < 1} R_2 \|_{L^1_\xi L^\infty_t}
&\lesssim \eps  \| \mathbf{1}_{|\xi| < 1} |\xi|^{-1} \|_{L^{p'}_\xi} \left\| \int_{\Omega'_\eta} \int_{\Omega'_\zeta} \| \widehat g( \xi - \eta) \|_{L^\infty_t L^2_v}  \| \widehat g(\eta-\zeta) \|_{L^\infty_t L^2_v}  \| |\zeta| \widehat g(\zeta) \|_{L^2_t L^2_v} \,   \d \zeta \, \d \eta \right\|_{L^p_\xi} \\
&\lesssim \eps   \| \widehat g \|_{L^p_\xi L^\infty_t L^2_v}  \| \widehat g \|_{L^1_\xi L^\infty_t L^2_v}  \| |\xi| \widehat g \|_{L^1_\xi L^2_t L^2_v},
\end{aligned}
\end{equation}
where we have used that $\mathbf{1}_{|\xi| < 1} |\xi|^{-1} \in L^{p'}_\xi$ since $p>3/2$.

\medskip\noindent
\textit{Step 3.}
It only remains to compute the term $\widehat{\Psi}^{\eps \#}$, for which we first write, for all $t \ge 0$ and $\xi \in \Omega'_\xi$, 
$$
\| \widehat{\Psi}^{\eps \#}[g,g] (t,\xi) \|_{L^2_v}
\lesssim \frac{1}{\eps} \int_0^t \| \widehat{U}^{\eps \#} (t-\tau, \xi )\widehat{\Gamma}(g(\tau), g(\tau))(\xi) \|_{L^2_v} \, \d \tau.
$$
In the hard potentials case $\gamma+2s \ge 0$, thanks to \eqref{U hard estimate} we have, for all $t \ge 0$ and $\xi \in \Omega'_\xi$,
$$
\begin{aligned}
\| \widehat{\Psi}^{\eps \#}[g,g] (t,\xi) \|_{L^2_v}
&\lesssim \frac{1}{\eps} \int_0^t e^{-\lambda_1 \frac{(t-\tau)}{\eps^2}} \| \widehat{\Gamma}(g(\tau), g(\tau))(\xi) \|_{L^2_v} \, \d \tau \\
&\lesssim \frac{1}{\eps} \| \widehat{\Gamma}(g, g)(\xi) \|_{L^\infty_t L^2_v} \int_0^t e^{-\lambda_1 \frac{(t-\tau)}{\eps^2}}  \, \d \tau \\
&\lesssim \eps \| \widehat{\Gamma}(g, g)(\xi) \|_{L^\infty_t L^2_v}.
\end{aligned}
$$
For the soft potentials case $\gamma + 2s < 0$, observing that $\mathbf{P} \widehat \Gamma (g,g) = 0$ we fix $\ell>0$ such that $ \frac{\ell}{|\gamma+2s|} >1$ then we use \eqref{U soft estimate} to obtain, for all $t \ge 0$ and $\xi \in \Omega'_\xi$,
$$
\begin{aligned}
\| \widehat{\Psi}^{\eps \#}[g,g] (t,\xi) \|_{L^2_v}
&\lesssim \frac{1}{\eps} \int_0^t \left( 1+ \frac{(t-\tau)}{\eps^2}\right)^{- \frac{\ell}{|\gamma+2s|}} \| \widehat{\Gamma}(g(\tau), g(\tau))(\xi) \|_{L^2_v (\la v \ra^\ell)} \, \d \tau \\
&\lesssim \frac{1}{\eps} \| \widehat{\Gamma}(g, g)(\xi) \|_{L^\infty_t L^2_v (\la v \ra^\ell)} \int_0^t \left( 1+ \frac{(t-\tau)}{\eps^2}\right)^{- \frac{\ell}{|\gamma+2s|}}   \, \d \tau \\
&\lesssim \eps \| \widehat{\Gamma}(g, g)(\xi) \|_{L^\infty_t L^2_v (\la v \ra^\ell)}.
\end{aligned}
$$

Taking the $L^1_\xi L^\infty_t$ norm in above estimates and using \eqref{eq:hatGamma_L1xiL2v} yields, for both hard potentials and soft potentials cases,
\begin{equation}\label{eq:Psieps_hard}
\| \widehat{\Psi}^{\eps \#}[g,g]  \|_{L^1_\xi L^\infty_t L^2_v}
\lesssim \eps \| \widehat g \|_{L^1_\xi L^\infty_t L^2_v}^2.
\end{equation}

\medskip\noindent
\textit{Step 4: Conclusion.} We conclude the proof by gathering estimates \eqref{eq:Psij0-Psij1-Psij2}, \eqref{eq:I1-I3-I4}, \eqref{eq:R1}, \eqref{eq:R2}, \eqref{eq:R2bis}, and \eqref{eq:Psieps_hard} together with the bounds for $g$ from Theorem~\ref{theo:NSF}.
\end{proof}

\subsection{Proof of Theorem~\ref{theo:hydro_limit}}

Let $f^\eps$, for any $\eps \in (0,1]$, be the unique global mild solution to~\eqref{eq:feps_intro} associated to the initial data $f^\eps_0$ constructed in Theorem~\ref{theo:boltzmann}.

Let $g = \P g$ be the kinetic distribution defined by \eqref{eq:g(t)} through the unique global mild solution $(\rho,u,\theta)$ to \eqref{Navier-Stokes-Fourier system} associated to the initial data $(\rho_0,u_0,\theta_0)$ constructed in Theorem~\ref{theo:NSF}, and denote also $g_0 = \P g_0$ the initial kinetic distribution defined by \eqref{eq:g_0} through the initial data $(\rho_0,u_0,\theta_0)$.

We now from \cite{BU,GT} for instance, that $g$ verifies the equation
\begin{equation}
g(t) = U(t) g_0 + \Psi[g,g](t),
\end{equation}
where we recall that $U(t)$ is defined in \eqref{eq:def:U(t)}, and $\Psi(t)$ in \eqref{eq:def:Psi(t)}. Taking the Fourier transform in $x \in \Omega_x$, we then have
\begin{equation}
\widehat g(t,\xi) = \widehat U(t,\xi) \widehat g_0(\xi) + \widehat \Psi[g,g] (t,\xi) .
\end{equation}
for all $\xi \in \Omega'_\xi$, and where we recall that $\widehat U$ is defined in \eqref{eq:def:hatUxi}, and $\widehat \Psi$ in \eqref{eq:def:hatPsixi}.

We first observe that the difference $f^\eps - g$ satisfies
\begin{equation}\label{eq:feps-g}
\begin{aligned}
\widehat f^\eps (\xi) - \widehat g(\xi) 
&= \widehat U^\eps(t,\xi) \widehat f^\eps_0(\xi) - \widehat U(t,\xi) \widehat g_0(\xi) + \widehat \Psi^\eps [f^\eps,f^\eps](t,\xi)
- \widehat \Psi [g,g](t,\xi) \\
&= \widehat U^\eps(t,\xi) \left\{ \widehat f^\eps_0(\xi) - \widehat g_0(\xi) \right\}
+ \left\{\widehat U^\eps(t,\xi)-\widehat U(t,\xi)  \right\} \widehat g_0(\xi)  \\
&\quad
+ \left\{ \widehat \Psi^\eps [g,g](t,\xi)
- \widehat \Psi [g,g](t,\xi) \right\}
+ \left\{  \widehat \Psi^\eps [f^\eps,f^\eps](t,\xi)
-\widehat \Psi^\eps [g,g](t,\xi)  \right\} \\
&=: T_1 + T_2 + T_3 + T_4,
\end{aligned}
\end{equation}
and we estimate each one of these terms separately.

For the first term, from Lemma~\ref{basic property U epsilon} we have
$$
\| \widehat U^\eps(\cdot) \{ \widehat f^\eps_0 - \widehat g_0 \} \|_{L^1_\xi L^\infty_t L^2_v}
\lesssim \|\widehat f^\eps_0 - \widehat g_0 \|_{L^1_\xi L^2_v}.
$$
Thanks to Lemma~\ref{convergence U epsilon U} and an interpolation argument, we obtain for the second term, for any $\delta \in [0,1]$,
$$
\|  \{\widehat U^\eps(\cdot)-\widehat U(\cdot)  \} \widehat g_0\|_{L^1_\xi L^\infty_t L^2_v} \lesssim  \eps^\delta \|\la \xi \ra^\delta \widehat g_0 \|_{L^1_\xi L^2_v}.
$$
For the third term we use Lemma~\ref{convergence Psi epsilon Psi}, which yields
$$
\|  \widehat \Psi^\eps [g,g] - \widehat \Psi [g,g] \|_{L^1_\xi L^\infty_t L^2_v} \lesssim \eps \left( \| \widehat g_0 \|_{L^1_\xi L^2_v}^2 + \| \widehat g_0 \|_{L^1_\xi L^2_v}^3 \right),
$$
in the case $\Omega_x=\T^3$, and
$$
\|  \widehat \Psi^\eps [g,g] - \widehat \Psi [g,g] \|_{L^1_\xi L^\infty_t L^2_v} \lesssim \eps \left( \| \widehat g_0 \|_{L^1_\xi L^2_v}^2 + \| \widehat g_0 \|_{L^1_\xi L^2_v}^3 + \| \widehat g_0 \|_{L^p_\xi L^2_v}^2 + \| \widehat g_0 \|_{L^p_\xi L^2_v}^3 \right),
$$
in the case $\Omega_x=\R^3$.

For the fourth term $T_4$, we first decompose $f^\eps = \P^\perp f^\eps + \P f^\eps$ and use that $g = \P g$ to write
$$
\begin{aligned}
T_4 &=\widehat \Psi^\eps [f^\eps,f^\eps](t,\xi) -\widehat \Psi^\eps [g,g](t,\xi) \\
&= \widehat \Psi^\eps [\P^\perp f^\eps, \P^\perp f^\eps](t,\xi) 
+ 2\widehat \Psi^\eps [\P f^\eps, \P^\perp f^\eps](t,\xi)  \\
&\quad
+\widehat \Psi^\eps [\P f^\eps, \P (f^\eps-g)](t,\xi) + \widehat \Psi^\eps [\P g, \P (f^\eps-g)](t,\xi) .
\end{aligned}
$$
Thanks to Proposition~\ref{prop:estimate_Ueps_regularization} and Lemma~\ref{lem:nonlinear} we have
$$
\begin{aligned}
\| \widehat \Psi^\eps [\P^\perp f^\eps, \P^\perp f^\eps] \|_{L^1_\xi L^\infty_t L^2_v}
& \lesssim  \| \widehat \Gamma [\P^\perp f^\eps, \P^\perp f^\eps] \|_{L^1_\xi L^2_t (H^{s,*}_v)'} \\
&\lesssim \| \P^\perp \widehat f^\eps \|_{L^1_\xi L^\infty_t L^2_v}
\| \P^\perp \widehat f^\eps \|_{L^1_\xi L^2_t H^{s,*}_v},
\end{aligned}
$$
moreover
$$
\begin{aligned}
\| \widehat \Psi^\eps [\P f^\eps, \P^\perp f^\eps] \|_{L^1_\xi L^\infty_t L^2_v}
& \lesssim  \| \widehat \Gamma [\P f^\eps, \P^\perp f^\eps] \|_{L^1_\xi L^2_t (H^{s,*}_v)'} + \| \widehat \Gamma [\P^\perp f^\eps, \P f^\eps] \|_{L^1_\xi L^2_t (H^{s,*}_v)'} \\
&\lesssim \| \P \widehat f^\eps \|_{L^1_\xi L^\infty_t L^2_v}
\| \P^\perp \widehat f^\eps \|_{L^1_\xi L^2_t H^{s,*}_v},
\end{aligned}
$$
where we have used that $\| \P \phi \|_{H^{s,*}_v} \lesssim \| \P \phi \|_{L^2_v}$ and $\| \la v \ra^{-(\gamma/2+s)_{-}}  \phi \|_{L^2_v} \lesssim \min\{ \| \phi \|_{L^2_v} , \| \phi \|_{H^{s,*}_v} \}$.
This implies
\begin{equation}\label{eq:I1_I2_I3}
\begin{aligned}
\| \widehat \Psi^\eps [\P^\perp f^\eps, \P^\perp f^\eps] \|_{L^1_\xi L^\infty_t L^2_v} 
&+ 2\| \widehat \Psi^\eps [\P f^\eps, \P^\perp f^\eps] \|_{L^1_\xi L^\infty_t L^2_v}  \\
&\lesssim \| \widehat f^\eps \|_{L^1_\xi L^\infty_t L^2_v}
\| \P^\perp \widehat f^\eps \|_{L^1_\xi L^2_t H^{s,*}_v}.
\end{aligned}
\end{equation}
Therefore, using the bounds of Theorem~\ref{theo:boltzmann}, we deduce from \eqref{eq:I1_I2_I3} that
$$
\begin{aligned}
&\| \widehat \Psi^\eps [\P^\perp f^\eps, \P^\perp f^\eps] \|_{L^1_\xi L^\infty_t L^2_v} 
+ 2\| \widehat \Psi^\eps [\P f^\eps, \P^\perp f^\eps] \|_{L^1_\xi L^\infty_t L^2_v} 
\lesssim \eps \| \widehat f^\eps_0 \|_{L^1_\xi L^2_v}^2,
\end{aligned}
$$
in the case $\Omega_x=\T^3$, and
$$
\begin{aligned}
&\| \widehat \Psi^\eps [\P^\perp f^\eps, \P^\perp f^\eps] \|_{L^1_\xi L^\infty_t L^2_v} + 2\| \widehat \Psi^\eps [\P f^\eps, \P^\perp f^\eps] \|_{L^1_\xi L^\infty_t L^2_v} 
\lesssim \eps \left( \| \widehat f^\eps_0 \|_{L^1_\xi L^2_v}^2 + \| \widehat f^\eps_0 \|_{L^p_\xi L^2_v}^2\right),
\end{aligned}
$$
in the case $\Omega_x=\R^3$.

Furthermore, from Proposition~\ref{prop:estimate_Ueps_regularization} and Lemma~\ref{lem:nonlinear}, and also using that $\| \P \phi \|_{H^{s,*}_v} \lesssim \| \P \phi \|_{L^2_v}$, we have
$$
\begin{aligned}
&\| \widehat \Psi^\eps [\P f^\eps , \P (f^\eps-g)] \|_{L^1_\xi L^\infty_t L^2_v}\\
&\qquad\lesssim \| \widehat \Gamma ( \P (f^\eps-g),\P f^\eps ) \|_{L^1_\xi L^2_t (H^{s,*}_v)'} + \| \widehat \Gamma ( \P f^\eps, \P (f^\eps-g) ) \|_{L^1_\xi L^2_t (H^{s,*}_v)'}\\
&\qquad\lesssim \| \P (\widehat f^\eps - \widehat g) \|_{L^1_\xi L^\infty_t L^2_v} \| \P \widehat f^\eps \|_{L^1_\xi L^2_t L^2_v},
\end{aligned}
$$
and similarly
$$
\begin{aligned}
&\|  \widehat \Psi^\eps [\P g, \P (f^\eps-g)] \|_{L^1_\xi L^\infty_t L^2_v}\\
&\qquad\lesssim \| \widehat \Gamma ( \P g, \P (f^\eps-g) ) \|_{L^1_\xi L^2_t (H^{s,*}_v)'} + \| \widehat \Gamma (  \P (f^\eps-g), \P g ) \|_{L^1_\xi L^2_t (H^{s,*}_v)'} \\
&\qquad\lesssim \| \P (\widehat f^\eps - \widehat g) \|_{L^1_\xi L^\infty_t L^2_v} \| \P \widehat g \|_{L^1_\xi L^2_t L^2_v}.
\end{aligned}
$$
In the case of the torus $\Omega_x = \T^3$, we can use the bounds of Theorem~\ref{theo:boltzmann}--(1) and Theorem~\ref{theo:NSF}--(1) to obtain
$$
\begin{aligned}
&\| \widehat \Psi^\eps [\P f^\eps , \P (f^\eps-g)] \|_{L^1_\xi L^\infty_t L^2_v}
+ \|  \widehat \Psi^\eps [\P g, \P (f^\eps-g)] \|_{L^1_\xi L^\infty_t L^2_v} \\
&\qquad 
\lesssim \left(  \| \widehat f^\eps_0 \|_{L^1_\xi L^2_v} + \| \widehat g_0 \|_{L^1_\xi L^2_v} \right) \| \widehat f^\eps - \widehat g \|_{L^1_\xi L^\infty_t L^2_v} \\
&\qquad 
\lesssim \eta_2 \| \widehat f^\eps - \widehat g \|_{L^1_\xi L^\infty_t L^2_v} .
\end{aligned}
$$
In the case of the whole space $\Omega_x = \R^3$, we first use \eqref{eq:fPf_Hs*v} to write
$$
\| \P f^\eps \|_{L^1_\xi L^2_t L^2_v} \lesssim 
\left\| \frac{|\xi|}{\la \xi \ra} \P f^\eps \right\|_{L^1_\xi L^2_t L^2_v}
+ \left\| \frac{|\xi|}{\la \xi \ra} \P f^\eps \right\|_{L^p_\xi L^2_t L^2_v},
$$
then we use the bounds of Theorem~\ref{theo:boltzmann}--(2) and Theorem~\ref{theo:NSF}--(2) to get
$$
\begin{aligned}
&\| \widehat \Psi^\eps [\P (f^\eps-g), \P f^\eps] \|_{L^1_\xi L^\infty_t L^2_v}
+ \|  \widehat \Psi^\eps [\P g, \P (f^\eps-g)] \|_{L^1_\xi L^\infty_t L^2_v} \\
&\qquad 
\lesssim \left(  \| \widehat f^\eps_0 \|_{L^1_\xi L^2_v}  + \| \widehat f^\eps_0 \|_{L^p_\xi L^2_v} + \| \widehat g_0 \|_{L^1_\xi L^2_v} + \| \widehat g_0 \|_{L^p_\xi L^2_v}  \right) \| \widehat f^\eps - \widehat g \|_{L^1_\xi L^\infty_t L^2_v} \\
&\qquad 
\lesssim \eta_2 \| \widehat f^\eps - \widehat g \|_{L^1_\xi L^\infty_t L^2_v} .
\end{aligned}
$$

Gathering previous estimates and using that $\eta_2>0$ is small enough, so that when taking the $L^1_\xi L^\infty_t L^2_v$ norm of~\eqref{eq:feps-g} the  fourth and fifth terms in the right-hand side of~\eqref{eq:feps-g} can be absorbed by the left-hand side, we deduce
\begin{equation}\label{final hydrodynamical result1}
\begin{aligned}
\| \widehat f^\eps - \widehat g \|_{L^1_\xi L^\infty_t L^2_v}
&\lesssim  \|\widehat f^\eps_0 - \widehat g_0 \|_{L^1_\xi L^2_v}
+ \eps^\delta \|\la \xi \ra^\delta \widehat g_0 \|_{L^1_\xi L^2_v} \\
&\quad
+ \eps \left(  \| \widehat g_0 \|_{L^1_\xi L^2_v}^2 + \| \widehat g_0 \|_{L^1_\xi L^2_v}^3 \right)
+\eps \| \widehat f^\eps_0 \|_{L^1_\xi L^2_v}^2
\end{aligned}
\end{equation}
in the case $\Omega_x=\T^3$, and
\begin{equation}\label{final hydrodynamical result2}
\begin{aligned}
\| \widehat f^\eps - \widehat g \|_{L^1_\xi L^\infty_t L^2_v}
&\lesssim  \|\widehat f^\eps_0 - \widehat g_0 \|_{L^1_\xi L^2_v}
+ \eps^\delta \|\la \xi \ra^\delta \widehat g_0 \|_{L^1_\xi L^2_v} \\
&\quad
+ \eps \left(  \| \widehat g_0 \|_{L^1_\xi L^2_v}^2 + \| \widehat g_0 \|_{L^1_\xi L^2_v}^3 +  \| \widehat g_0 \|_{L^p_\xi L^2_v}^2 + \| \widehat g_0 \|_{L^p_\xi L^2_v}^3 \right) \\
&\quad
+\eps \left( \| \widehat f^\eps_0 \|_{L^1_\xi L^2_v}^2 + \| \widehat f^\eps_0 \|_{L^p_\xi L^2_v}^2 \right)
\end{aligned}
\end{equation}
in the case $\Omega_x=\R^3$.
From these estimates, we first conclude that
$$
\lim_{\eps \to 0} \| \widehat f^\eps - \widehat g \|_{L^1_\xi L^\infty_t L^2_v} = 0,
$$
assuming moreover that $\la \xi \ra^\delta \widehat g_0 \in L^1_\xi L^2_v$ for some $\delta \in (0,1]$. We can finally prove Theorem~\ref{theo:hydro_limit}, where we only assume $\widehat g_0 \in L^1_\xi L^2_v$, by using the previous convergence and arguing by density as in \cite{CRT}.  This completes the proof of Theorem~\ref{theo:hydro_limit}.

\bigskip

\end{document}